\documentclass{amsart}

\usepackage{amsfonts,amssymb,amsmath,stmaryrd,mathtools,amsthm}
\usepackage[english]{babel}
\usepackage[super]{nth}
\usepackage{lmodern}
\usepackage[latin1]{inputenc}
\usepackage[T1]{fontenc}
\usepackage{enumerate}
\usepackage[all]{xy}
\usepackage{mathdots}

\usepackage{subcaption}
\usepackage{graphicx, transparent}
\usepackage{arydshln}
\usepackage{nicefrac}
\usepackage{multirow}
\usepackage[hyphens]{url}

\DeclareMathOperator{\N}{\mathbb{N}}
\DeclareMathOperator{\R}{\mathbb{R}}
\DeclareMathOperator{\Z}{\mathbb{Z}}

\DeclareMathOperator{\Q}{\mathbb{Q}}
\DeclareMathOperator{\C}{\mathbb{C}}
\DeclareMathOperator{\Br}{Br}
\DeclareMathOperator{\PBr}{PBr}
\DeclareMathOperator{\EBr}{\mathit{E}Br}
\DeclareMathOperator{\BBr}{\mathit{B}Br}

\DeclareMathOperator{\Conf}{Conf}
\DeclareMathOperator{\PConf}{PConf}
\DeclareMathOperator{\id}{id}
\DeclareMathOperator{\Hur}{Hur}
\DeclareMathOperator{\CHur}{CHur}

\DeclareMathOperator{\BL}{\mathit{BL}}
\DeclareMathOperator{\FP}{F\mathcal{P}_{\underline\delta}^{\underline\xi}}
\DeclareMathOperator{\cP}{\mathcal{P}_{\underline\delta}^{\underline\xi}}
\DeclareMathOperator{\FO}{F\mathcal{O}_{\underline\delta}^{\underline\xi}}
\DeclareMathOperator{\cO}{\mathcal{O}_{\underline\delta}^{\underline\xi}}
\DeclareMathOperator{\Star}{Star}

\DeclareMathOperator{\conn}{conn}
\DeclareMathOperator{\Link}{Link}
\DeclareMathOperator{\DEL}{\underline\delta}
\DeclareMathOperator{\XI}{\underline\xi}
\DeclareMathOperator{\Tor}{Tor}
\DeclareMathOperator{\plant}{\mathcal{O}^{[\mathit{n},\underline \xi]}}
\DeclareMathOperator{\rplant}{|\mathcal{O}^{[\mathit{n},\underline \xi]}|}
\DeclareMathOperator{\plantq}{\mathcal{O}_{\mathit{q}}^{[\mathit{n},\underline \xi]}}
\DeclareMathOperator{\plantqi}{\mathcal{O}_{\mathit{q}-1}^{[\mathit{n},\underline \xi]}}

\DeclareMathOperator{\coker}{coker}
\DeclareMathOperator{\Diff}{Diff}
\DeclareMathOperator{\Aut}{Aut}

\DeclareMathOperator{\BG}{\mathit{BG}}
\DeclareMathOperator{\EG}{\mathit{EG}}
\DeclareMathOperator{\ord}{ord}

\DeclareMathOperator{\GL}{GL}

\DeclareMathOperator{\Hom}{Hom}

\DeclareMathOperator{\cplant}{\mathbf{O}^{[\mathit{n},\underline \xi]}}
\DeclareMathOperator{\cplantq}{\mathbf{O}_{\mathit{q}}^{[\mathit{n},\underline \xi]}}
\DeclareMathOperator{\cplantqi}{\mathbf{O}_{\mathit{q}-1}^{[\mathit{n},\underline \xi]}}

\DeclareMathOperator{\Map}{Map}

\DeclareMathOperator{\bd}{\mathbf{d}}
\DeclareMathOperator{\sgn}{sgn}
\DeclareMathOperator{\cc}{\mathbf{c}}
\DeclareMathOperator{\dd}{{d}}
\DeclareMathOperator{\KK}{{\mathcal{K}}}
\DeclareMathOperator{\im}{im}

\theoremstyle{plain}
\numberwithin{equation}{section}

\newtheorem{theorem-main}{Theorem}
\newtheorem{proposition}[equation]{Proposition}
\newtheorem{corollary}[equation]{Corollary}
\newtheorem{lemma}[equation]{Lemma}
\newtheorem{theorem}[equation]{Theorem}
\newtheorem*{theorem*}{Theorem}

\theoremstyle{remark}
\newtheorem{remark}[equation]{Remark}
\newtheorem{example}[equation]{Example}

\theoremstyle{definition}
\newtheorem{definition}[equation]{Definition}

\usepackage{tikz}

\usepackage{hyperref}

\title{Plant Complexes and Homological Stability for Hurwitz Spaces}
\author{J.~Frederik~Tietz}
\address{J.~Frederik Tietz, Institut f\"ur Algebraische Geometrie, Leibniz Universit\"at Hannover, Welfengarten~1, 30167 Hannover, Germany}
\email{ftietz@math.uni-hannover.de}
\date{\today}

\begin{document}

\begin{abstract}
We study Hurwitz spaces with regard to homological stabilization. By a Hurwitz space, we mean a moduli space of branched, not necessarily connected coverings of a disk with fixed structure group and number of branch points. We choose a sequence of subspaces of Hurwitz spaces which is suitable for our investigations.

In the first part (Sections~\ref{plants} to~\ref{combinatorics}), we introduce and study plant complexes, a large new class of simplicial complexes, generalizing the arc complex on a surface with marked points. In the second part (Sections~\ref{hurwitz-spaces} to~\ref{application}), we generalize a result from \cite{0912.0325} by showing that homological stabilization of our sequence of Hurwitz spaces depends only on properties of their zeroth homology groups.
\end{abstract}

\maketitle

\addtocontents{toc}{\protect\setcounter{tocdepth}{1}}
\tableofcontents

\section{Introduction}

Understanding the topology of moduli spaces is a key aspect in order to grasp the behavior in families of the parametrized objects. Now, moduli spaces often come in sequences, such as the moduli spaces $\mathcal{M}_{g}$ of Riemann surfaces of genus~$g$. 
In some cases, such sequences satisfy \emph{homological stability}. Indeed, by \cite{MR786348}, the homology groups $H_{*}(\mathcal{M}_{g}; \Q)$ are independent of $g$ in a range of dimensions growing with $g$. The stable rational cohomology of $\mathcal{M}_{g}$ is the subject of \emph{Mumford's conjecture}, proved in~\cite{MR2335797}.

In recent years, the study of (co-)homological stability phenomena has been of great interest in algebraic and geometric topology as well as algebraic geometry. Classical results include stabilization for the group homology of the sequences of symmetric groups $\mathfrak{S}_{n}$ (\cite{MR0112134}), general linear groups $\GL_{n}$ (\cite{Maazen}, \cite{MR586429}), and Artin braid groups $\Br_{n}$ (\cite{MR0274462}). The theorem for braid groups builds a bridge to moduli spaces: $\Br_{n}$ is classified by the unordered configuration space $\Conf_{n}$, which parametrizes subsets of size $n$ of a disk $D$. Hence, the sequence $\{\Conf_{n}\}$ is homologically stable.

In fact, several homological stability theorems are concerned with sequences of (classifying spaces of) groups. There is by now a standard approach (cf.~\cite{MR2736166}) to the proof of such results which requires a highly connected simplicial complex with a nicely behaved group action in order to study the associated spectral sequence.

Hurwitz spaces as moduli spaces of branched covers of $\C$ appeared in the second half of the 18th century in the work of Hurwitz (\cite{MR1510692}). Their properties helped proving the connectivity of $\mathcal{M}_{g}$ in \cite{MR0245574}, and they play an important role in arithmetic applications such as the Regular Inverse Galois Problem (cf.~\cite{MR1119950}).

In this paper, we study the topology of Hurwitz spaces with respect to homological stabilization. It is worth mentioning the proximity of these spaces to both moduli spaces of Riemann surfaces and configuration spaces: The total space of a branched covering of $\C$ is a Riemann surface, whereas the branch locus defines an element of a configuration space. Having the homological stability theorems for both $\mathcal{M}_{g}$ and $\Conf_{n}$ in mind, it seems worthwhile to study Hurwitz spaces in this direction.

\subsection*{Braids and configurations}
Let $n\in\N$. By $\Br_{n}$, we denote the classical \emph{(Artin) braid group}, generated by $\sigma_{1},\ldots, \sigma_{n-1}$, subject to the relations
\begin{equation*}\label{braidrel}
\begin{alignedat}{2}
 \sigma_i \sigma_{i+1} \sigma_i &= \sigma_{i+1}\sigma_i\sigma_{i+1},\:\:\:\: &&1\leq i \leq n-2, \\
\sigma_i\sigma_j &= \sigma_j\sigma_i, &&|i-j|\geq 2,
\end{alignedat}
\end{equation*}
cf.~\cite{MR3069440}. The \emph{pure braid group} $\PBr_{n} \subset \Br_{n}$ is the kernel of the surjection $\Br_{n} \to \mathfrak{S}_{n}$ which maps $\sigma_{i}$ to the transposition $(i, i+1)$. 
If $\underline\zeta = (\zeta_{1}, \ldots, \zeta_{t})$ is a partition of~$n$, the \emph{colored braid group with coloring $\underline\zeta$} is defined as the kernel of the map $\Br_{n} \to \mathfrak{S}_{n}/ \mathfrak{S}_{\underline\zeta}$, with $\mathfrak{S}_{\underline\zeta} \cong \mathfrak{S}_{\zeta_{1}} \times \ldots \times \mathfrak{S}_{\zeta_{t}}$. For a presentation of these groups, cf.~\cite{MR1465028} and \cite{MR2607077}.

By \cite{MR0141126}, the \emph{(unordered) configuration space} $\Conf_{n}$ of $n$ points in (the interior of) a two-dimensional closed disk $D$ is of type $K(\Br_{n}, 1)$. Associated to the inclusion $\PBr_{n} \subset \Br_{\underline\zeta} \subset \Br_{n}$ of subgroups, there is a sequence of covering space maps
$$
\PConf_{n} \to \Conf_{\underline\zeta} \to \Conf_{n},
$$
between aspherical spaces, where $\PConf_{n}$ is the \emph{ordered configuration space} of $n$ points in $D$. The space $\Conf_{\underline\zeta} = \PConf_{n}/\mathfrak{S}_{\underline\zeta}$ is called the \emph{colored configuration space} of $n$ points in $D$ with coloring $\underline\zeta$.

By \cite{MR0274462}, for any $p\geq 0$, we have $$H_{p}(\Conf_{n}; \Z) \cong H_{p}(\Conf_{n+1};\Z)$$ for $n\geq 2p-2$. If $\XI \in\N^{t}$ and $n\cdot\XI = (n\xi_{1}, \ldots, n\xi_{t})$, it follows from \cite{1312.6327} that for any $p\geq 0$, 
\begin{equation}\label{tran}
H_{p}(\Conf_{n\cdot\XI};\Z) \cong H_{p}(\Conf_{(n+1)\cdot\XI}; \Z)
\end{equation}
for $n\geq \frac{2p}{\min\XI}$, where $\min\XI$ denotes the smallest entry of $\XI$. Notable homological stability results for configuration spaces of surfaces include \cite{MR0358766}, \cite{MR533892}, \cite{MR2909770}, and \cite{MR3032101}, among others.

\subsection*{Homological stability for Hurwitz spaces}

Let $n\in\N$. Furthermore, let $G$ be a finite group, $c = (c_{1}, \ldots, c_{t})$ a tuple of $t$ distinct non-trivial conjugacy classes in $G$, and $\XI = (\xi_{1}, \ldots, \xi_{t})\in\N^{t}$ a partition of $\xi\in\N$. We replace $\C$ by a closed two-dimensional disk~$D$ and consider \emph{marked $n\cdot\XI$-branched $G$-covers of $D$}: We prescribe the covers' \emph{shape vectors} $n\cdot\XI$. With this, we mean that for $i=1,\ldots,t$, exactly $n\xi_{i}$ local monodromies around the branch points must lie in $c_{i}$. We refer to~Section~\ref{hurwitz-spaces} for a more thorough introduction to this kind of branched covers. We denote the space of such covers by $\Hur_{G,n\cdot\XI}^{c}$. This Hurwitz space must be a covering space of $\BBr_{n\cdot\XI} \cong \Conf_{n\cdot\XI}$ with fiber $\cc^{n} = (c_{1}^{\xi_{1}} \times \ldots \times c_{t}^{\xi_{t}})^{n}$, thus
$$
\Hur_{G, n\cdot\XI}^{c} = \EBr_{n\cdot\XI} \times_{\Br_{n\cdot\XI}} \cc^{n},
$$
up to homotopy, where the $\Br_{n\cdot\XI}$-action on $\cc^{n}$ is given by the restriction of the full \emph{Hurwitz action} of $\Br_{n\xi}$ on $G^{n\xi}$ described in~(\ref{hurwitz-action}).

The prior homological stability result for Hurwitz spaces deals with the case where $c \subset G$ is a single conjugacy class and $\XI = 1 \in\N$.
A conjugacy class $c \subset G$ is called \emph{non-splitting} if $c$ generates $G$ and for all subgroups $H\subset G$, $c\cap H$ is either empty or a conjugacy class in $H$.

\begin{theorem*}[\textsc{Ellenberg--Venkatesh--Westerland}, \cite{0912.0325}]
Let $c \subset G$ be a non-splitting conjugacy class. Let $A$ be a field of characteristic zero or prime to the order of $G$. Then there are positive constants $a$, $b$, $d$ such that for all $p\geq 0$,
$$
H_{p}(\Hur_{G,n}^{c}; A) \cong H_{p}(\Hur_{G,n+d}^{c};A)
$$
for $n > ap+b$.
\end{theorem*}

In Section~\ref{homstabhurwitz}, we follow the ideas of Sections~4 through~6 of \cite{0912.0325}. The main technical complication in comparison to the prior result is the fact that the colored braid group action on the set of $q$-simplices of the \emph{colored plant complexes} we introduce in Section~\ref{plants} is in general not transitive.

In Section~\ref{hurwitz-spaces}, we explain why the $A$-module
$$
R = \bigoplus_{n\geq 0} H_{0}(\Hur_{G,n\cdot\XI}^{c};A)
$$
has the structure of a graded ring, where the grading is in the $n$-variable. For a central homogeneous element $U \in R$, we define $D_{R}(U) = \max\{\deg R/UR, \deg R[U]\}$, where $R[U]$ is the $U$-torsion in $R$.

Our main theorem is proved in Section~\ref{homstabhurwitz}:

\begin{theorem-main}\label{the-theorem}
Suppose there is a central homogeneous element $U\in R$ of positive degree such that $D_{R}(U)$ is positive and finite. Then, for any $p\geq 0$, multiplication by~$U$ induces an isomorphism
$$
H_{p}(\Hur_{G,n\cdot \XI}^{c}; A) \overset{\sim}{\to} H_{p}(\Hur_{G,(n+\deg U)\cdot \XI}^{c}; A)
$$
whenever $n > (8 D_{R}(U) + \deg U)p + 7 D_{R}(U) + \deg U$.
\end{theorem-main}

Our theorem generalizes the prior theorem to the case of multiple conjugacy classes. Indeed, for $c$ a single non-splitting conjugacy class, \cite[Lemma~3.5]{0912.0325} shows that the condition of Theorem~\ref{the-theorem} is satisfied.

We say that $G$ is \emph{invariably generated} by $c$ if for all choices of elements $g_{i} \in c_{i}$, $i = 1, \ldots, t$, the group generated by $g_{1}, \ldots, g_{t}$ is equal to $G$. 
We denote by $\partial U$ the product of the entries of a vector $U \in \cc^{d}$. Note that such a vector may be identified with a homogeneous element of $R$, cf.~Remark~\ref{comb-descr}.
Applying Theorem~\ref{the-theorem}, a result from \cite{1212.0923}, and (\ref{tran}), we are able to deduce a concrete homological stability statement in the case where $c$ invariably generates $G$. 

\begin{theorem-main}[Theorem~\ref{thm-connected}]
Assume $c$ invariably generates $G$. Then, for any $U \in \cc^{d}$ with $\partial U = \id$ and any $p \geq 0$, there are isomorphisms
\begin{align*}
H_{p}(\Hur_{G, n\cdot\XI}^{c}; \Z ) &\cong H_{p}(\Hur_{G, (n+d)\cdot\XI}^{c}; \Z )\\
H_{p}(\Hur_{G, n\cdot\XI}^{c}; \Q ) &\cong H_{p}(\Hur_{G, (n+1)\cdot\XI}^{c}; \Q )
\end{align*}
for $n> (8D_{R}(U) + d)p+7D_{R}(U) + d$, and a constant $b\in\N$ such that
$$
H_{p}(\Hur_{n\cdot\XI}^{c};\Q) \cong H_{p}( \Conf_{n\cdot\XI}; \Q ) \otimes_{\Q} \Q^{b}
$$
in the same range.
\end{theorem-main}

\subsection*{Simplicial complexes in homological stability proofs}

In the study of homological stability, simplicial complexes are ubiquitous. Given a sequence of groups $\{G_{n}\}$ and highly connected simplicial complexes $\mathcal{O}^{n}$ for all $n$ such that $G_{n}$ acts transitively on the set of $q$-simplices of $\mathcal{O}^{n}$ for all $q$ and with stabilizers isomorphic to $G_{n-q-1}$, the spectral sequence associated to the semi-simplicial space $\EG_{n} \times_{G_{n}} \mathcal{O}^{n}$ yields a description of the homology of $\BG_{n}$ in terms of the homology of spaces $\BG_{m}$, for $m<n$. This makes inductive arguments possible.

The \emph{ordered arc complex} (cf.~\cite{MR3135444}) turns out to have the right properties for mapping class groups of surfaces (leading to the homological stability theorem for $\mathcal{M}_{g}$). For the Artin braid group, the \emph{arc complex} (though not used in the original article \cite{MR0274462}) is a suitable choice. This complex (which is contractible by \cite{Damiolini}) has also been employed in the homological stability proof in \cite{0912.0325}.

We run into a couple of problems when examining homological stability for Hurwitz spaces: First, the Hurwitz spaces we consider are usually disconnected. This can be fixed by using the fact that they are finite covers of $K(G,1)$ spaces, where $G$ is a {colored braid group}. Secondly, there is no highly connected simplicial complex at hand which admits a well-behaved colored braid group action. For this purpose, we define and investigate  \emph{plant complexes} in Sections~\ref{plants} to~\ref{combinatorics}. Thirdly, the group action on these complexes is generally not transitive. This last point makes a more extensive homological analysis in Section~\ref{homstabhurwitz} necessary.

The definition of plant complexes generalizes both the arc complex and the \emph{fern complex} from~\cite{1410.0923}, hence the name. In Section~\ref{combinatorics}, we focus on a specific class of \emph{colored} plant complexes in order to obtain the following result which is essential to our homological stability proof:

\begin{theorem-main}[Theorem~\ref{delta-conn}, Lemma~\ref{orbits}, Lemma~\ref{stabilizers}]\label{thm-5}
For $n\in\N$ and $\XI\in\N^{t}$, there exists an $(n-1)$-dimensional and at least $\left( \lfloor\frac{n}{2}\rfloor -2\right)$-connected simplicial complex which admits a generally non-transitive action by the colored braid group $\Br_{n\cdot\XI}$. The stabilizer of a $q$-simplex under this action is isomorphic to $\Br_{(n-q-1)\cdot\XI}$.
\end{theorem-main}

\subsection*{Acknowledgements}
This paper contains the central result of my 2016 Ph.D.~thesis. I would particularly like to thank my advisor Michael L\"onne for all the inspiring discussions. Furthermore, I am thankful to Craig Westerland for his supportive and helpful answers to my questions, and to Matthias Zach for numerous mathematical dialogues. I would like to appreciate the excellent mathematical environment and the pleasant colleagues I was offered by the Institute of Algebraic Geometry at the Leibniz University of Hannover over the last three years.

%%%

\section{Plants and plant complexes}\label{plants}

Let $S$ be a connected surface with non-empty boundary and $\DEL = (\delta_1, \ldots, \delta_t)\in\N^{t}$ a partition of $\delta = \sum_{i=1}^t \delta_i$. Let furthermore $\Delta$ be a set of $\delta$ points in the interior of $S$, partitioned as $\Delta = \Delta_1 \sqcup \ldots \sqcup \Delta_t$, where $|\Delta_{i}| = \delta_{i}$ for all $i=1,\ldots, t$.
Finally, let $*$ be a fixed point in $\partial S$.

An \emph{arc} is a smooth embedding $\gamma\colon I \to S$ with $\gamma(0) = *$ and $\gamma(1) \in \Delta$, meeting the boundary transversally, and with interior entirely in $S\setminus(\partial S\cup \Delta)$.

\begin{definition}\label{def-plant}
Let
 ${\XI} = (\xi_1, \ldots, \xi_t)\in \N^{t}$ and $\xi = \sum_{i=1}^{t}\xi_i $.
\begin{enumerate}[(i)]
\item
A $\XI$\emph{-plant} in $(S,\Delta)$ is an unordered $\xi$-tuple of arcs in $S$ which only meet at~$*$, where for some permutation $\sigma\in \mathfrak{S}_{t}$, exactly $\xi_i$ arcs end at points of $\Delta_{\sigma(i)}$, for $i=1,\ldots,t$.
The tuple $\XI$ is called the \emph{pattern}.

\item A \emph{colored $\XI$-plant} in $(S,\Delta)$ is a $\XI$-plant in $(S,\Delta)$ with the requirement that for $i=1, \ldots t$, exactly $\xi_{i}$ arcs end at points of $\Delta_{i}$.

\item Two $\XI$-plants $v, w$ in $(S,\Delta)$ are called \emph{equivalent} if there is an isotopy of $S$ fixing $\partial S \cup\Delta$ pointwise that transforms one plant into the other.

\item
For any plant $u$,
we write $u^\circ = u\setminus (\{*\} \cup \Delta)$ for its \emph{interior}.

\item
We say that two plants (or arcs) $v$ and $w$ (not necessarily of the same pattern) have $s$ \emph{points of intersection}
if $s$ is the minimal number such that there are are plants $v'$ and $w'$ equivalent to $v$ and $w$, respectively, such that $v'^{\circ}$ and $w'^{\circ}$ share $s$ points in $S\setminus(\Delta\cup\partial S)$. We write $v.w=s$. For $v.w=0$, we call $v$ and $w$ \emph{disjoint}.
\end{enumerate}
\end{definition}

\begin{figure}
\centering
\captionsetup[subfigure]{labelformat=empty}
\begin{subfigure}[b]{0.31\linewidth}
\centering
\includegraphics[scale=0.14]{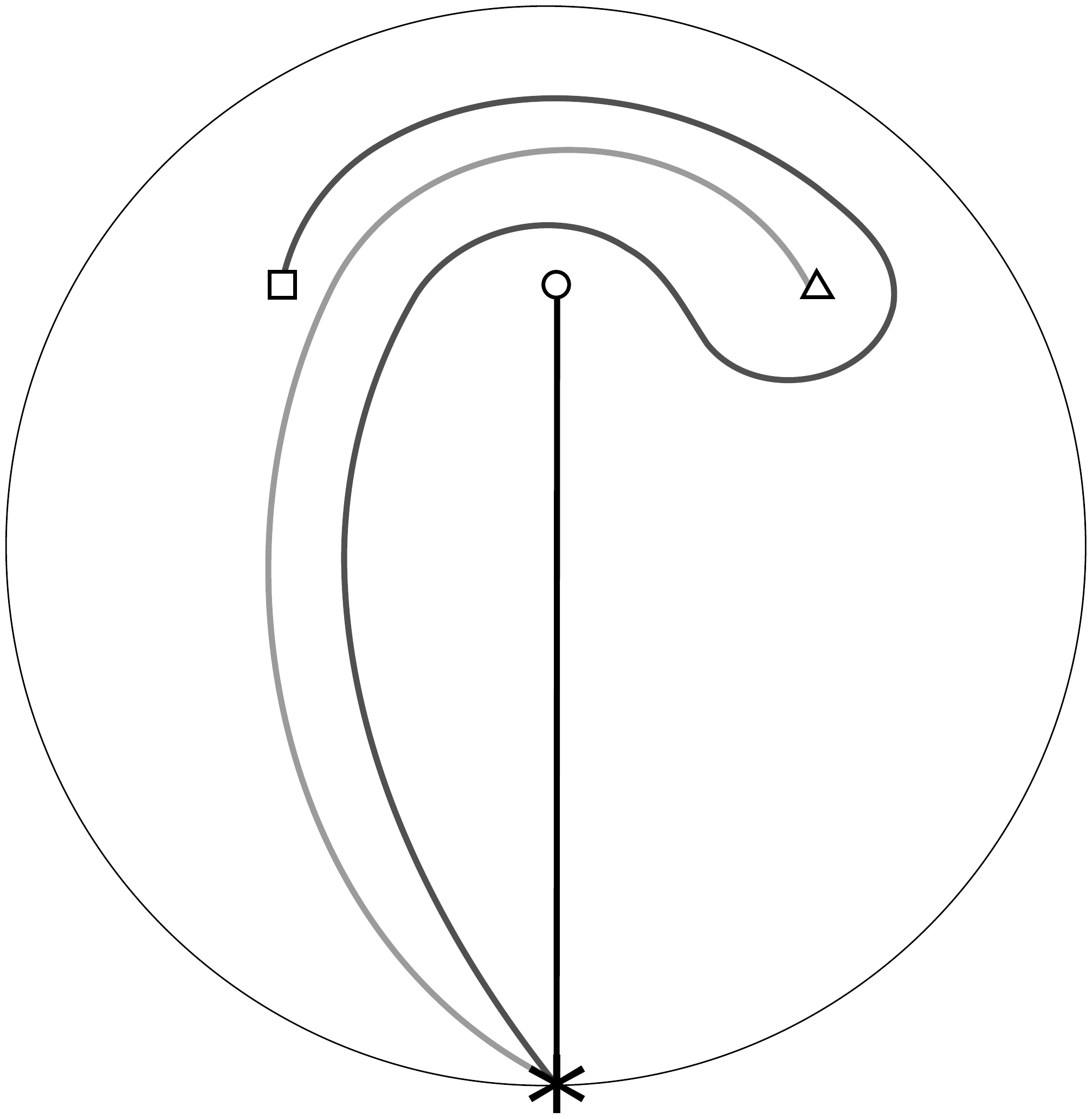}
\caption{$\DEL = \XI = (1,1,1)$}
\end{subfigure}
\begin{subfigure}[b]{0.31\linewidth}
\centering
\includegraphics[scale=0.17]{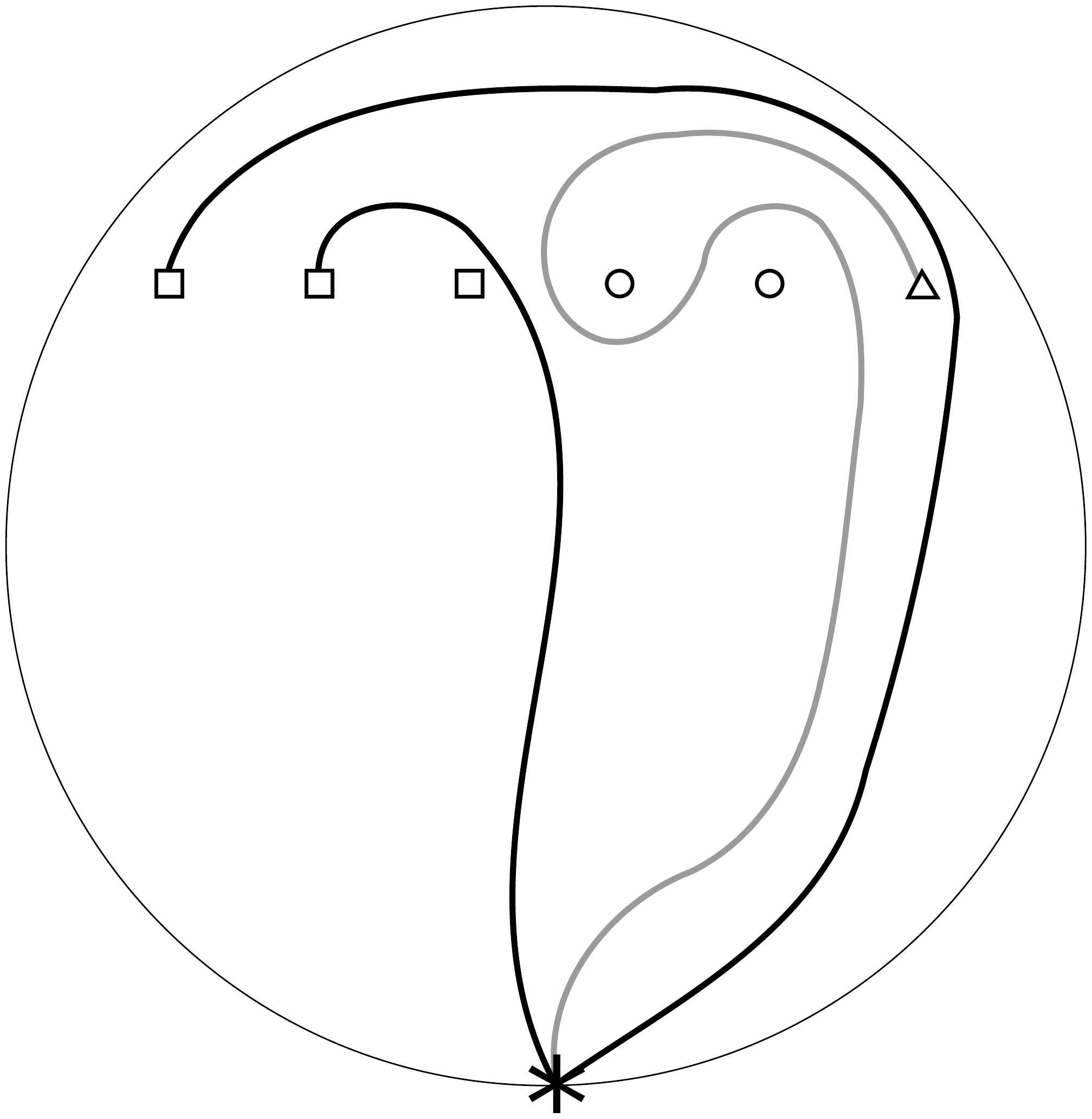}
\caption{$\DEL = (3,2,1)$, $\XI = (0,2,1)$}
\end{subfigure}
\begin{subfigure}[b]{0.31\linewidth}
\centering
\includegraphics[scale=0.17]{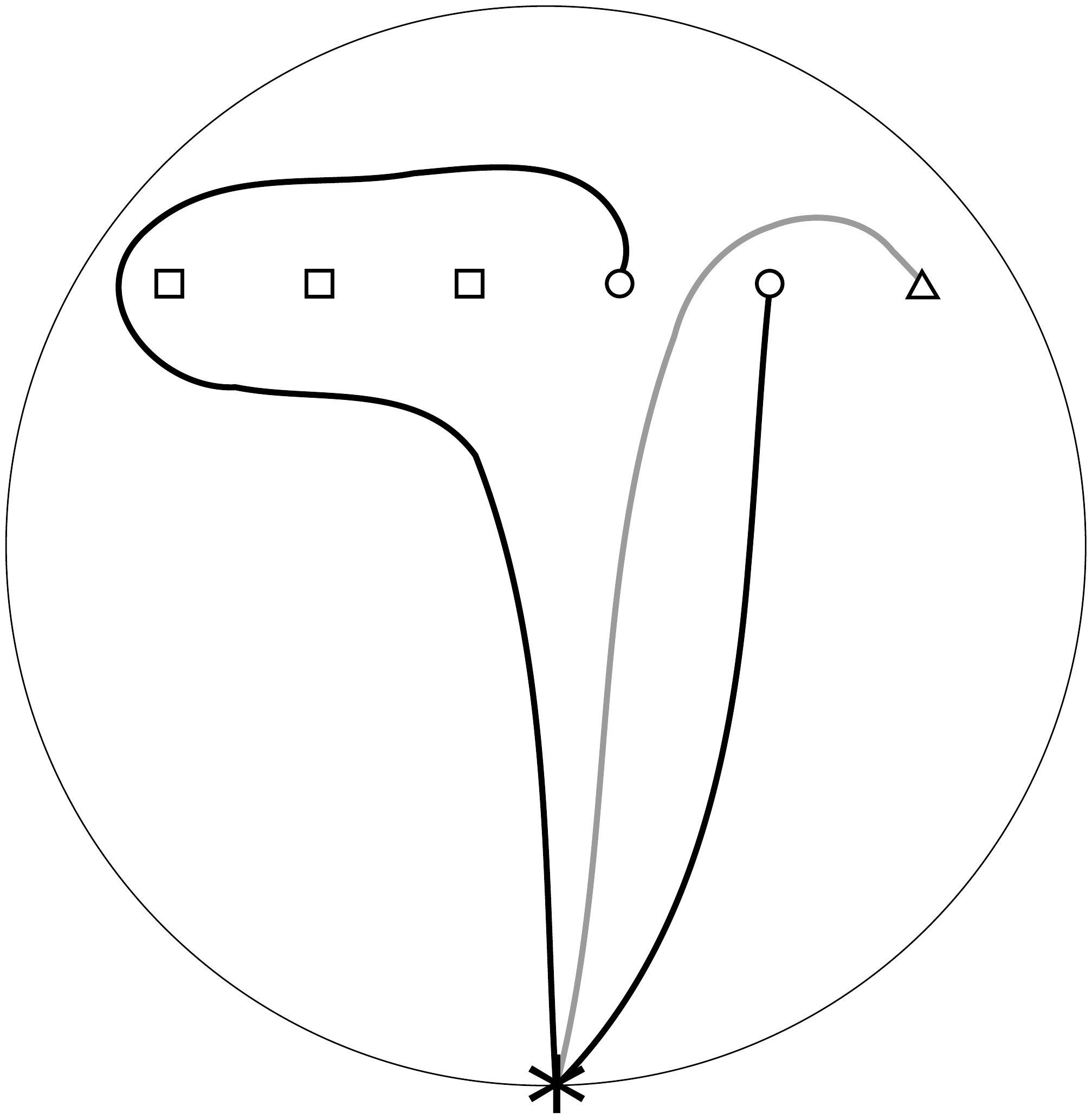}
\caption{$\DEL = (3,2,1)$, $\XI = (0,2,1)$} 
\end{subfigure}
\caption{Examples of $\XI$-plants in a disk.}
\label{first-plants}
\end{figure}

First examples of plants in a two-dimensional closed disk $D$ can be seen in Figure~\ref{first-plants}. Given $\DEL$ and $\XI$ as in the captions, the left and right plants are colored. Changing $\XI$ to $(2,0,1)$, the middle plant is colored as well. 

\begin{lemma}\label{intersection_lemma}
Let $v = (a_1, \ldots, a_\zeta)$ and $w = (b_1, \ldots, b_\xi)$ be plants in $(S, \Delta)$ with arbitrary patterns.
The product $v.w$ is finite and arcwise distributive, i.e., we have $v.w = \sum_{i=1}^\zeta\sum_{j=1}^\xi a_i.b_j.$
\end{lemma}

\begin{proof}
For the first part, it suffices to show that generically, two arcs $a_{1}$, $b_{1}$ meet in finitely many points. By the transversality theorem (cf.~\cite{MR0061823}), the space of smooth embeddings $b_{1}\colon I \to S$ which are transversal to $a_{1}$ is dense in the space of all smooth embeddings. Now, $a_{1}\colon I \to S$ and $b_{1}\colon I \to S$ being transversal implies that $b_{1}^{-1}(a_{1}(I)) \subset I$ is a $0$-dimensional submanifold. Such a submanifold necessarily consists of only finitely many points.

The inequality '$\geq$' is clear by definition of the products $a_i.b_j$.
For the other inequality, let $v$, $w$ be in minimal position, i.e.,
$v.w = |v^{\circ} \cap w^{\circ}|$ and all intersections are transversal. Assume $v.w >  \sum_{i=1}^\zeta\sum_{j=1}^\xi a_i.b_j$.

By assumption, there exist indices $p$, $q$ with
\begin{equation}\label{gnull}
| a^{\circ}_p \cap  {b^{\circ}_q}| > a_p.b_q.
\end{equation}
Therefore, there must be segments of $a_{p}$ and $b_{q}$ whose union is a continuous loop. Choose $k$ and $l$ among all such $p$, $q$ such that such a loop has no intersection with further arcs of $v$ or $w$. Then there is a closed disk $D_0 \subset S$ bounded by segments of $a_k$ and $ b_l$, containing no other arc segments of $v$ or $w$.

%%%%%
\begin{figure}
\centering
\captionsetup[subfigure]{labelformat=empty}
\begin{subfigure}[b]{0.28\linewidth}
\centering
\def\svgwidth{\columnwidth}
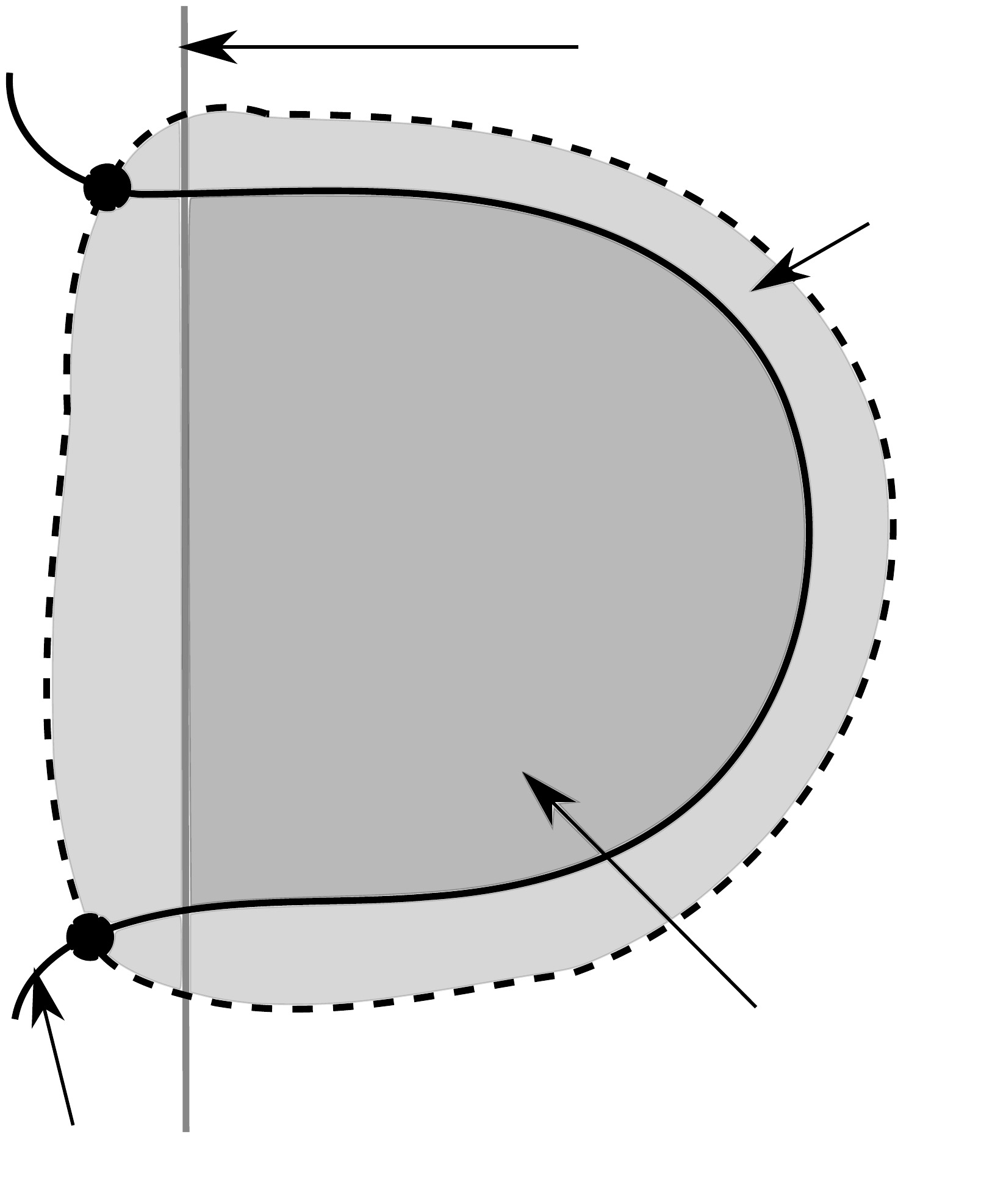
\end{subfigure}
\begin{minipage}[b][3cm][c]{0.05\linewidth}
$\rightsquigarrow$
\vspace{2cm}
\end{minipage}
\begin{subfigure}[b]{0.245\linewidth}
\centering
\def\svgwidth{\columnwidth}
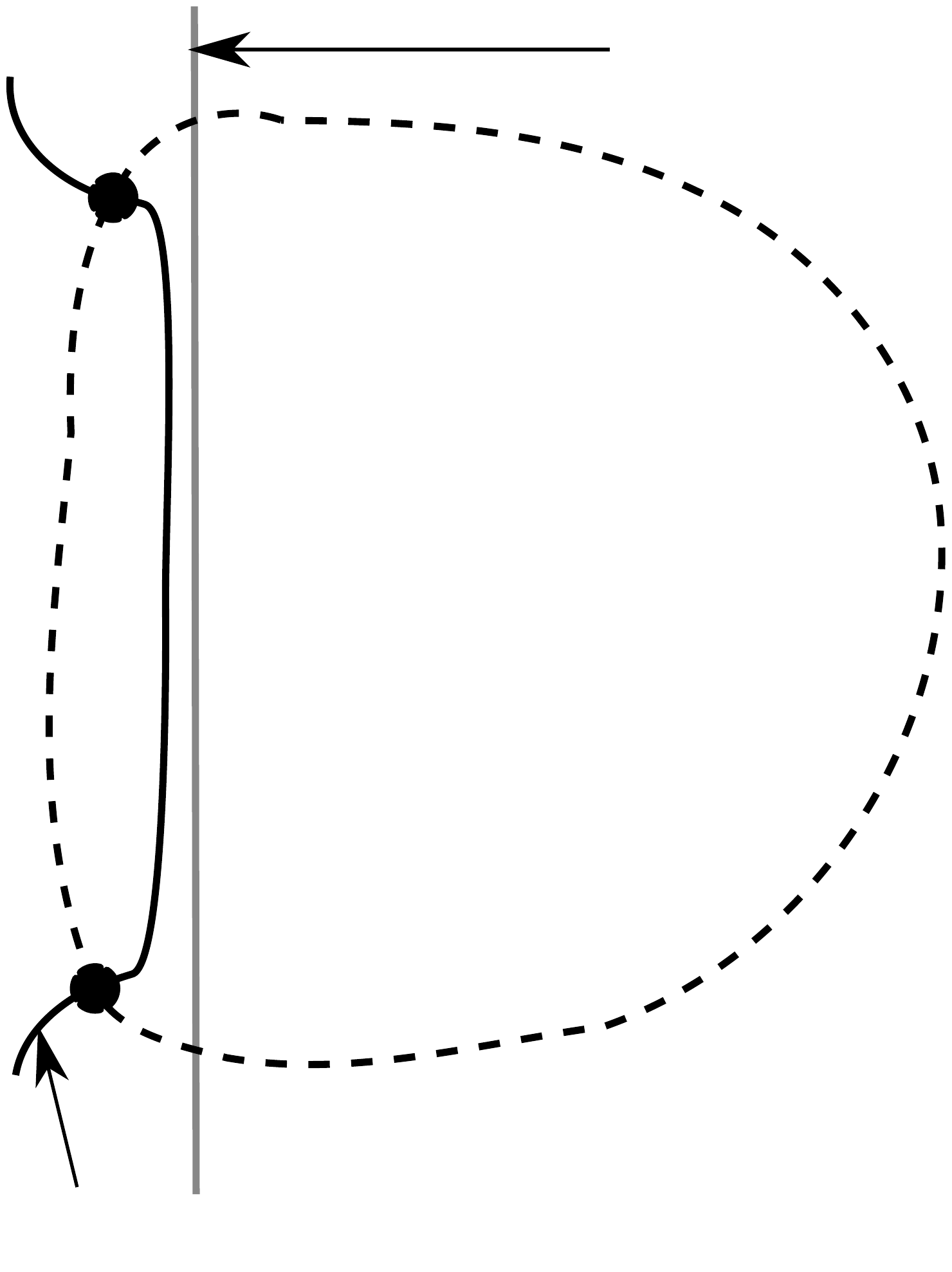
\end{subfigure}
\caption{The isotopy $H$ from the proof of Lemma~\ref{intersection_lemma}.}
\end{figure}
%%%%%

After eventual slight smooth deformations of $a_{k}$ or $ b_l$, we may assume that neither~$*$ nor~$\Delta$ share points with $D_0$. These deformed arcs $a'_{k}$, $b'_{l}$ can be chosen such that for the plants $v'$, $w'$ defined by replacing $a_{k}$ by $a'_{k}$ and $b_{k}$ by $b'_{k}$, respectively, we have
\begin{equation}\label{cond3}
|v'^{\circ} \cap w'^{\circ}| \leq |v^{\circ} \cap w^{\circ}| + 1.
\end{equation}
Indeed, if both $*$ and~$\Delta$ intersected $D_0$, there would be no interior point of intersection between $a_k$ and $ b_l$, contradicting~(\ref{gnull}). Denote by $D_{1}$ the resulting closed disk bounded by segments of $a'_{k}$ and $b'_{l}$.

We now define an isotopy $H\colon I\times S \to S$ that satisfies
\begin{align}
\label{cond1}
|H(\{1\} \times {a'_{k}}^{\circ}) \cap  {{b_{l}'}^{\circ}}| &<| {a_{k}'}^\circ \cap  {b_{l}'}^\circ| , \\
\label{cond2}
|H(\{1\} \times {v'}^{\circ}) \cap {w'}^{\circ}| &\leq  | {v'}^{\circ} \cap  {w'}^{\circ}| - 2.
\end{align}

Let $\epsilon>0$ and $D_1^\epsilon$ an open $\epsilon$-neighborhood of $D_1$, where we choose $\epsilon$ such that there are no segments of arcs in $D_1^\epsilon$ other than $a'_{k}$ and $b'_{l}$, and such that $D_1^\epsilon$ lies entirely in the interior of $S$.
Now, the arc segment $a'_{k} \cap \overline{D_1^\epsilon}$ is isotopic (fixing endpoints) to an arc segment in $\overline{D_{1}^{\epsilon}}$ which does not intersect $b'_{l} \cap \overline{D_1^\epsilon}$. By the isotopy extension theorem (cf.~\cite{MR0123338}), this isotopy may be extended to an ambient isotopy $h\colon I \times \overline{D_{0}^{\epsilon}} \to  \overline{D_1^\epsilon}$ which fixes the boundary circle $\partial\overline{D_1^\epsilon}$ pointwise.
We extend $h$ by the identity on $S \setminus D_1^\epsilon$ and denote the resulting isotopy of $S$ by  $H$.

Now, (\ref{cond1}) is satisfied since we push $D_1$ across $ b'_l$ and thus remove two intersections. As the potential slight deformation of $a'_k$ and $ b'_l$ creates at most one extra intersection, the application of $H$ removes at least one intersection point. Then, (\ref{cond2}) follows from the choice of $\epsilon$.

From (\ref{cond3}) and (\ref{cond2}), we obtain $|H(\{1\} \times {v'}^{\circ}) \cap {w'}^{\circ}| < |v^{\circ} \cap w^{\circ}| = v.w$, which contradicts the definition of the intersection number. The assertion follows.
\end{proof}

\begin{definition}
With notation as above, we define:
\begin{enumerate}[(i)]
 \item The \emph{full $\XI$-plant complex} $\FP(S)$ is the simplicial complex with isotopy classes of $\XI$-plants in $(S,\Delta)$ as vertices. A $q$-simplex in $\FP(S)$ is a set of $q+1$ isotopy classes of $\XI$-plants on $(S,\Delta)$ which can be embedded with disjoint interiors.
 \item The \emph{$\XI$-plant complex} $\mathcal{P}_{\DEL}^{\XI}(S)$ is the subcomplex of $\FP(S)$
 which contains the simplices $\alpha\in\FP(S)$ such that no two plants of $\alpha$ share a point in $\Delta$.
 \item The \emph{full colored $\XI$-plant complex} $\FO(S) \subset \FP(S)$ and the \emph{colored $\XI$-plant complex}  $\cO(S) \subset \cP(S)$ are the subcomplexes defined by the restriction that only colored plants are allowed as vertices.
 \end{enumerate}
\end{definition}

\begin{figure}
\centering
\captionsetup[subfigure]{labelformat=empty}
\begin{subfigure}[b]{0.49\linewidth}
\centering
\includegraphics[scale=0.2]{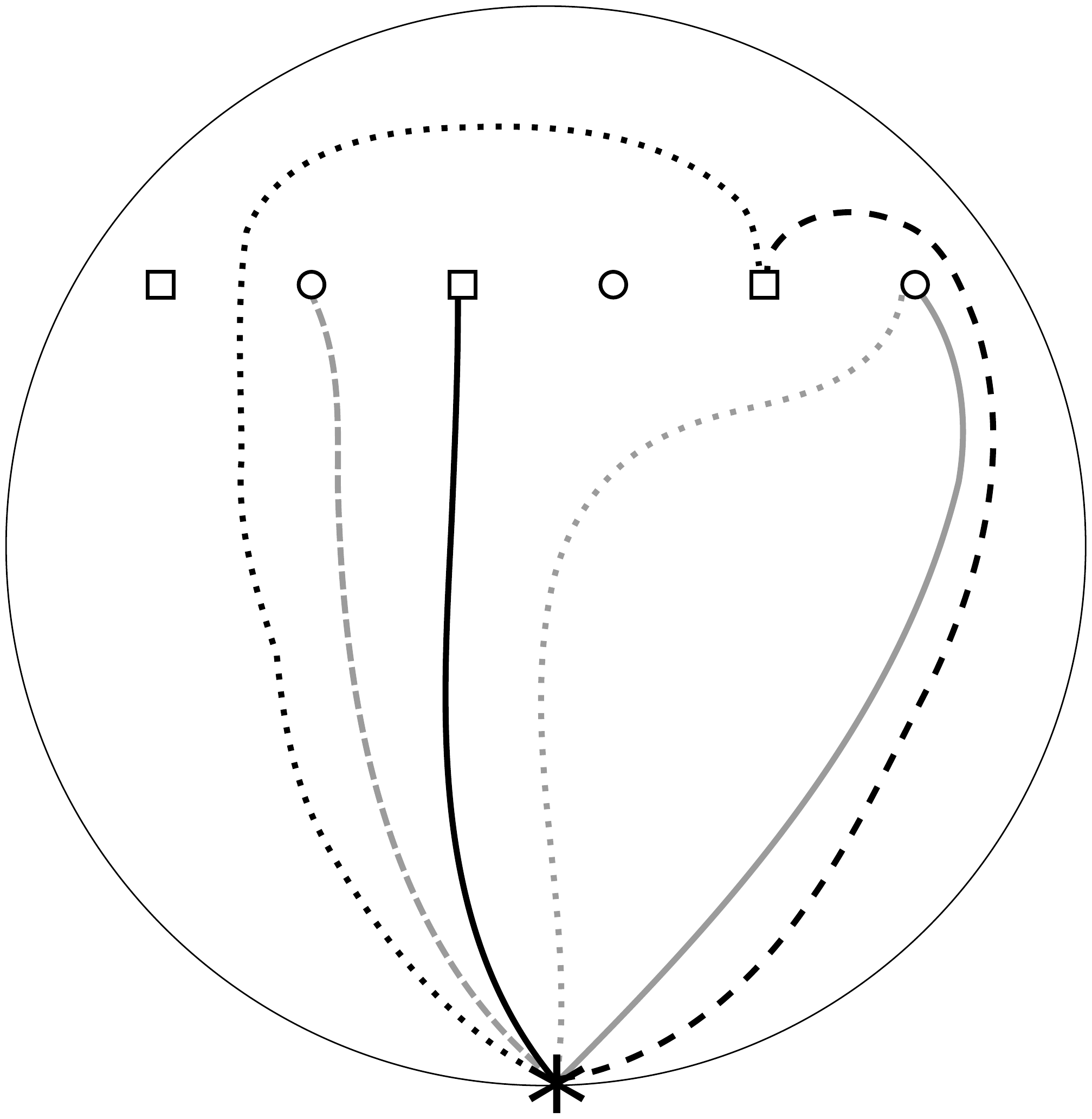}
\caption{$2$-simplex in $\mathrm{F}\mathcal{P}_{(3,3)}^{(1,1)} (D)$}
\end{subfigure}
\begin{subfigure}[b]{0.49\linewidth}
\centering
\includegraphics[scale=0.2]{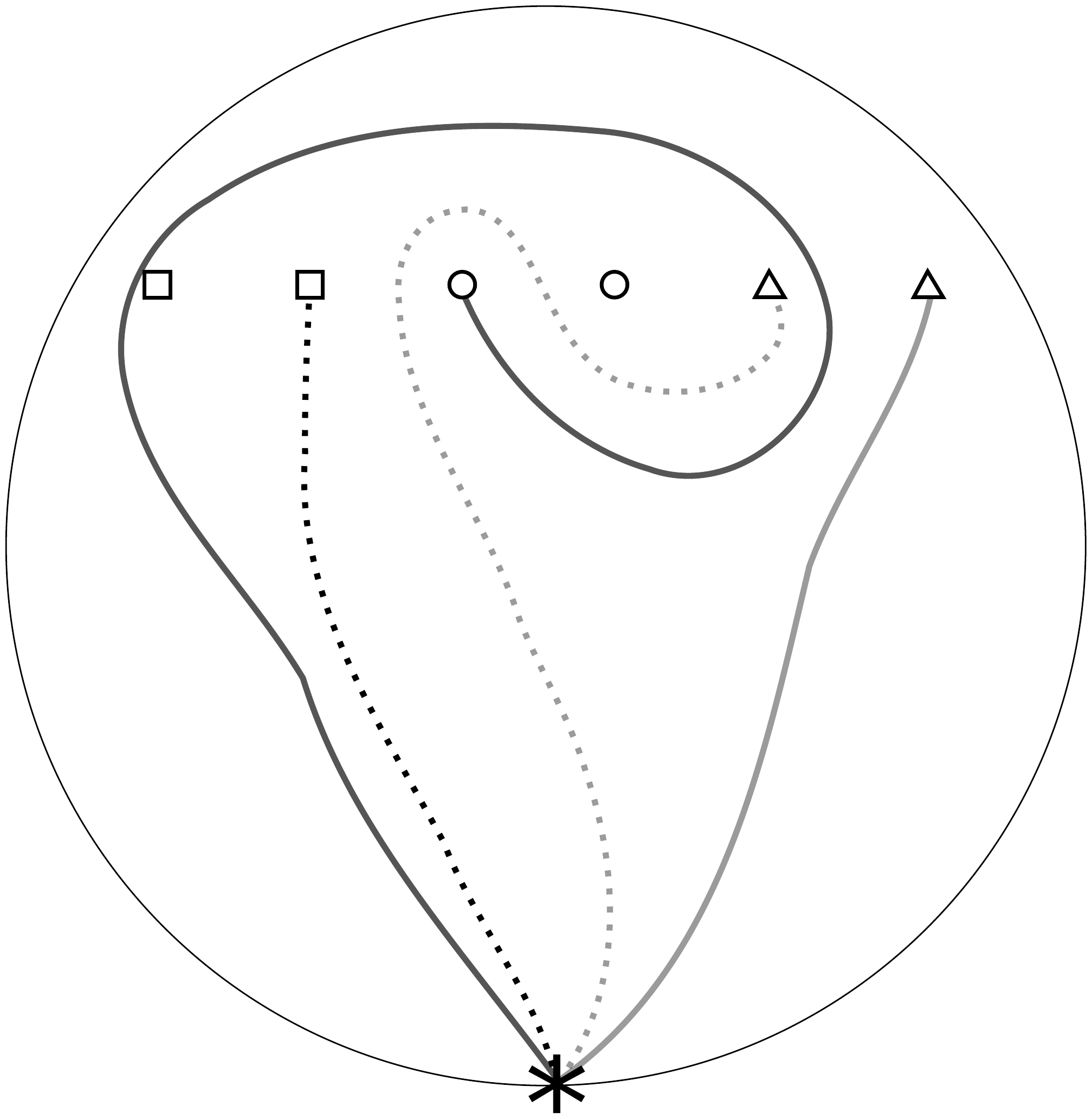}
\caption{$1$-simplex in $\mathcal{P}_{(2,2,2)}^{(1,1,0)} (D)$}
\end{subfigure}
\caption{Representatives of simplices in plant complexes (1) -- colors indicate the type of endpoints, line styles distinguish between different plants.}
\end{figure}

\begin{remark}
By definition, we have the following diagram:
$$
\xymatrix{
\cO(S) \ar@{}[d]|-*[@]{\subset} \ar@{}[r]|-*[@]{\subset} &\cP(S) \ar@{}[d]|-*[@]{\subset}\\
\FO(S) \ar@{}[r]|-*[@]{\subset} &\FP(S)}
$$
As there are only finitely many isotopy classes of arcs in $(S,\Delta)$, all plant complexes have a finite number of simplices.
\end{remark}

We use two partial orderings of $\N_{0}^t$:
\begin{itemize}
\item $(x_1, \ldots, x_t) \preccurlyeq (y_1, \ldots, y_t)$ if there is a permutation $\sigma\in \mathfrak{S}_{t}$ such that for all $i=1,\ldots, t$, we have $x_i \leq y_{\sigma{(i)}}$, and
\item $(x_1, \ldots, x_t) \leq (y_1, \ldots, y_t)$ if $x_i \leq y_i$ for all $i=1,\ldots, t$.
\end{itemize}
Immediately from the definitions, we obtain:

\begin{lemma}\label{nonempty}
Both $\FP(S)$ and $\cP(S)$ are non-empty if and only if $\XI\preccurlyeq\DEL$, and
$\FO(S)$ and $\cO(S)$ are non-empty if and only if $\XI \leq \DEL$.
\end{lemma}

\begin{remark}\label{plant-order} 
There is a natural total order on the vertices of simplices in $\cP(S)$ ($\cO(S)$) which induces the structure of an ordered simplicial complex: By imposing a Riemannian structure on $S$, we may assume that all arcs are parametrized by arc length. Now for plants $v$, $w$ of the same pattern, we write $v<w$ if and only if in a non-intersecting realization of $v$ and $w$, there is an arc in $v$ whose inward pointing unit tangent vector at $*$ occurs in clockwise order before any inward pointing unit tangent vector of an arc in $w$.
\end{remark}

\begin{figure}
\centering
\captionsetup[subfigure]{labelformat=empty}
\begin{subfigure}[b]{0.49\linewidth}
\centering
\includegraphics[scale=0.2]{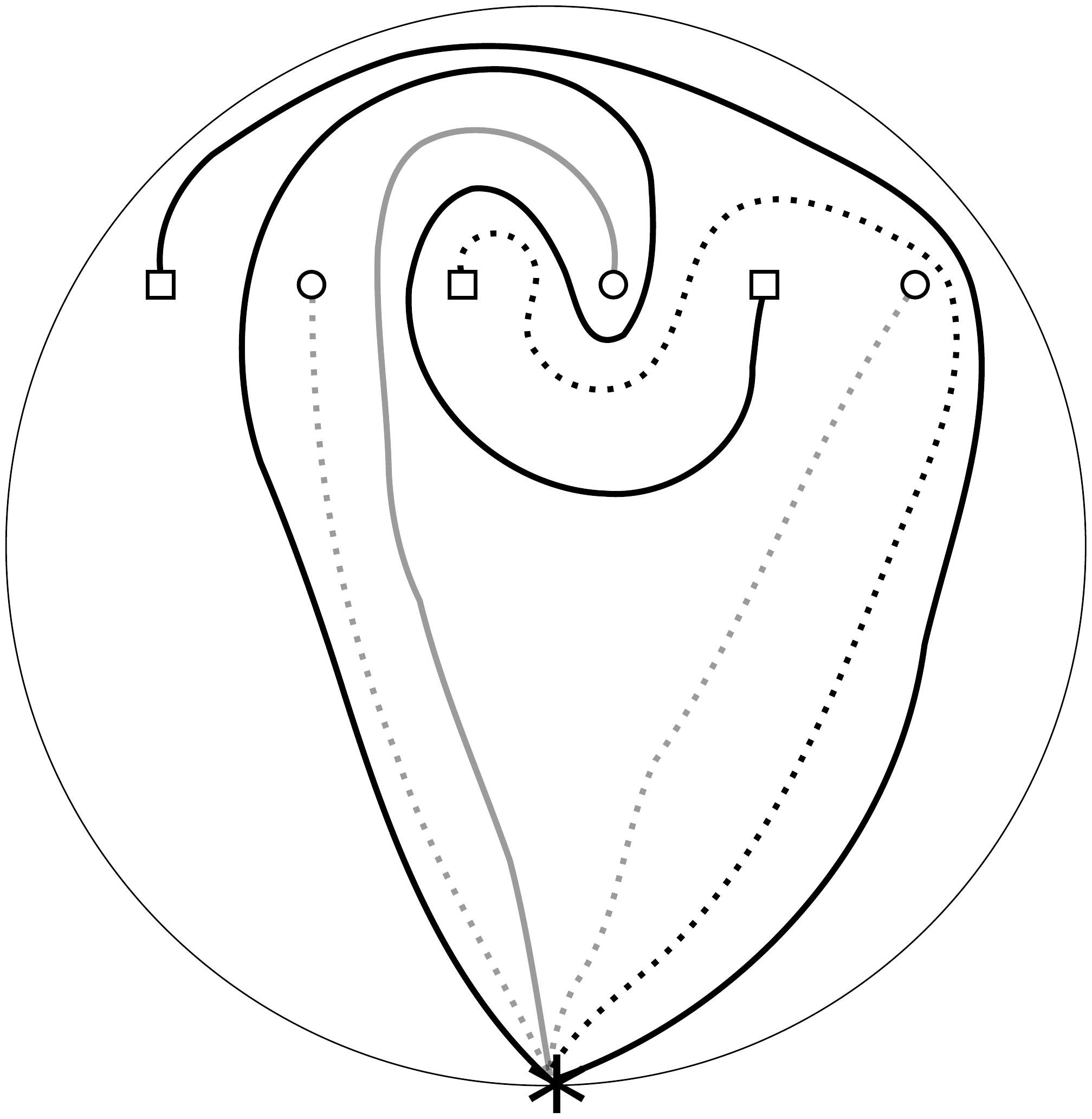}
\caption{$1$-simplex in $\mathcal{P}_{(3,3)}^{(2,1)} (D)$}
\end{subfigure}
\begin{subfigure}[b]{0.49\linewidth}
\centering
\includegraphics[scale=0.2]{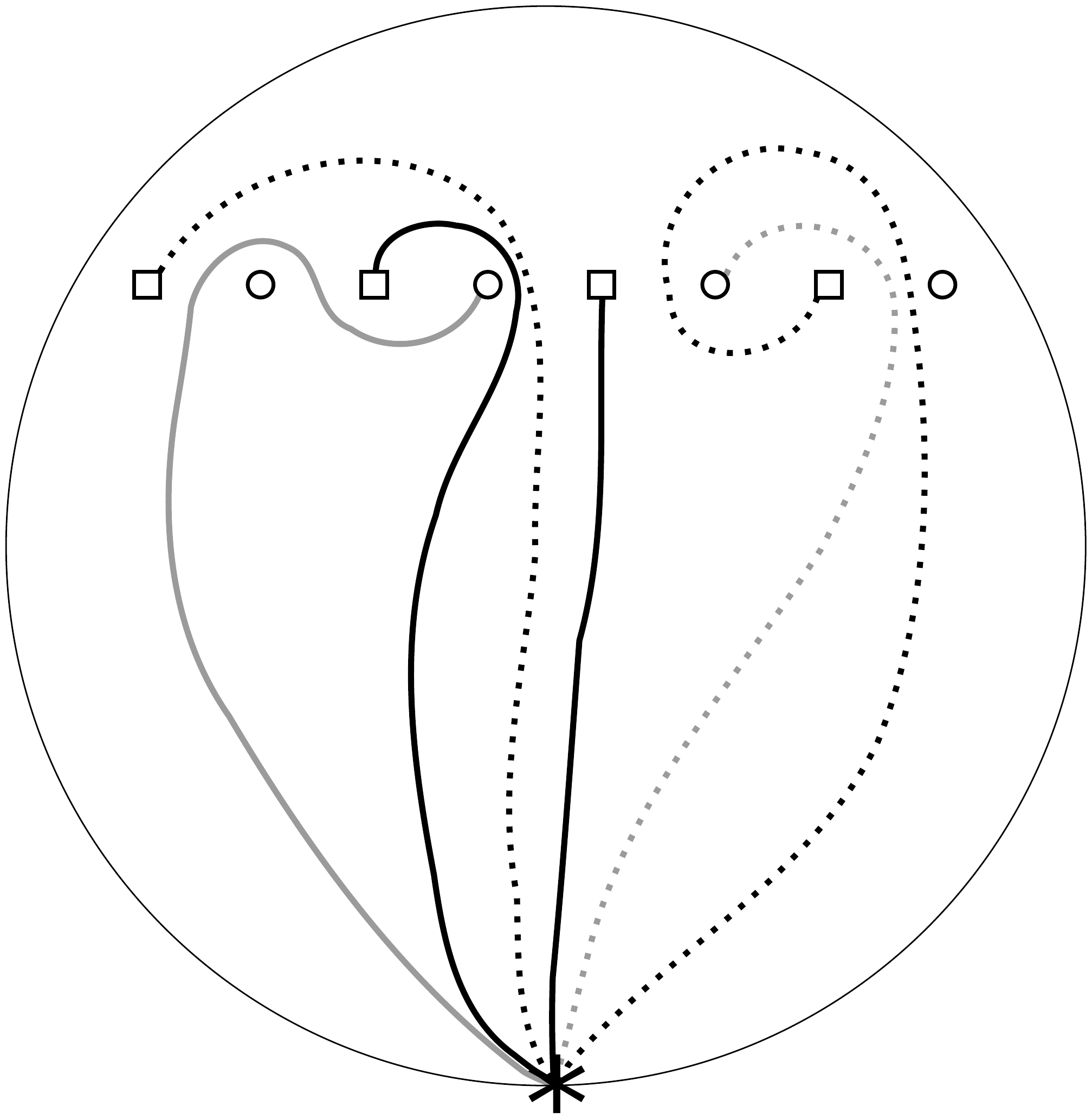}
\caption{$1$-simplex in $\mathcal{O}_{(4,4)}^{(2,1)} (D)$}
\end{subfigure}
\caption{Representatives of simplices in plant complexes (2).}
\end{figure}

\section{Connectivity analysis}

In this section, we always assume that the complexes are non-empty, so we have $\XI\preccurlyeq\DEL$ for plant complexes, and $\XI\leq\DEL$ for colored plant complexes.
Similar connectivity proofs can be found in~\cite[Sect.~4]{MR3135444} and~\cite[Sect.~2]{1410.0923}.

We defined the simplicial complexes abstractly. In order to talk about the \emph{connectivity} of a complex $\mathcal{O}$, we need to consider its geometric realization, which we denote by $|\mathcal{O}|$.

\begin{proposition}\label{contractible}
Both $|\FP(S)|$ and $|\FO(S)|$ are contractible.
\end{proposition}

\begin{proof}
In the proof, we construct a flow similar to the \emph{Hatcher flow} introduced in~\cite{MR1123262}. The proof is carried out for $|\FP(S)|$. It is fully analogous for $|\FO(S)|$. 

In the following, we switch freely between plants as subsets of $S$, plants as vertices of plant complexes, and their respective isotopy classes if no misunderstandings are possible.

 We fix a plant $v$ in $\FP(S)$ with arcs $a_1, \ldots, a_\xi$ in a fixed order.
 Our goal is to show that $|\FP(S)|$ deformation retracts onto $|\Star(v)|$, which is contractible.
 
 We order the interior points of $v$ in the following way:
 \begin{itemize}
  \item $x \prec y$ if $x\in a_i$, $y\in a_j$ for $i<j$
  \item If $x,y \in a_i$, $x\prec y$ if $x$ is closer to $*$ along $a_i$ than $y$.
 \end{itemize}
 
 Let $\alpha = \langle w_0, \ldots, w_p \rangle$ be a $p$-simplex of $\FP(S)$ with representative plants~$w_{i}$ chosen such that the number of intersections with $v$ is minimal.
 The ordering of the interior points of $v$ induces an order on the set $(w_{0}\cup\ldots\cup w_{p})\cap v^\circ$ of intersection points, which we denote by $g_1, \ldots, g_k$.
 At $g_i$, the plant $w_{j_i}$ intersects the arc $a_{l_i}$.
 
 In an $\epsilon$-neighborhood of $g_{1}$, we erase the segments of the arc of $w_{j_{1}}$ which contains~$g_{1}$ and join the two loose ends to $*$ by straight lines. We denote by $C(\alpha)$ the plant that is obtained from $w_{j_1}$ by replacing the arc containing $g_1$ with a smooth approximation of the one of the two newly created paths that is an arc.
 Because of the order we put on the intersection points, $C(\alpha)$ is a plant which is disjoint from the plants in $\alpha = \alpha^{(1)}$.

In the following, we misuse notation by allowing vertices to occur more than once in a simplex. In this sense, by $\langle c_{0}, \ldots, c_{p}\rangle$ we denote the simplex with vertices $\{c_{0}, \ldots, c_{p}\}$, which might be of dimension smaller than $p$.

We now define a finite sequence of simplices inductively. We start with $i=1$, the first intersection point $g_{i} = g_{1}$, and the simplex $$r_1(\alpha) = \langle w_0, \ldots, w_p, C(\alpha)\rangle = \langle \alpha^{(1)},  C(\alpha^{(1)})\rangle,$$ and execute the procedure below. In every step, we choose representative plants for the vertices of the simplices such that the number of intersections with $v$ is minimal.

\begin{enumerate}[(1)]
\item Increase $i$ by one. Stop if $i=k+1$, otherwise go to the next step.
  \item If the intersection at $g_{i}$ is not yet resolved in $\alpha^{(i)}$, replace the plant of $\alpha^{(i)}$ that contains~$g_i$ with $C(\alpha^{(i)})$, denote the resulting $p$-simplex by $\alpha^{(i+1)}$, and set
  $r_{i}(\alpha) = \langle \alpha^{(i)} , C(\alpha^{(i)})\rangle.$
  Else, set $\alpha^{(i+1)} = \alpha^{(i)}$ and $r_{i}(\alpha) = \langle \alpha^{(i)} , w_{j_{i}}'\rangle$, where $w_{j_{i}}'$ is the $j_{i}$-th plant of $\alpha^{(i)}$.
 \item 
 Go to step 1.
\end{enumerate}

By the above remarks about disjointness, we produce simplices $r_{i}(\alpha)$ of dimension at most $p+1$ at each step. The $p$-simplex $\alpha^{(k+1)}$ is in the star of $v$, since all of its plants are disjoint from $v$.
By construction, $\alpha^{(k+1)}$ is a face of $r_{k}(\alpha)$.

Now, we may use the sequence $r_1(\alpha), \ldots, r_{k}(\alpha)$ to define a deformation retraction of $|\FP(S)|$ onto $|\Star(v)|$. Using barycentric 
coordinates\footnote{Here, we use the same order on the vertices of $\FP(S)$ as in the definition of the $r_{i}(\alpha)$. If a vertex~$c_{j}$ appears more than once in $r_{i}(\alpha)$, adding up the corresponding entries of a given tuple $T$ yields the barycentric coordinate of the point we refer to.},
any point on the realization of the $p$-simplex $\alpha = \langle w_{0}, \ldots, w_{p} \rangle$ can be identified with a tuple $T = (t_0, \ldots, t_p)$, where $t_j\geq 0$ and $\sum_{i=0}^{p} t_i = 1$.
For $i=0,\ldots, p$, let $k_i = |w_i^\circ \cap v^\circ| = w_i.v$, where the second equality is due to the choice of the~$w_{i}$.
Given a tuple $T$ and $i\in\{1, \ldots, k\}$, we assign to~$g_i$ the weight $\omega_i(T) = t_{j_i}/k_{j_i}$ if $k_{j_{i}}>0$, and $\omega_{j}(T) = 0$ else, such that $\sum_{j=1}^{k} \omega_j(T) = 1$.

For fixed $\alpha$ and $T$, we define $f\colon I \to |\FP(S)|$ by
$$
f_{\alpha}^{T}(s) = [r_i(\alpha), (x_0, \ldots, x_{p+1})]
$$
for $\sum_{j=1}^{i-1} \omega_j(T)\leq s \leq \sum_{j=1}^{i} \omega_j(T)$, where $i \in\{ 1, \ldots, k\}$.
Here, we set $x_l = t_l$ for all $l$, except for the pair
$$
(x_{j_i}, x_{p+1}) = (t_{j_i} - k_{j_i} (s-\sum_{j=1}^{i-1}\omega_j), k_{j_i}(s-\sum_{j=1}^{i-1}\omega_j )).
$$

The map $f^{T}_{\alpha}$ is well-defined:
\begin{align*}
f^{T}_{\alpha}\left(\sum_{j=1}^{i} \omega_j\right) &= [r_{i+1}(\alpha), (t_0, \ldots,t_p, 0)] = [r_{i}(\alpha), (t_0, \ldots, t_{j_{i} -1}, 0, t_{j_{i} + 1}, \ldots, t_p, t_{j_{i}})].
\end{align*}

By construction, $f_{\alpha}^{T}(1)$ lies in $\alpha^{(k+1)}\in \Star(v)$. 
We may now patch the maps $f^{T}_{\alpha}$ for all simplices $\alpha$ and coordinates $T$ with only non-zero entries in order to obtain a global homotopy $f\colon I \times |\FP(S)| \to |\FP(S)|$ with image in $\Star(v)$. 

We still need to prove that $f$ is continuous. By \cite[Thm.~3.1.15]{MR0210112}, we only have to show continuity for the restriction of $f$ to the geometric realization of any simplex.

In the interior of the realization of $\alpha = \langle w_{0}, \ldots, w_{p}\rangle$, continuity follows from the definition of the $\omega_i(T)$. It remains to show that we may go to a subsimplex of $\alpha$ continuously.
That is, for $\beta=\langle w_0, \ldots, w_{p-1}\rangle$,
we must show that for all $s\in I$,
$
f_{\alpha}^{(t_0, \ldots, t_{p-1},0)}(s) = f_{\beta}^{(t_0, \ldots, t_{p-1})}(s)
$. 

This follows from Lemma~\ref{intersection_lemma}:
The number of intersections of $w_p$ and~$v$ does not depend on the simplex $\alpha$; in other words, we have $v.\beta = v.\alpha - v.w_p$. Thus, going to $\beta$ corresponds exactly to $t_{p}$ and any corresponding weight going to zero. We can therefore pass from $\alpha$ to any facet $\beta$ of $\alpha$ continuously. For an arbitrary subsimplex, the claim follows inductively.
\end{proof}

\begin{theorem}\label{delta-conn}
For the connectivity of plant complexes, we have:
\begin{enumerate}[(i)]
\item If $\min\XI>0$, $\conn |\cO(S)| \geq  \min_{i=1, \ldots, t} \left\lfloor\frac{\delta_{i}}{2 \xi_{i}}  \right\rfloor - 2$. \label{colored-conn}
\item If $\min\XI>0$, $\conn |\cP(S)|\geq \left\lfloor \frac{\min\DEL}{2 \max \XI} \right\rfloor -2$. \label{plant-conn}
\item
Let $r>0$, $t\geq m> 0$. If $\DEL = (r, \ldots, r)\in\N^{t}$ and $\XI = (\overbrace{r, \ldots, r}^{m \text{ times}}, 0, \ldots, 0)\in\N_{0}^{t}$, $ \conn|\cP(S)| \geq \left\lfloor \frac{t}{2m-1}\right\rfloor -2$.\label{multi-conn}
\end{enumerate}
\end{theorem}

\begin{remark}
For $m=1$, the complex from Theorem~\ref{delta-conn}(\ref{multi-conn}) is the \emph{fern complex} from~\cite{1410.0923}. By the same article, the fern complex is at least $(t-2)$-connected. Our connectivity result generalizes this bound to the case $m>1$ which we call the \emph{multifern} case.
\end{remark}

If $\gamma$ is a collection of arcs in $(S,\Delta)$ which only meet at $*$, we write $S_\gamma$ for the connected space $(S\setminus\gamma)\cup\{*\}$. We can define (colored) plant complexes on $S_{\gamma}$ accordingly.
In particular, the arguments from Proposition~\ref{contractible} carry over to $S=S_\gamma$, so spaces of the form $|\FP(S_\gamma)|$ ($|\FO(S_{\gamma})|$) are contractible.

\begin{proof}[Proof of Theorem~\ref{delta-conn}]
We prove the proposition for a surface with boundary $S$ or a space of the form $S_\gamma$. The proof is performed in detail for part~(\ref{colored-conn}), which is the result needed in subsequent sections. Some remarks on the proofs on the other two parts are included below.

\textbf{Claim (\ref{colored-conn})} is proved by induction on $\min\DEL$, with $\XI$ fixed. The claim is void if there is an $i\in\{1, \ldots, t\}$ such that $\delta_{i} < 2\xi_{i}$, and we assume for the proof that $\delta_{i} \geq \xi_{i}$ for all $i$.

Let now $k\leq \min_{i=1, \ldots, t} \left\lfloor\frac{\delta_{i}}{2 \xi_{i}}  \right\rfloor - 2$, and consider a map
$
f\colon S^k \to |\cO(S)|.
$
We have to show that $f$ factors through a $(k+1)$-disk.
By the contractibility of $|\FO(S)|$, we have a commutative diagram:

$$
\xymatrix{
S^k \ar@{^{(}->}[d] \ar[r]^f &|\cO(S)| \ar@{^{(}->}[d]\\
D^{k+1} \ar[r]^{\hat f} &|\FO(S)|}
$$

By simplicial approximation, we may assume that all maps are simplicial. That is, they are the geometric realization of simplicial maps $\mathcal{F} \colon \mathcal{S}^{k}\to \cO(S)$ and $\hat{\mathcal{F}} \colon \mathcal{D}^{k+1}\to \FO(S)$ for some finite PL triangulations $\mathcal{S}^{k}$ and $\mathcal{D}^{k+1}$ of the $k$-sphere and the $(k+1)$-disk, respectively. It suffices to show that $\mathcal{F}$ factors through $\mathcal{D}^{k+1}$.

Our goal is to deform $\hat{ \mathcal{F}}$ such that its image lies entirely in $\cP(S)$.
We call a simplex~$\alpha$ of~$\mathcal{D}^{k+1}$ \emph{bad} if in each plant in $\hat{\mathcal{F}}(\alpha)$, there is at least one arc that shares an endpoint with an arc from another plant in $\hat{\mathcal{F}} (\alpha)$ (note that vertices are good).
In particular, a simplex of $\mathcal{D}^{k+1}$ with image in $\cP(S)$ cannot contain any bad subsimplices.

Let $\alpha$ be a bad simplex of $\mathcal{D}^{k+1}$ of maximal dimension $p \leq k+1$ among all bad simplices. Now, $\hat{\mathcal{F}}$ restricts to a map
$$
\hat{\mathcal{F}}|_{\Link(\alpha)}\colon \Link(\alpha) \to J_\alpha \coloneqq  \mathcal{O}_{\DEL'}^{\XI}(S_{\hat{\mathcal{F}}(\alpha)}),
$$
where $\DEL'$ is obtained from $\DEL$ by removing the endpoints of the arcs in $\alpha$ from an instance of $\DEL$. We still need to argue why the image of $\Link(\alpha)$ lies in
$\mathcal{O}_{\DEL'}^{\XI}(S_{\hat{\mathcal{F}}(\alpha)})$: If it did not, there would be a bad simplex $\beta\in\Link(\alpha)$, hence $\alpha*\beta$ would be bad, contradicting the maximality of the dimension of $\alpha$
(note that $\alpha$ and $\beta$ are joinable as $\beta$ is in $\Link(\alpha)$).

For all $i = 1, \ldots, t$, any $p$-simplex uses at most $(p+1)\cdot \xi_{i}$ endpoints of $\Delta_i$, so
\begin{equation}\label{ineq}
\delta_{i}'\geq \delta_{i} - (p+1)\cdot\xi_{i}.
\end{equation}
Furthermore, we have $p\leq k+1 \leq \min_{j=1, \ldots, t} \lfloor\frac{\delta_{j}}{2 \xi_{j}}  \rfloor - 1$, so we obtain from~(\ref{ineq}) and the assumption $\delta_{i} \geq \xi_{i}$:
\begin{align*}
\delta'_{i} &\geq \delta_{i} - (p+1)\cdot \xi_{i} \\
&\geq \delta_{i} - \xi_{i} \cdot \min_{j=1, \ldots, t} \left\lfloor\frac{\delta_{j}}{2 \xi_{j}}  \right \rfloor\\
& \geq \delta_{i} - \xi_{i} \left\lfloor \frac{\delta_{i}}{2\xi_{i}} \right\rfloor \\
& \geq \xi_{i}
\end{align*}
Here, the last inequality follows from the fact that for $a\geq b > 0$, the inequality
$
a - b\cdot \left\lfloor \frac{a}{2b} \right\rfloor \geq b
$
holds.
From the assumption $\min\XI>0$, we get $\min\DEL'<\min\DEL$. Therefore, the induction hypothesis is applicable to $J_\alpha = \mathcal{O}_{\DEL'}^{\XI}(S_{\hat{\mathcal{F}}(\alpha)})$:
\begin{align*}
\conn J_{\alpha} &\geq \min_{i=1, \ldots, t} \left\lfloor\frac{\delta'_{i}}{\xi_{i}}  \right\rfloor - 2 \\
&\geq \min_{i=1, \ldots, t} \left\lfloor\frac{\delta_{i} - (p+1)\cdot \xi_{i}}{2\xi_{i}}-2  \right\rfloor \\
&= \left\lfloor \min_{i=1, \ldots, t}\left( \frac{\delta_{i}}{2 \xi_{i}} \right) -2- \frac{p+1}{2}\right\rfloor\\
&\geq k-p,\end{align*}
since $p\geq 1$.

The rest is \emph{standard machinery}, cf. also the end of the proof of~\cite[Thm.~4.3]{MR3135444}:
By the above connectivity bound for $J_\alpha$, as the link of $\alpha$ is a $(k+1)-p-1 = (k-p)$-sphere, there is a commutative diagram
$$
\xymatrix{
\Link(\alpha) \ar[r]^{\hat {\mathcal{F}}|_{\Link(\alpha)}} \ar@{^{(}->}[d] &J_\alpha \ar[r] &\cO(S)\\
K \ar[ru]_{\hat f'}&&
}
$$
with $K$ a $(k-p+1)$-disk with boundary $\partial K = \Link(\alpha)$. The right map identifies plants on $S_\alpha$ with plants on $S$.
Now, in the triangulation $\mathcal{D}^{k+1}$, replace the $(k+1)$-disk $\Star(\alpha) = \alpha*\Link(\alpha)$ with the $(k+1)$-disk $\partial\alpha*K$.
This works because both $\Star(\alpha)$ and $\partial\alpha*K$ have the same boundary $\partial\alpha*\Link(\alpha)$. 
On $\partial\alpha*K$, modify $\hat{\mathcal{F}}$ by
$$
\hat {\mathcal{F}} * \hat {\mathcal{F}}'\colon \partial\alpha*K \to \FO(S).
$$
This is possible since $\hat {\mathcal{F}}'$ agrees with $\hat {\mathcal{F}}$ on $\Link(\alpha) = \partial K$.

New simplices in $\partial \alpha*K$ are of the form $\tau = \beta_{1}*\beta_{2}$, where $\beta_{1}$ is a proper face of $\alpha$ and $\beta_{2}$ is mapped to $J_\alpha$.
Therefore, if $\tau$ is a bad simplex in $\partial\alpha*K$, then $\tau=\beta_{1}$ since plants of $\hat {\mathcal{F}}'(\beta_{2})$ do not share any endpoints with other plants
of $\hat {\mathcal{F}}'(\beta_{2})$ or $\hat {\mathcal{F}}'(\beta_{1})$, so they cannot contribute to a bad simplex. But $\beta_{1}$ is a proper face of $\alpha$, so we have decreased the
number of top dimensional bad simplices. By induction on the number of top dimensional bad simplices, the result follows.

The proof of \textbf{claim~(\ref{plant-conn})} is widely analogous to the proof of claim~(\ref{colored-conn}), replacing $\delta_{i}$ by $\min\DEL$ and $\xi_{i}$ by $\max\XI$ where necessary.

\textbf{Claim~(\ref{multi-conn})} is proved by induction on $t$, for fixed $m$ and $r$. The multifern complex $\cP(S)$ is always non-empty, so the base case $t=m$ is trivial, as are all cases with $t< 4m-2$. We assume for the induction that $t\geq 2m$ holds.

We argue as in part~(\ref{colored-conn}): Let $\alpha$ be a bad simplex of $\mathcal{D}^{k+1}$ of maximal dimension $p$. An arbitrary $p$-simplex in $\FP$ uses at most $(p+1) m$ different $\Delta_i$.
Since $\alpha$ is bad, it uses at most $(p+1)(m-1) + \lfloor \frac{p+1}{2}\rfloor$ different $\Delta_i$.

 Let $\DEL'$ be the remainder of $\DEL$ after removing the arcs of $\alpha$, and $t'$ be the number of positive entries of~$\DEL'$. Recall that by the definition of multiferns, a multifern $\beta$ having endpoints at $\Delta_i$ necessarily implies that $\Delta_i$ disappears in $S_\beta$. Then, we have
\begin{align}
t' &\geq t - \left((p+1)(m-1) + \left\lfloor \frac{p+1}{2}\right\rfloor \right) \label{n-eqn}\\
&= t - \left\lfloor (p+1)\left(m-\frac{1}{2}\right) \right\rfloor\nonumber  \\
&\geq t -  \left\lfloor \left\lfloor\frac{t}{2m-1}\right\rfloor \left(m-\frac{1}{2}\right)\right\rfloor \label{eqn_x} \\
&\geq t - \left\lfloor\frac{t}{2}\right\rfloor\nonumber \\
 &\geq m. \label{eqn_y}
\end{align}
Here, (\ref{eqn_x}) is due to the fact that $p+1 \leq k+2 \leq \left\lfloor \frac{t}{2m-1}\right\rfloor$, and~(\ref{eqn_y}) is true since we demanded that $t\geq 2m$.
Consequently, we can apply the induction hypothesis to $J_\alpha$. Using~(\ref{n-eqn}), we obtain
\begin{align*}
 \conn(J_\alpha) &\geq \left\lfloor \frac{t'}{2m-1}\right\rfloor -2 \\
 &\geq \left\lfloor \frac{t-(p+1)(m-1) - \left\lfloor \frac{p+1}{2}\right\rfloor }{2m-1} - 2\right\rfloor \\
 &\geq \left\lfloor k -  \frac{\left\lfloor (p+1)\left(m-\frac{1}{2}\right) \right\rfloor}{2m-1}\right\rfloor \\
 &\geq \left\lfloor k - \frac{p+1}{2} \right\rfloor \\
 &\geq k-p,
\end{align*}
as $p\geq 1$ (there are no bad vertices).

The rest of the proof can be copied from above (\emph{standard machinery}).
\end{proof}

%%%
\section{Combinatorics of colored plant complexes}\label{combinatorics}

In this section, we focus on a specific class of colored plant complexes on a closed disk~$D$. Let $\XI$ be a $t$-tuple of positive integers and $n\in\N$. We consider colored plant complexes of the form
$
\plant \coloneqq \mathcal{O}_{n\cdot\XI}^{\XI}(D)
$
and write $\plantq$ for the set of $q$-simplices of $\plant$.
Because of the specific constellation $\DEL = n\cdot\XI$, it is clear that the dimension of $\plant$ equals $n-1$. 

After applying a suitable homeomorphism, we may assume that $D$ lies in the complex plane as a disk of radius $1$ centered at $0\in\C$, that we have $* = -\mathrm{i}$, and that the $n\xi$ points of $\Delta$ are all real and arranged from left to right in $n$~\emph{clusters} of $\xi$ points each, where $\xi_{i}$ points in each cluster lie in $\Delta_{i}$, for $i=1, \ldots, t$. For each of these clusters, we suppose that the points in $\Delta_{i}$ are placed to the left of the points in $\Delta_{j}$, for $1\leq i<j\leq t$.

\subsection*{The braid action}

We can identify the full braid group $\Br_{n\xi}$ with the mapping class group $\Map(D\setminus \Delta)$, where the standard generator~$\sigma_{i}$ corresponds to a half twist in counterclockwise direction which interchanges the $i$-th and the $(i+1)$-th marked point. The colored braid group $\Br_{n\cdot\XI} \subset \Br_{n\xi}$ may then be identified with the set of mapping classes which leave the given partition of $\Delta$ invariant. This gives a well-defined left action of $\Br_{n\cdot\XI}$ on $\plantq$ by isotopy classes of orientation-preserving diffeomorphisms (or homeomorphisms, which is equivalent) for all $q = 0, \ldots, n-1$. We write $\alpha \mapsto \sigma\cdot\alpha$ for the action of $\sigma\in\Br_{n\cdot\XI}$ on a simplex $\alpha\in\plant$.

Let $\alpha\in\plantq$ be a $q$-simplex. The sorting of the inward pointing unit tangent vectors of representative arcs of $\alpha$ at $*$ in clockwise order is well-defined, cf.~Remark~\ref{plant-order}. This ordering is invariant under the $\Br_{n\cdot\XI}$-action.

We assign an \emph{index}~$i\in\{0, \ldots, q\}$ and a \emph{color}~$j\in\{1, \ldots, t\}$ to each arc in $\alpha$:
\begin{itemize}
\item 
An arc is labeled with the index $i$ if it belongs to the $(i+1)$-th plant in $\alpha$, where we use the order on the plants of $\alpha$ described in Remark~\ref{plant-order}.
\item
An arc is labeled with the color $j$ if its endpoint lies in $\Delta_{j}$.
\end{itemize}
Consequently, any $q$-simplex $\alpha$ defines a unique sequence
\begin{equation}\label{ic-sequence}
\omega_{\alpha} = (i_{1}, j_{1}), (i_{2}, j_{2}), \ldots, (i_{(q+1)\cdot\xi}, j_{(q+1)\cdot\xi}),
\end{equation}
where $i_{k}$ is the index and $j_{k}$ the color of the $k$-th arc of $\alpha$, in clockwise order at $*$.
By definition of the mapping class group, this sequence is $\Br_{n\cdot\XI}$-invariant.

\begin{figure}
\centering
\captionsetup[subfigure]{labelformat=empty}
\begin{subfigure}[b]{0.49\linewidth}
\centering
\includegraphics[scale=0.2]{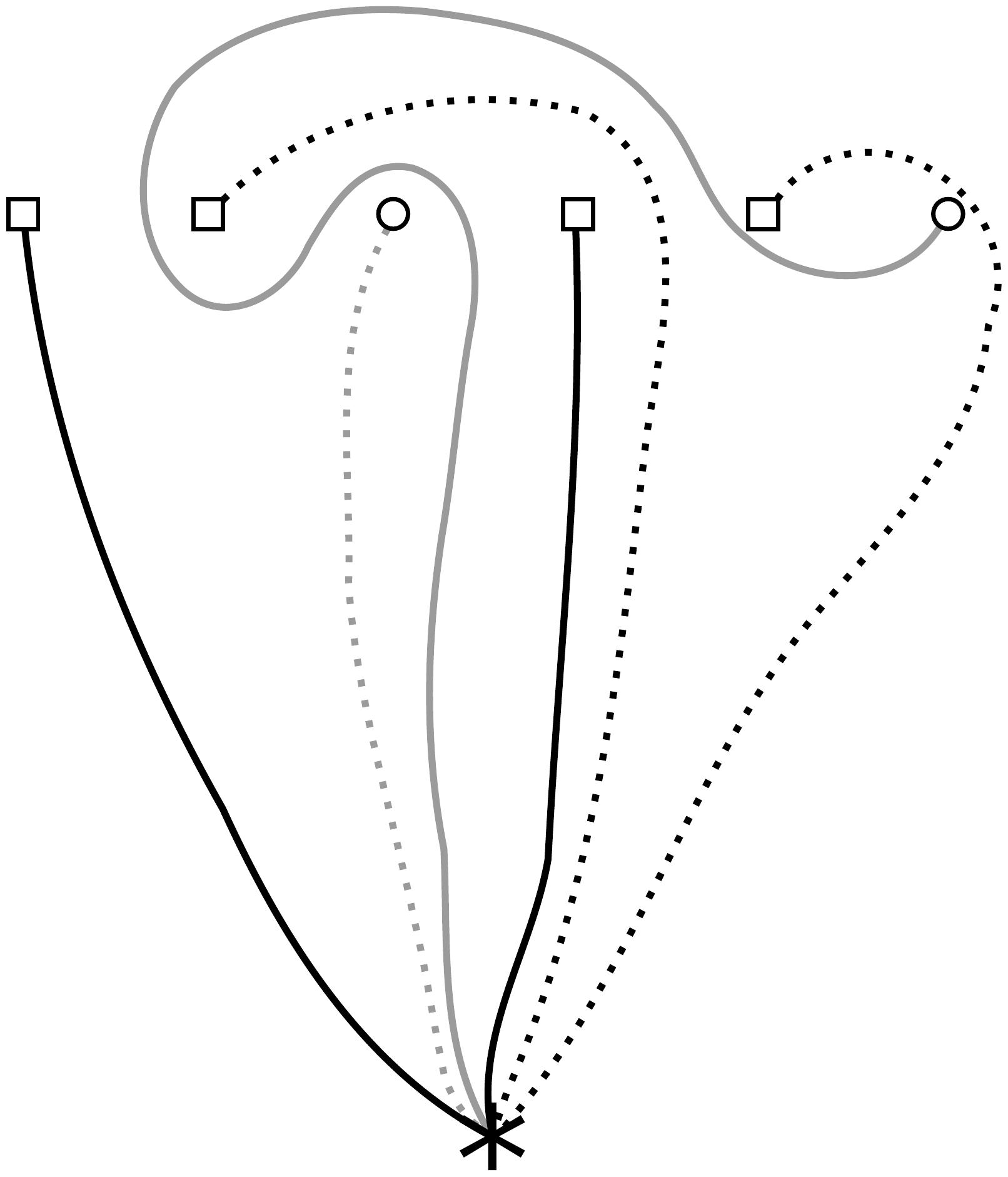}
\caption{$\omega_{\alpha} = (0,1),(1,2),(0,2),(0,1),(1,1),(1,1) $}
\end{subfigure}
\begin{subfigure}[b]{0.49\linewidth}
\centering
\includegraphics[scale=0.2]{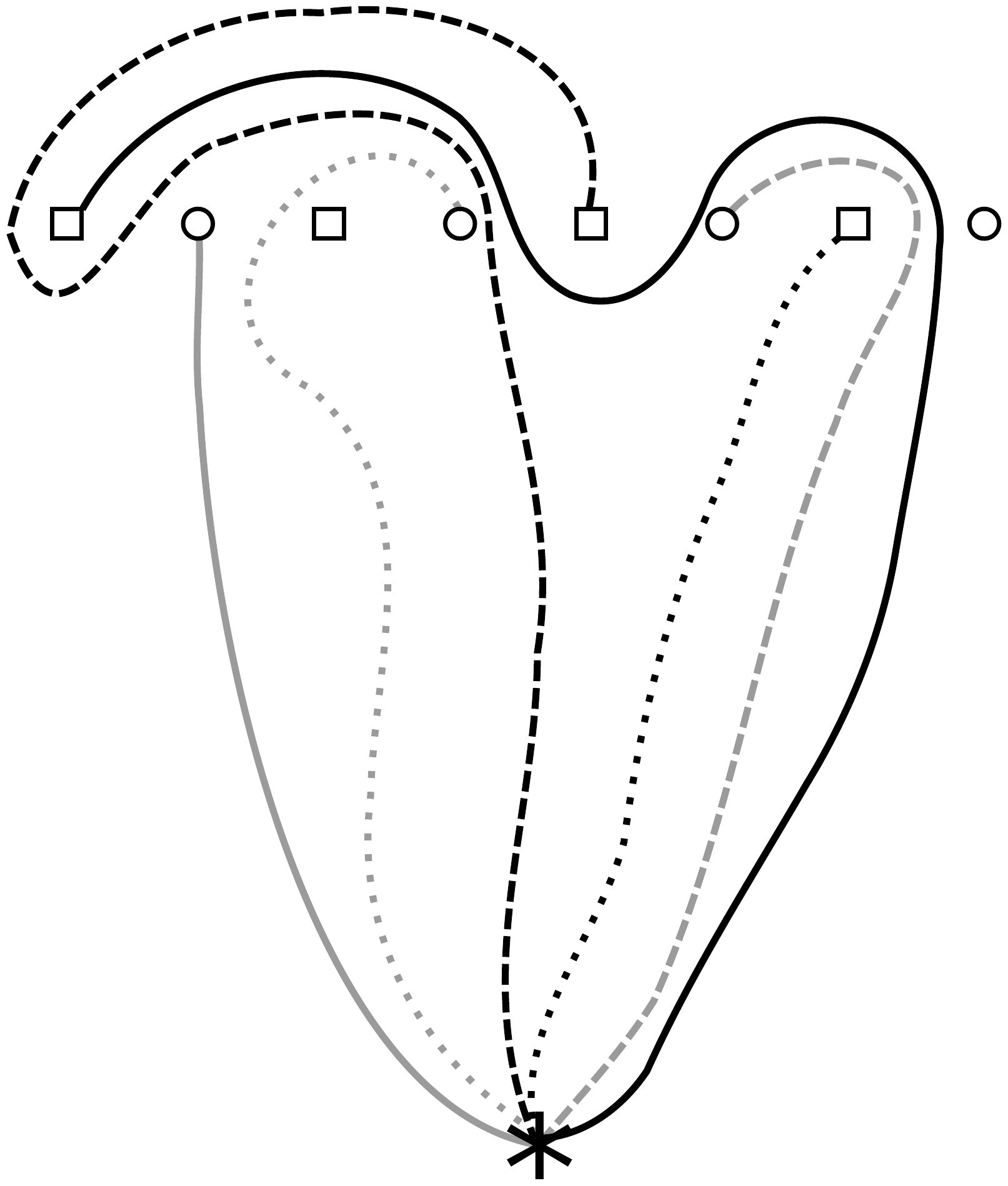}
\caption{$\omega_{\beta} = (0,2),(1,2),(2,1),(1,1),(2,2),(0,1)$}
\end{subfigure}
\caption{IC-sequences of simplices $\alpha\in\mathcal{O}^{[2,(2,1)]}_{1}$ and $\beta\in\mathcal{O}^{[4,(1,1)]}_{2}$.}
\label{fig-sequences}
\end{figure}

\begin{definition}
We call~(\ref{ic-sequence}) the \emph{IC-sequence} (index-color-sequence) of $\alpha\in\plantq$.
\end{definition}

With the help of IC-sequences, we are able to count the orbits of the $\Br_{n\cdot\XI}$-action:

\begin{lemma}\label{orbits}
Let $q < n$.
The set $\plantq$ decomposes into
$$
l_{q}^{\XI} \coloneqq \#\left\{\plantq / \Br_{n\cdot \XI} \right\} = \frac{(\xi(q+1))!}{(q+1)!\cdot(\xi_{1}!\cdot\ldots\cdot\xi_{t}!)^{q+1}}
$$
$\Br_{n\cdot\XI}$-orbits. The orbits are in bijective correspondence to the set of occurring IC-sequences for $q$-simplices.
\end{lemma}

\begin{remark}
A priori, the number $l_{q}^{\XI}$ depends not only on $q$ and $\XI$ but also on $n$. As a consequence of the lemma, the actual quantity is independent on $n$ as long as $q<n$; therefore, it makes sense to omit $n$ from the notation.
\end{remark}

\begin{proof}
We start by counting the number of possible IC-sequences.
We say a sequence of type~(\ref{ic-sequence}) is \emph{$q$-feasible} if each index $i\in\{0, \ldots, q\}$ is assigned $\xi$ times; for each index, each color $j = \{1, \ldots, t\}$ is assigned $\xi_{j}$ times; and the index $i+1$ does not appear before the index $i$, for all $i = 0, \ldots, q-1$.
Since $D$ is path-connected, a $q$-feasible sequence appears as the IC-sequence of a $q$-simplex if and only if $q< n$, i.e., as long as there are enough endpoints for the arcs.

We may now count the number of $q$-feasible sequences: There are
$$
{\binom{\xi (q+1)}{\underbrace{\xi, \ldots, \xi}_{(q+1)\text{ times}}}}\frac{1}{(q+1)!} = \frac{(\xi(q+1))!}{(q+1)!\cdot(\xi!)^{q+1}}
$$
different partitionings of $\xi (q+1)$ arcs into subsets of size $\xi$. Any such partition gives a unique indexing (recall that the first arc with index $i$ appears before the first arc with index $i+1$).
 
Given an index $i$, the $\xi$ arcs labeled with it can be colored in
$
\binom{\xi}{\xi_{1}, \ldots, \xi_{t}}
$
different ways, which makes it
$$
l_{q}^{\XI} = \frac{(\xi(q+1))!}{(q+1)!\cdot(\xi!)^{q+1}}\cdot \binom{\xi}{\xi_{1}, \ldots, \xi_{t}}^{q+1}
= \frac{(\xi(q+1))!}{(q+1)!\cdot(\xi_{1}!\cdot\ldots\cdot\xi_{t}!)^{q+1}}
$$
different choices of $q$-feasible sequences, as there are $(q+1)$ different indices.

We have already seen above that the IC-sequence of a $q$-simplex is invariant under the $\Br_{n\cdot\XI}$-action. In the second part of the proof, we will now show that the group $\Br_{n\cdot\XI}$ acts transitively on simplices with according IC-sequences by an argument similar to \cite[Prop.~5.6]{0912.0325}. An alternate proof can be obtained by adapting the methods from the proof of \cite[Prop.~2.2(1)]{MR3135444}.

Let $\alpha, \beta$ be two $q$-simplices with the same IC-sequence $\omega$.  We choose representative non-intersecting collections of the plants of $\alpha$ and $\beta$ with arcs $a_{1}, \ldots, a_{(q+1)\xi}$ and $b_{1}, \ldots, b_{(q+1)\xi}$, respectively, subscripts chosen such that the arcs are arranged in clockwise order at $*$. Since $\Br_{n\cdot\XI}$ surjects onto $\mathfrak{S}_{n\cdot\XI}$, we may assume that $a_{i}$ and $b_{i}$ have the same endpoint $a_{i}(1) = b_{i}(1)$ for all $i = 1, \ldots, (q+1)\xi$. Furthermore, after a suitable isotopy, we may as well assume that for some $\epsilon>0$, we have $a_{i}(t) = b_{i}(t)$ for all $0 \leq t \leq \epsilon$. Hence, if we choose a continuous increasing function $h\colon I \to \R$ with $h(t) = t$ for $0 \leq t \leq \epsilon/2$ and $h(1) = \epsilon$, we obtain $a_{i} \circ h = b_{i} \circ h$ for all $i$.

It remains to show that there is an orientation-preserving homeomorphism $G$ of~$D$ which, for all $i = 1, \ldots, (q+1)\xi$, retracts the arc $a_{i}$ to $a_{i} \circ h$, and fixes those marked points which are no endpoints of arcs in $\alpha$. In addition, we construct a similar map $H$ which carries $b_{i}$ to $b_{i} \circ h$ for all $i$. Then, the homeomorphism $H^{-1} \circ G$ defines a mapping class which carries $\alpha$ to $\beta$, and which corresponds to an element in $\Br_{n\cdot\XI}$ because the IC-sequences of $\alpha$ and $\beta$ coincide. To construct $G$, choose disjoint closed tubular neighborhoods $U_{i}$ of $a_{i}|_{[\epsilon/3, 1]}$ for all $i$. Such neighborhoods exist since the arcs are disjoint except at $*$. Now, $U_{i}$ is homeomorphic to a closed disk, and so there exists a homeomorphism which restricts to the identity on $\partial U_{i}$ and which carries the arc segment $U_{i} \cap a_{i}$ to its retraction $U_{i} \cap (a_{i} \circ h)$. Combining these homeomorphisms and extending them by the identity on $D \setminus \bigcup_{i} U_{i}$ yields the desired homeomorphism~$G$. The construction of $H$ is analogous.
\end{proof}

\begin{lemma}\label{stabilizers}
Let $q< n$. For any simplex $\alpha\in\plantq$, the stabilizer of $\alpha$ under the $\Br_{n\cdot\XI}$-action is isomorphic to $\Br_{(n-q-1)\cdot\XI}$. In particular, there is a bijection between the elements of the orbit $\Br_{n\cdot\XI}\cdot\alpha$ and the cosets in $\Br_{n\cdot\XI} / \Br_{(n-q-1)\cdot\XI}$.
\end{lemma}
\begin{proof}
We show that the stabilizer of any simplex $\alpha\in\plantq$ is isomorphic to $\Br_{(n-q-1)\cdot\XI}$. Then, the second assertion follows directly from the orbit-stabilizer-theorem.

Let $\Sigma\subset D$ be the union of a representative set of the arcs of $\alpha$, only intersecting at~$*$. In clockwise order around $*$, denote the arcs in $\Sigma$ by $a_1, \ldots, a_{(q+1)\xi}$. By $\Map(D\setminus \Delta, \Sigma)$, we denote the group of isotopy classes of orientation-preserving diffeomorphisms of $D\setminus \Delta$ which fix $\Sigma$ pointwise. Since $\Sigma$ is contractible, the group $\Map(D\setminus\Delta, \Sigma)$ may be identified with $\Map(D\setminus (\Delta \setminus \Sigma)) \cong \Br_{(n-q-1)\cdot\XI}$, which itself may be identified with a subgroup of $\Br_{n\cdot\XI}$. We will show that the inclusion of subgroups of  $\Map(D\setminus\Delta)$
\begin{equation*}\label{injection-mcg}
\Map(D\setminus (\Delta \setminus \Sigma)) \hookrightarrow \left(\Map\left(D\setminus\Delta\right)\right)_{\alpha}
\end{equation*}
is surjective and hence an isomorphism.

For this part, we follow the similar proof in \cite[Prop.~2.2(2)]{MR3135444}. Choose an element $\phi \in\Diff^+(D\setminus\Delta)$ which stabilizes the simplex $\alpha$. We have to show that $\phi$ is isotopic to a diffeomorphism that fixes $\Sigma$ pointwise.

By definition, $\phi(a_1)$ is isotopic to $a_1$. The isotopy extension theorem \cite{MR0123338} implies that we can extend a corresponding isotopy to an ambient isotopy, so we may assume that $\phi$ fixes $a_1$ pointwise. We proceed by induction on the number of fixed arcs. Let $j>1$, and assume that $\phi$ fixes $\Sigma_{j} = a_1\cup \ldots\cup a_{j-1}$ pointwise. The arc $a_j$ is isotopic to $\phi(a_j)$, and we must show that the corresponding isotopy can be chosen disjointly from~$\Sigma_{j}$. If this holds, another application of the isotopy extension theorem implies the inductive step and thus the statement.

Let $H\colon I \times I \to D$ be a smooth isotopy that conveys $\phi (a_{j})$ to $a_{j}$, and assume that $H$ is transverse to $\Sigma_{j}$, using the transversality theorem \cite{MR0061823}. Here, $H(0,-)$ and $H(1,-)$ correspond to the arcs $\phi(a_{j})$ and $a_{j}$, respectively. Furthermore, we have $H(-,0) = *$, and $H(-,1) \in\Delta$ is the endpoint of $a_{j}$.

Now, consider the preimage $H^{-1}(\Sigma_{j})$. The line $I \times \{0\}$ is the preimage of $*$, and by transversality, all other components must be circles in the interior of $I \times I$.

Since the intersection number is finite, there is at least one such circle which encloses no further circle in $H^{-1}(\Sigma_{j})$. Let $D_{0}$ be the closed disk it encloses. Let furthermore  $\Sigma_{j}^{\delta}\subset D$ be a closed $\delta$-thickening of $\Sigma_{j}$ with $\delta>0$ chosen such that $\Sigma^{\delta}_{j}$ is still contractible. By continuity of $H$, we may now choose $\epsilon>0$ such that for a closed $\epsilon$-neighborhood $D_{0}^{\epsilon}$ of $D_{0}$, we have $H(\partial D_{0}^{\epsilon}) \subset \Sigma^{\delta}_{j}$, 

Restriction of $H$ to the closed disk $D_{0}^{\epsilon}$ defines an element of the relative homotopy group $\pi_{2}(D,\Sigma_{j}^{\delta} \setminus \Sigma_{j})$. This group is trivial. We may thus replace $H$ on $D_{0}^{\epsilon}$ by a homotopic map $H'$ with $H'|_{\partial D_{0}^{\epsilon}} = H|_{\partial D_{0}^{\epsilon}}$ and image in $\Sigma_{j}^{\delta} \setminus \Sigma_{j}$, which exists since $\Sigma_{j}^{\delta} \setminus \Sigma_{j}$ is simply connected.

By extending $H'$ to $I \times I$ by $H'|_{(I \times I)\setminus D_{0}^{\epsilon}} =  H|_{(I \times I)\setminus D_{0}^{\epsilon}}$, we obtain a homotopy $H'$ with $\pi_{0} (H'^{-1}(\Sigma_{j})) < \pi_{0} (H^{-1}(\Sigma_{j}))$.
Inductively, we construct a homotopy $H''$ which is disjoint from $\Sigma_{j}$.
Finally, by \cite[Thm.~3.1]{MR0214087}, $H''$ can be replaced by an isotopy in $(D\setminus \Sigma_{j}) \cup \{*\}$.
\end{proof}

\subsection*{Standard simplices}

As a consequence of Lemma~\ref{orbits} and Lemma~\ref{stabilizers}, there are bijections between the set $\plantq$ of $q$-simplices and the disjoint union of $l_{q}^{\XI}$ copies of $\Br_{n\cdot\XI} / \Br_{(n-q-1)\cdot\XI}$ for all $q= 0,\ldots, n-1$. 
Our next goal is to make these bijections compatible with the semi-simplicial structure on $\plant$: We want to fix bijections and describe the structure of a semi-simplicial set on
$$
\cplant = \bigsqcup_{q=0}^{n-1} \cplantq =  \bigsqcup_{q=0}^{n-1} \bigsqcup_{l_{q}^{\XI} \text{ copies}} \Br_{n\cdot\XI} / \Br_{(n-q-1)\cdot\XI}
$$
which is compatible with the face maps, thus defines a semi-simplicial isomorphism between $\cplant$ and $\plant$.

Let 
$\omega$
be a fixed IC-sequence. In what follows, we define a standard $q$-simplex in $\plantq$ for the IC-sequence $\omega$: 
We resort the terms of $\omega$ by the consecutive sorting criteria \emph{index}~(1st), \emph{color}~(2nd) and \emph{position in the IC-sequence}~(3rd), and draw the arcs of a $q$-simplex in this new order, respecting the order at~$*$ prescribed by the IC-sequence. We draw the arcs such that the endpoint of each arc is chosen as the leftmost free marked point, where we always \emph{undercross} marked points if possible. This is compatible with the coloring because of the arrangement of marked points.

This process generates a set of $(q+1)\cdot\xi$ arcs which is unique up to isotopy since we work on a disk $D$. Distinguishing by the indices, these arcs can be divided into $q+1$ colored $\XI$-plants, which in turn define a $q$-simplex $\alpha_{\omega} \in\plantq$. Using these simplices, every $q$-simplex can be written as $\sigma\cdot\alpha_{\omega}$ for some IC-sequence $\omega$ and some $\sigma\in\Br_{n\cdot\XI}$ as consequence of the transitivity of the $\Br_{n\cdot\XI}$-action on simplices with the same IC-sequence.

\begin{figure}
\centering
\captionsetup[subfigure]{labelformat=empty}
\begin{subfigure}[b]{0.48\linewidth}
\centering
\includegraphics[scale=0.2]{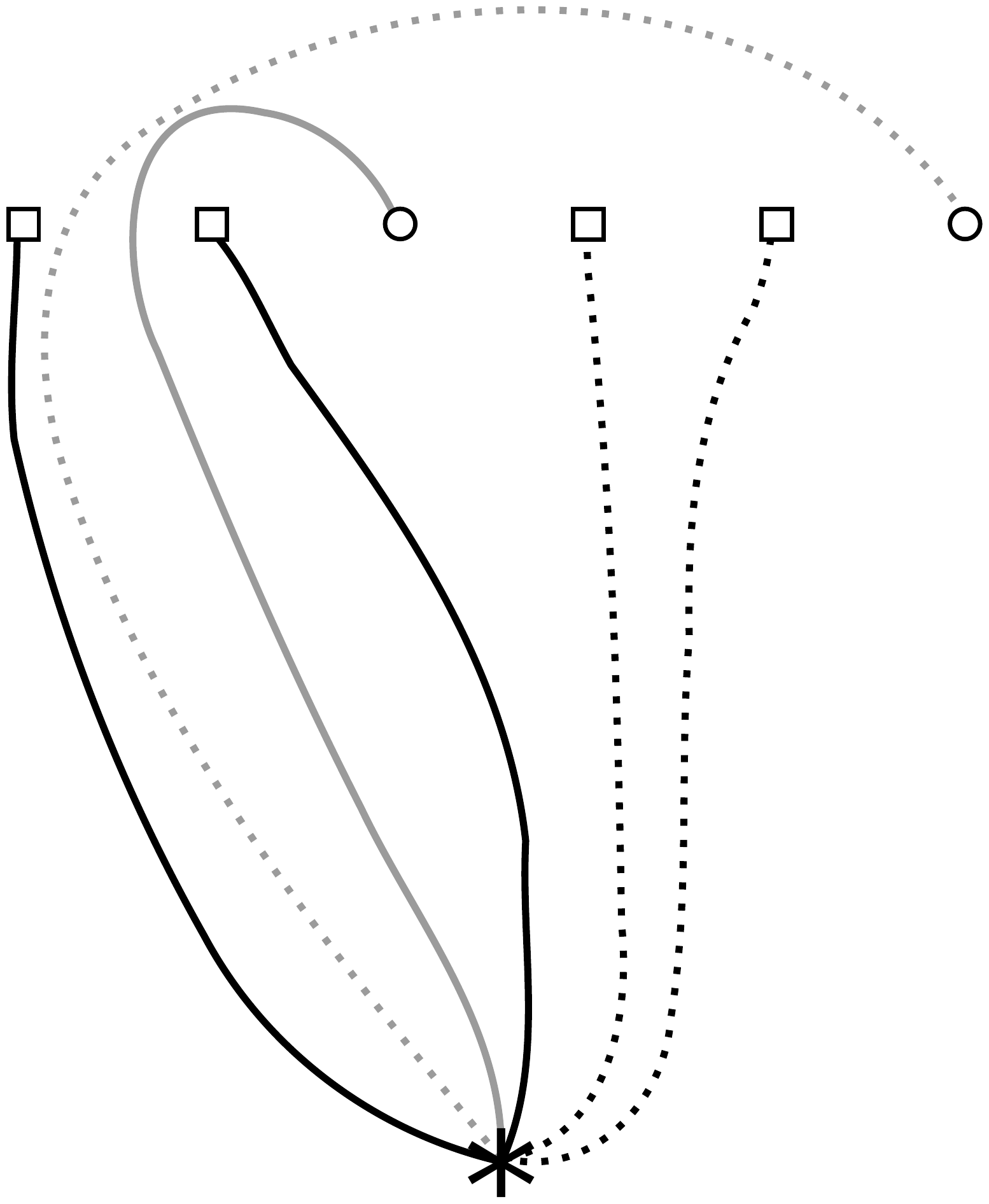}
\caption{$\alpha_{\omega_{\alpha}}\in\mathcal{O}^{[2,(2,1)]}_{1}$}
\end{subfigure}
\begin{subfigure}[b]{0.48\linewidth}
\centering
\includegraphics[scale=0.2]{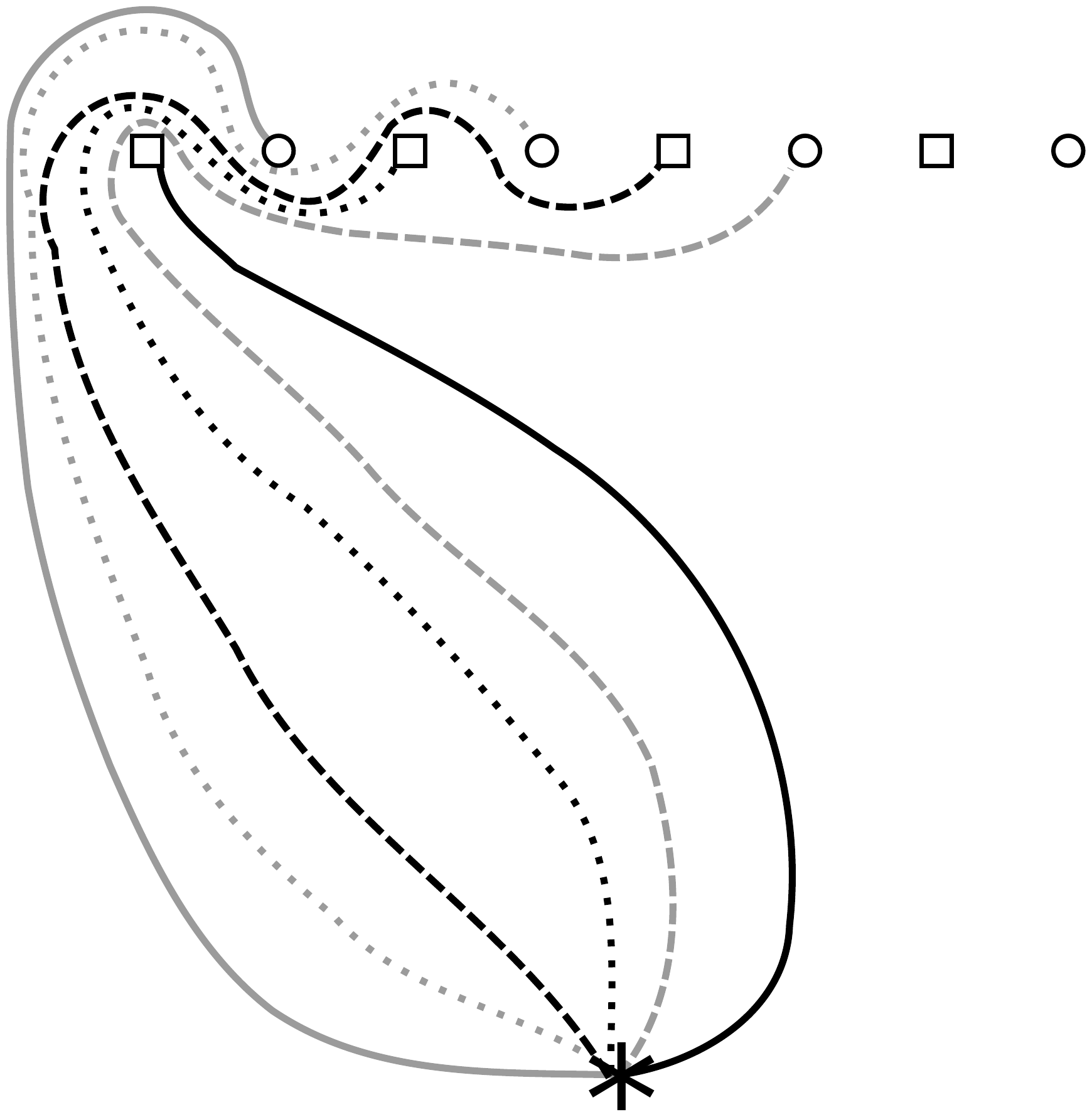}
\caption{$\alpha_{\omega_{\beta}}\in\mathcal{O}^{[4,(1,1)]}_{2}$}
\end{subfigure}
\caption{Standard simplices for the IC-sequences of $\alpha$ and $\beta$ from Figure~\ref{fig-sequences}.}
\label{fig-standard}
\end{figure}

\begin{definition}
The simplex $\alpha_{\omega} \in\plantq$ is called the \emph{standard simplex for the IC-sequence $\omega$}.
\end{definition}

\begin{definition}
Let $\alpha\in\plantq$ be a $q$-simplex, and $\omega$ its IC-sequence. The sequence
$
\widetilde\omega = (p_{1}, p_{2}, \ldots, p_{(q+1)\cdot\xi})
$
induced by the reordering of the IC-sequence of $\alpha$ described above, where $p_{i}$ is the position in the IC-sequence of the corresponding arc, is called the \emph{P-sequence} (position sequence) of $\alpha$. \end{definition}

\begin{example}
The P-sequences of the simplices in Figures~\ref{fig-sequences} and~\ref{fig-standard} are given by
$
\tilde\omega_{\alpha} = (1,4,3,5,6,2)
$
and
$
\tilde\omega_{\beta} = (6,1,4,2,3,5).
$
\end{example}

We identify $\Br_{n\xi}$ with the mapping class group of the $n\xi$-punctured disk in such a way that $\Br_{n\cdot\XI}$ is the stabilizer of the colored configuration of $n\xi$ points in $D$. The element $\sigma_{i\xi + j}$, for $i=0, \ldots, q$ and $j=1, \ldots, \xi-1$, describes the isotopy class of a a half twist that interchanges the $j$-th and the $(j+1)$-th point of the $(i+1)$-th cluster.
On the other hand, the elements of the form $\sigma_{i\xi}$, $i=1, \ldots, q$, describe a half twist that interchanges the $\xi$-th point of the $i$-th cluster with the first point of the $(i+1)$-th cluster. 

We know from Lemma~\ref{stabilizers} that the stabilizer of a $q$-simplex is isomorphic to the group $\Br_{(n-q-1)\cdot\XI}$. For any standard $q$-simplex $\alpha_{\omega}$, we may thus write
\begin{align*}
(\Br_{n\cdot\XI})_{\alpha_{\omega}}&= 
\left\langle
\sigma_{k}
\mid
(q+1)\cdot\xi+1 \leq  k \leq n \xi-1
\right\rangle \cap \Br_{n\cdot\XI}.
\end{align*}
As this expression is independent of $\omega$, we may denote the stabilizer of \emph{any} standard $q$-simplex by
$$
L_{q} = (\Br_{n\cdot\XI})_{\alpha_{\omega}} \cong \Br_{(n-q-1)\cdot\XI}.
$$

Now, once and for all, we fix the bijection
\begin{align*}
\Gamma_{\omega}\colon \Br_{n\cdot\XI}/L_{q}  &\to \Br_{n\cdot\XI}\cdot\alpha_{\omega} \\
\sigma L_{q} &\mapsto \sigma\cdot\alpha_{\omega}
\end{align*}
for each $\Br_{n\cdot\XI}$-orbit in $\plant$. Collecting these maps for all IC-sequences, we obtain a global bijection
\begin{align*}
\Gamma\colon\cplant&\to\plant\\
(\omega_{p}, \sigma L_{p}) &\mapsto \sigma\cdot\alpha_{\omega_{p}} \:\:\:\text{for all }p\geq 0,
\end{align*}
where $\omega_{p}$ is an IC-sequence of a $p$-simplex.

\subsection*{Face maps}

Recall that the colored plants in a $q$-simplex $\alpha = \langle v_{0}, \ldots, v_{q} \rangle \in\plantq$ are ordered by the tangential direction at $*$ of their respective leftmost arcs. For $i=0,\ldots,q$, the $i$-th \emph{face map} is given by leaving out the vertex $v_{i}$:
\begin{align*}
\partial_{i}\colon \plantq &\to \plantqi \\
\langle v_{0}, \ldots, v_{q} \rangle &\mapsto \langle v_{0}, \ldots, \hat{v_{i}}, \ldots v_{q}\rangle.
\end{align*}

We now determine \emph{face maps} in $\cplantq$ for all $q\geq 0$ which are compatible with the face maps in $\plant$ insofar as they give $\cplant$ the structure of a semi-simplicial set isomorphic to $\plant$. These maps are evidently given by $\Gamma^{-1}\circ \partial_{i} \circ \Gamma$. For later use, we need to describe them explicitly.

\begin{figure}
\centering
\captionsetup[subfigure]{labelformat=empty}
\begin{subfigure}[b]{0.32\linewidth}
\centering
\includegraphics[scale=0.2]{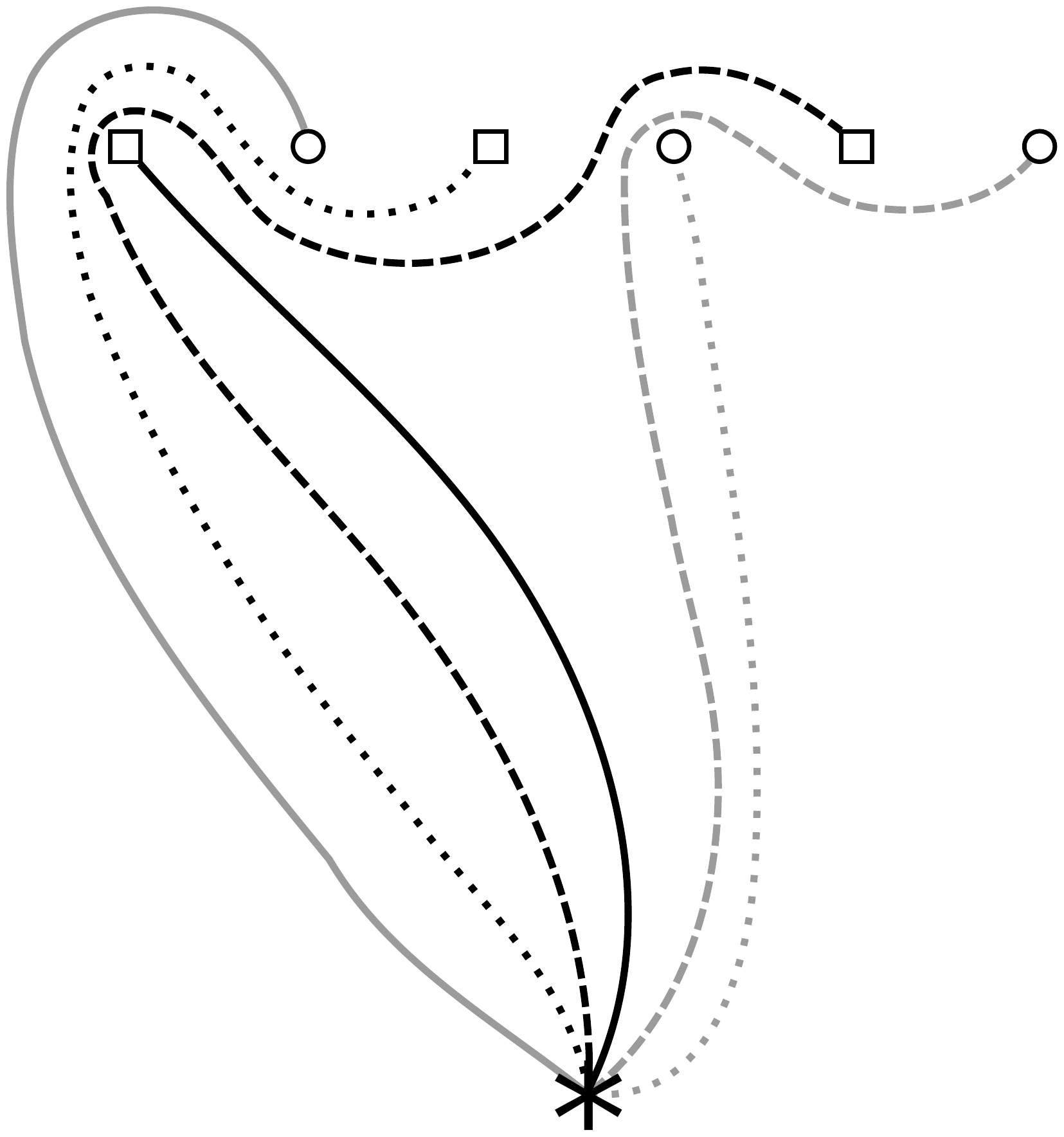}
\caption{$\alpha_{\omega}$}
\end{subfigure}
\begin{subfigure}[b]{0.32\linewidth}
\centering
\includegraphics[scale=0.2]{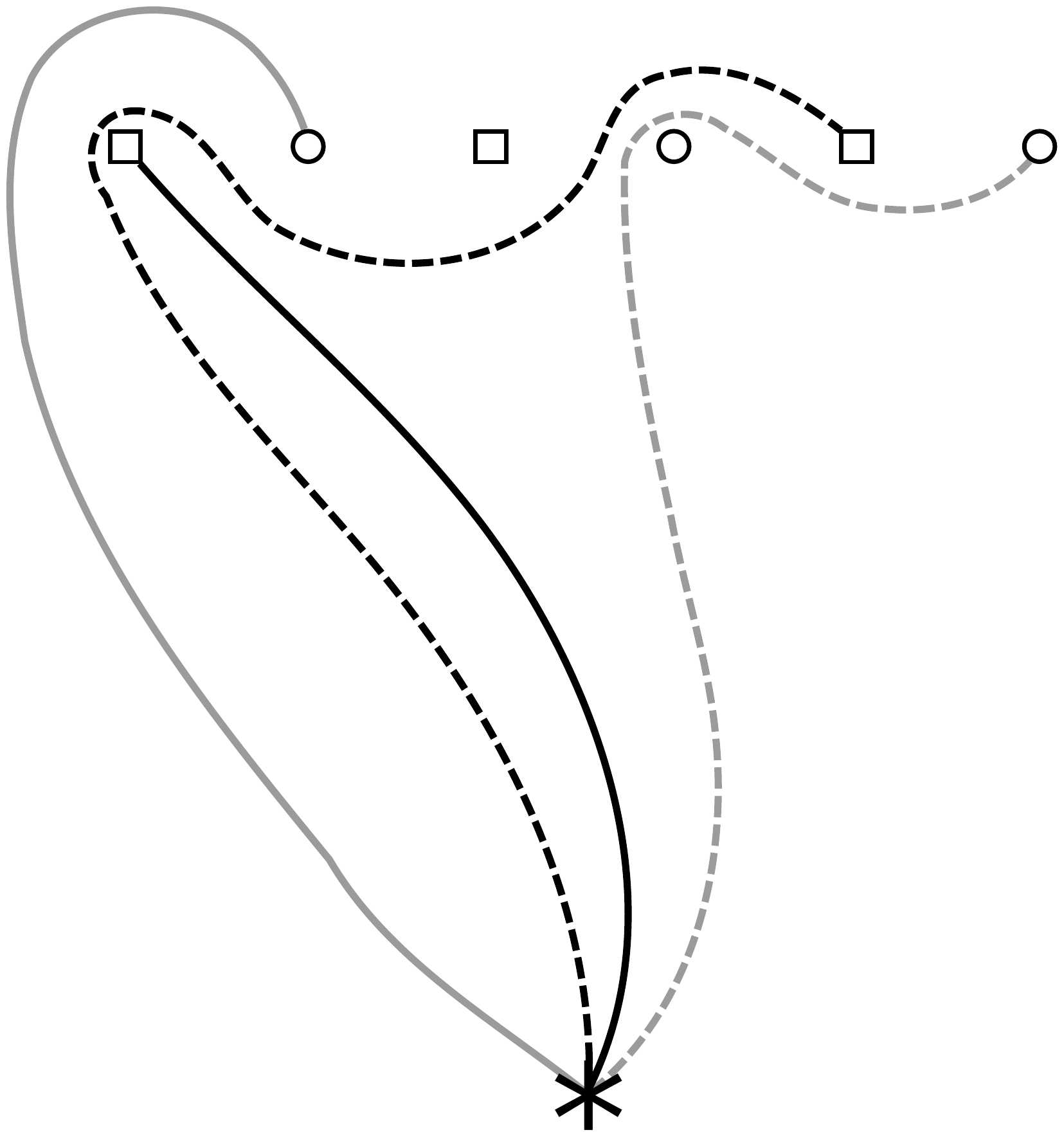}
\caption{$\partial_{1}\alpha_{\omega}$}
\end{subfigure}
\begin{subfigure}[b]{0.32\linewidth}
\centering
\includegraphics[scale=0.2]{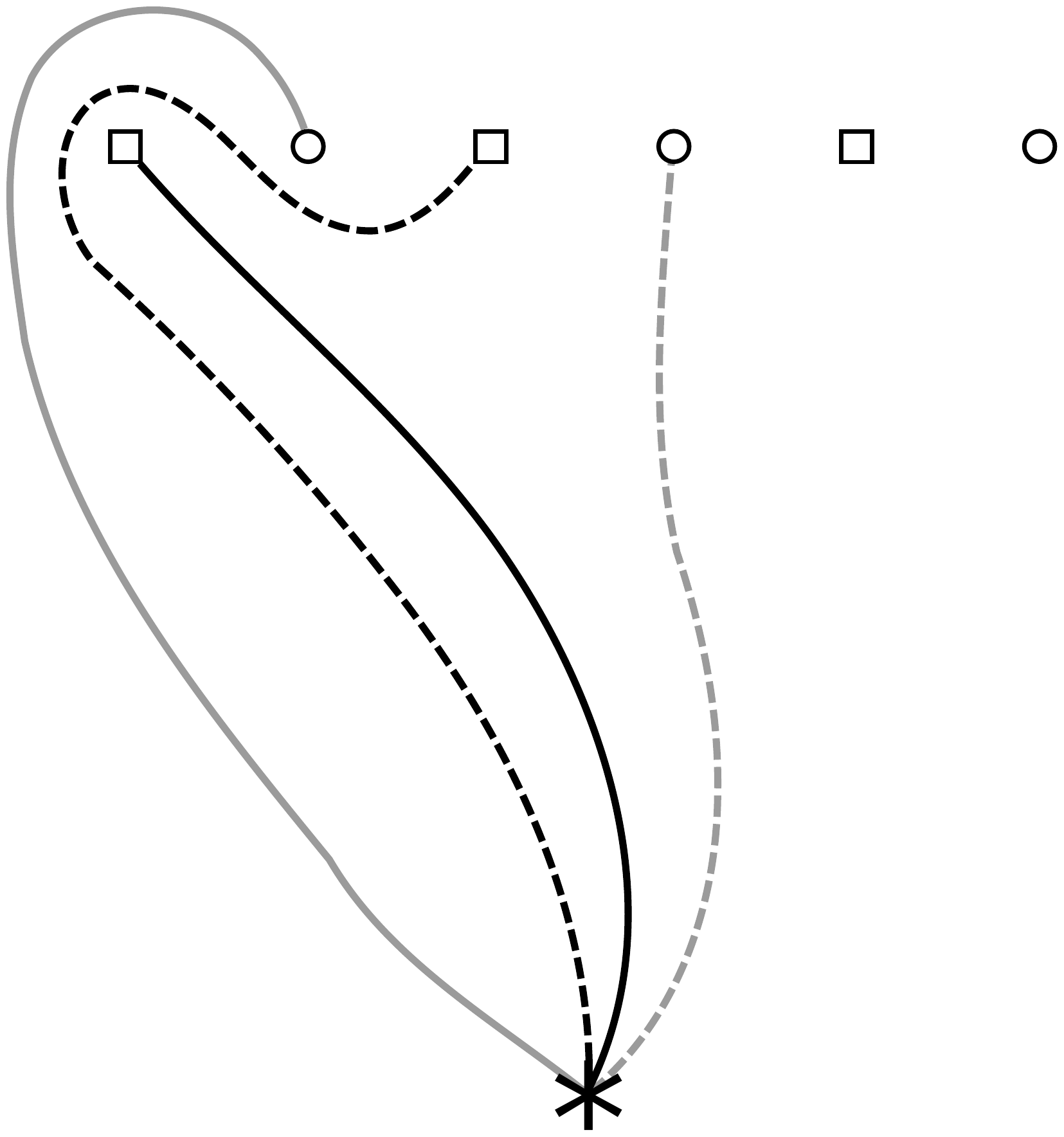}
\caption{$\alpha_{d_{1}\omega}$}
\end{subfigure}
\caption{The simplices $\alpha_{\omega}$, $\partial_{1}\alpha_{\omega}$, $\alpha_{d_{1}\omega}\in\mathcal{O}^{[3,(1,1)]}$ for $\omega$ from Example~\ref{ex-faces}.}
\label{fig-faces}
\end{figure}

Given a $q$-simplex with IC-sequence $\omega$, its $i$-th face map is given by removing the arcs with index $i$. Hence, we may define the $i$-th \emph{face} of $\omega$ as the IC-sequence $d_{i}\omega$ obtained by first removing all the pairs with index $i$ from $\omega$, and
secondly subtracting $1$ from the indices of the remaining elements with indices bigger than~$i$. The IC-sequence $d_{i} \omega$ defines a P-sequence which we denote by $d_{i} \tilde\omega$.

\begin{example}\label{ex-faces}	
Consider the IC-sequence $\omega = (0,2),(1,1),(2,1),(0,1),(2,2),(1,2)$ for a simplex in $\mathcal{O}^{[3,(1,1)]}$. The corresponding P-sequence is $\tilde\omega = (4,1,2,6,3,6)$, and the first faces of the sequences are given by
$d_{1}\omega = (0,2),(1,1),(0,1),(1,2)$ and $d_{1}\tilde\omega = (3,1,2,4)$. The inherent standard simplices are depicted in Figure~\ref{fig-faces}.
\end{example}

Our next goal is to find elements $\tau_{i,q}^{\omega}\in\Br_{n\cdot\XI}$ for all $i = 0, \ldots, q$, such that
\begin{enumerate}[(i)]
\item\label{cond11}
$
\partial_{i}\alpha_{\omega} = \tau_{i,q}^{\omega}\cdot\alpha_{d_{i}\omega},
$
and 
\item\label{cond21}
$\tau_{i,q}^{\omega}$ commutes with $L_{q-1}$.
\end{enumerate}
If we identify such elements, we are eventually able to define maps
\begin{align*}
\bd_{i}\colon \cplantq &\to \cplantqi \\
(\omega, \sigma L_{q}) &\mapsto (d_{i} \omega, \sigma\tau_{i,q}^{\omega}L_{q-1}),
\end{align*}
which are independent of the choice of a representative for the coset $\sigma L_{q}$ because of condition~(\ref{cond21}) and the fact $L_{q} \subset L_{q-1}$. Furthermore, by condition~(\ref{cond11}), such elements satisfy
\begin{align}
\bd_{i}(\omega, \sigma L_{q}) &= (d_{i} \omega, \sigma\tau_{i,q}^{\omega}L_{q-1})\nonumber \\
&= \Gamma^{-1}(\sigma\tau_{i,q}^{\omega}\cdot\alpha_{d_{i} \omega})\nonumber \\
&= \Gamma^{-1}(\sigma\cdot\partial_{i}\alpha_{\omega})\nonumber \\
&= (\Gamma^{-1}\circ\partial_{i})(\sigma\cdot \alpha_{\omega}) \label{geomfact}\\
&= (\Gamma^{-1}\circ\partial_{i}\circ\Gamma)(\omega, \sigma L_{q}) \nonumber,
\end{align}
as desired. Here, in~(\ref{geomfact}), we used the geometric fact that for any $\sigma\in\Br_{n\cdot\XI}$, we have $\sigma\cdot\partial_{i}\alpha_{\omega} = \partial_{i}(\sigma\cdot\alpha_{\omega})$: The plant deletion operator $\partial_{i}$ commutes with the action of the mapping class defined by $\sigma$.

The face of an IC-sequence of a simplex is defined as the IC-sequence of the corresponding face of a simplex. By Lemma~\ref{orbits}, the colored braid group $\Br_{n\cdot\XI}$ acts transitively on the set of simplices with the same IC-sequence. From these two facts, it is immediate that elements $\tau^{\omega}_{i,q}$ satisfying condition (\ref{cond11}) exist and that the coset $\tau^{\omega}_{i,q}L_{q-1}$ is unique. We are yet to determine them explicitly, and check whether they satisfy condition~(\ref{cond21}).

By the construction of standard simplices, the first $i$ vertices of $\partial_{i}\alpha_{\omega}$ and $\alpha_{d_{i}\omega}$ are identical.
Now, mapping $\alpha_{d_{i}\omega}$ to $\partial_{i}\alpha_{\omega}$ requires transferring the points of the $(q+1)$-th cluster to the $(i+1)$-th cluster in a suitable way: This transfer is performed for one point after the other, starting with the leftmost point. Let
$\widetilde\omega =  (p_{1}, \ldots, p_{(q+1)\cdot\xi})$
be the P-sequence of $\alpha_\omega$. If $p_{i\xi + m} < p_{r \xi + s}$ for some $m, s\in \{1, \ldots, \xi\}$ and  $q\geq r > i$, the $m$-th point of the $(q+1)$-th cluster has to be half-twisted around the endpoint of the arc which (in $\alpha_{d_{i}\omega}$) ends at the $s$-th point of the $r$-th cluster in a positive direction, and in a negative direction otherwise. A careful analysis of this procedure (where we take into account that the braid group acts from the left) yields the following formula:
\begin{equation}\label{tau}
\tau_{i,q}^{\omega} = \prod_{j=1}^{\xi}\left( \prod_{k=(i+1)\cdot\xi}^{(q+1)\cdot \xi - 1}
\left(\sigma_{k-j+1}\right)^{\sgn(p_{k+1} - p_{(i+1)\cdot\xi - j + 1})}\right)
\end{equation}
Here, $\sgn\colon \Z \to \{-1, 1\}$ is the signum function. Visibly, the largest index of a braid generator involved is $(q+1)\cdot \xi - 1$, so $\tau_{q,i}^{\omega}$ commutes with the elements of $L_{q}$, where the smallest index involved is $(q+1)\cdot\xi + 1$. Thus, condition~(\ref{cond21}) is satisfied.

An example for the stepwise construction of $\tau_{i,1}^{\omega}$ can be found in Figure~\ref{fig-tau}.

\begin{figure}
\centering
\captionsetup[subfigure]{labelformat=empty}
\begin{subfigure}[b]{0.28\linewidth}
\centering
\includegraphics[scale=0.15]{proc3.pdf}
\caption{$\alpha_{d_{1}\omega}$}
\end{subfigure}
\begin{minipage}[b][3cm][c]{0.04\linewidth}
$\overset{\sigma_{4}}{\longmapsto}$
\vspace{0.5cm}
\end{minipage}
\begin{subfigure}[b]{0.28\linewidth}
\centering
\includegraphics[scale=0.15]{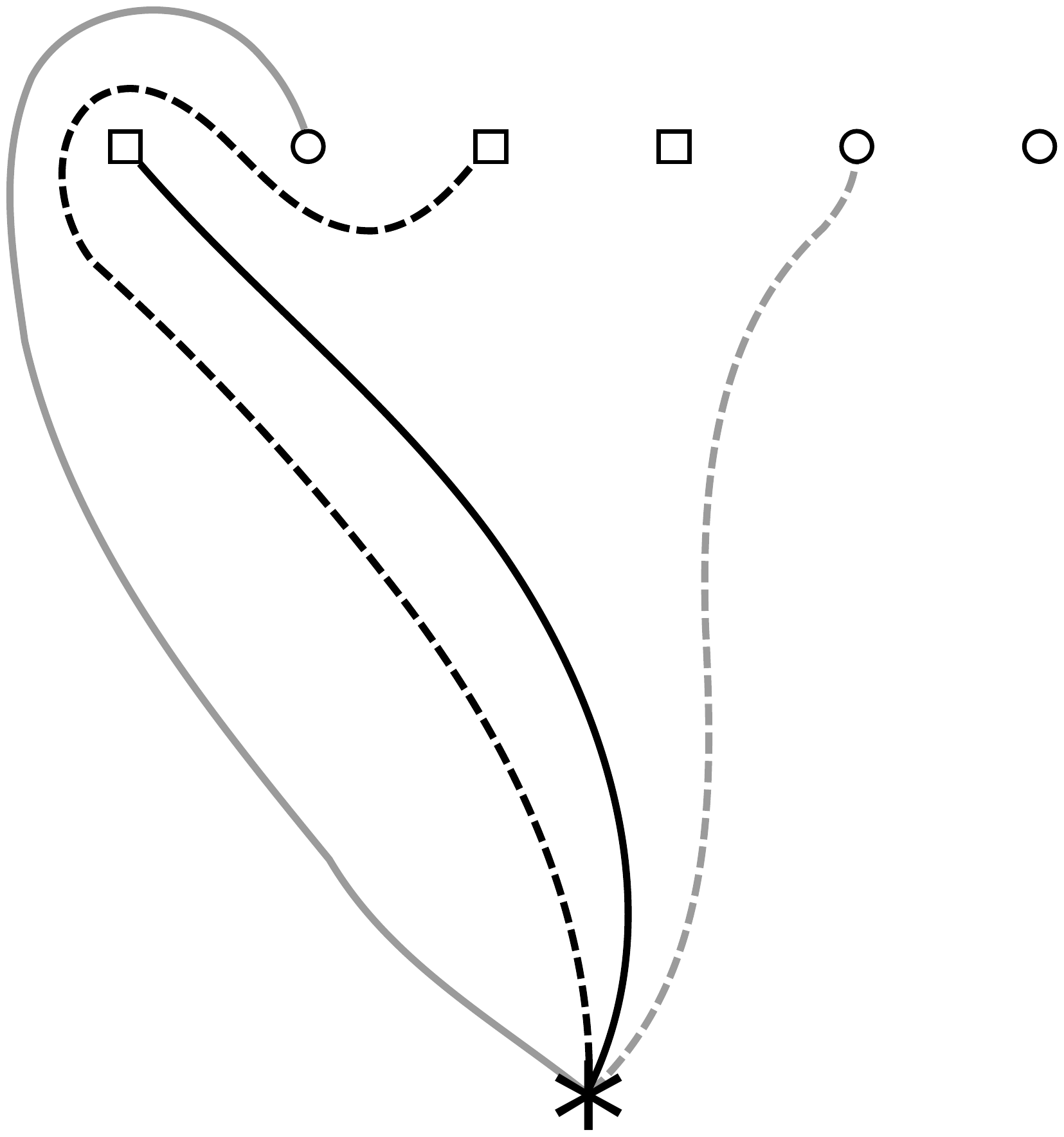}
\caption{$\;$}
\end{subfigure}
\begin{minipage}[b][3cm][c]{0.04\linewidth}
$\overset{\sigma_{3}}{\longmapsto}$
\vspace{0.5cm}
\end{minipage}
\begin{subfigure}[b]{0.28\linewidth}
\centering
\includegraphics[scale=0.15]{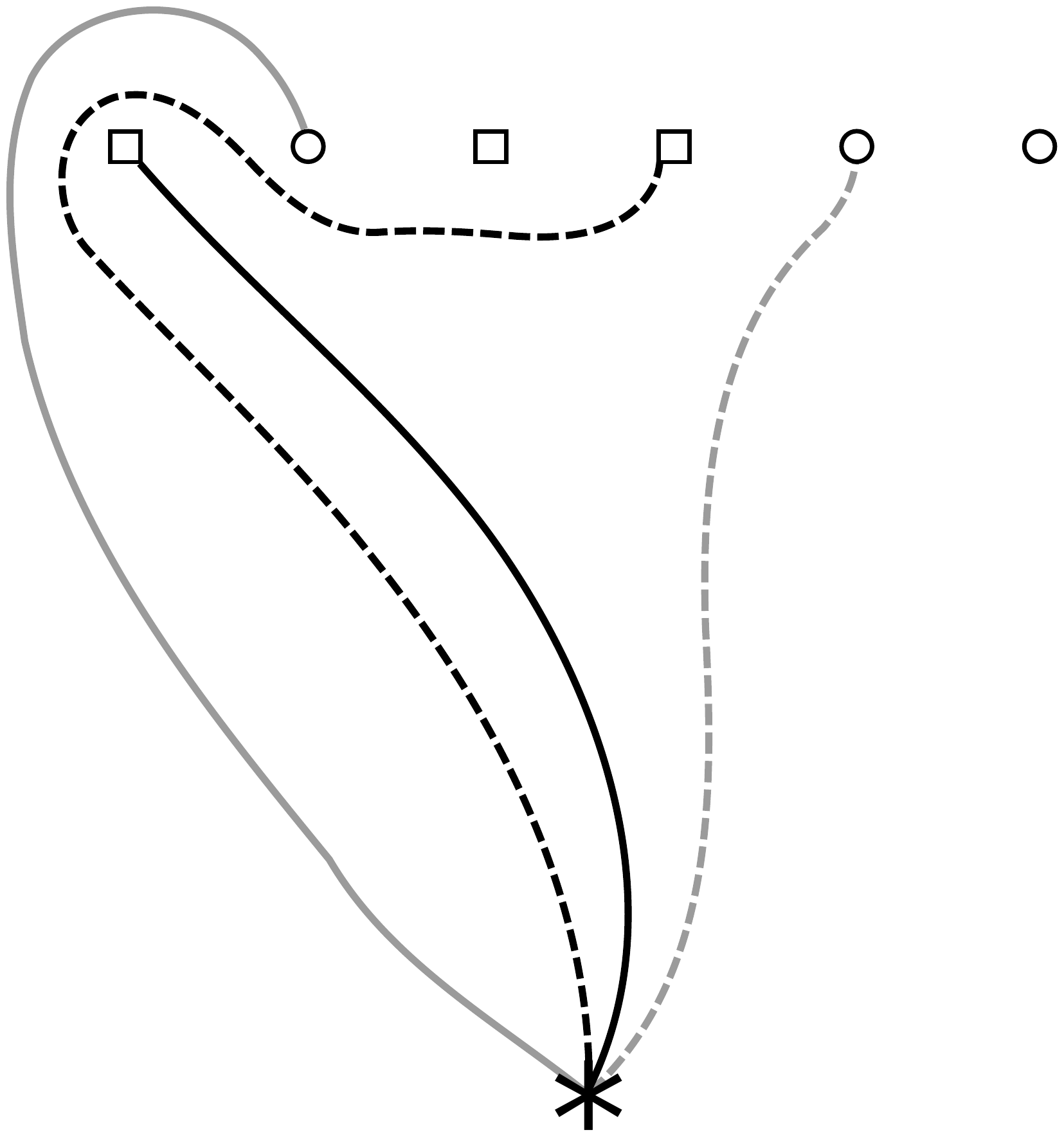}
\caption{$\;$}
\end{subfigure}

\vspace{4ex}

\begin{minipage}[b][3cm][c]{0.04\linewidth}
$\overset{\sigma_{5}^{-1}}{\longmapsto}$
\vspace{0.5cm}
\end{minipage}
\begin{subfigure}[b]{0.28\linewidth}
\centering
\includegraphics[scale=0.15]{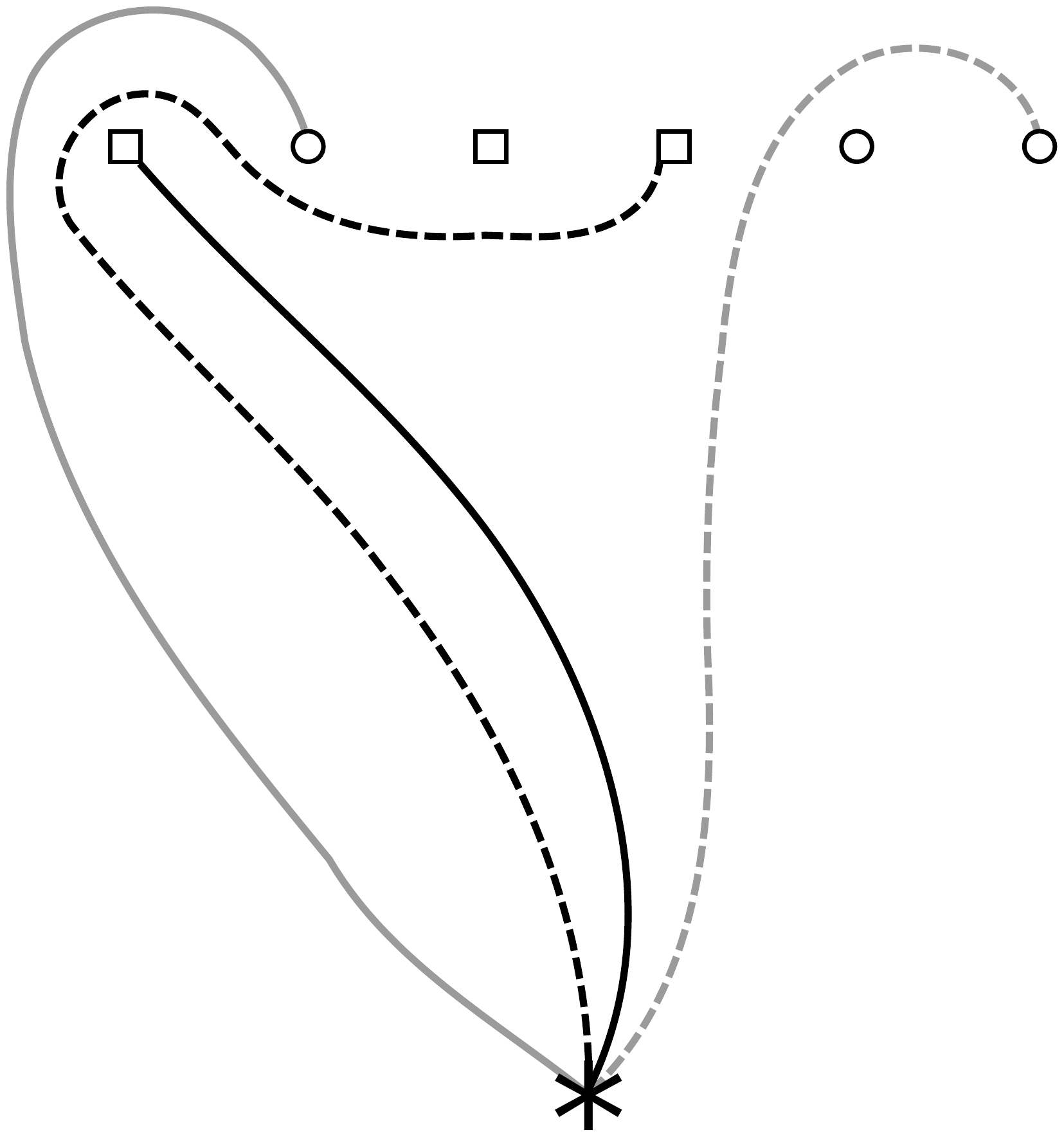}
\caption{$\;$}
\end{subfigure}
\begin{minipage}[b][3cm][c]{0.04\linewidth}
$\overset{\sigma_{4}^{-1}}{\longmapsto}$
\vspace{0.5cm}
\end{minipage}
\begin{subfigure}[b]{0.28\linewidth}
\centering
\includegraphics[scale=0.15]{proc2.pdf}
\caption{$\tau_{1,2}^{\omega}\alpha_{d_{1}\omega} = \partial_{1}\alpha_{\omega}$}
\end{subfigure}
\caption{Passing from $\alpha_{d_{1} \omega}$ to $\partial_{1}\alpha_{\omega}$ ($\omega$ from Example~\ref{ex-faces}).}
\label{fig-tau}
\end{figure}

\begin{remark}\label{pic-tau}
\begin{enumerate}[(i)]
\item
A priori, $\tau^{\omega}_{i,q}\in\Br_{n\cdot\XI}$ for some fixed $n>q$. We note that the definition in~(\ref{tau}) does not depend on the particular choice of $n$, so we can regard $\tau^{\omega}_{i,q}$ as a common element of all $\Br_{n\cdot\XI}$ for $n>q$, using the inclusions $\Br_{n\cdot\XI} \hookrightarrow \Br_{(n+1)\cdot\XI}$ given by attaching $\xi$ trivial strands with coloring $\XI$ to the right of a braid in $\Br_{n\cdot\XI}$.
\item
If $\tilde\omega$ is the P-sequence corresponding to the IC-sequence $\omega$, we sometimes also write $\tau_{i,q}^{\tilde\omega}$ to denote the element $\tau_{i,q}^{\omega}$.
\end{enumerate}
\end{remark}

We have just finished proving the following result:

\begin{proposition}\label{semisimpl-iso}
For $i=0,\ldots,q$, the maps
\begin{align*}
\bd_{i}\colon \cplantq &\to \cplantqi \\
(\omega, \beta L_{q}) &\mapsto (d_{i} \omega, \beta\tau_{i,q}^{\omega}L_{q-1}),
\end{align*}
give $\cplant$ the structure of a semi-simplicial set such that
$
\Gamma\colon \cplant \to \plant
$
is an isomorphism of semi-simplicial sets.
\end{proposition}

\begin{remark}\label{comb-action}
As there is a $\Br_{n\cdot\XI}$-action on $\plant$, there is also a $\Br_{n\cdot\XI}$-action on $\cplant$: A braid $\tau\in\Br_{n\cdot\XI}$ acts via 
$
\tau\cdot(\omega_{p}, \sigma L_{p}) = (\omega_{p}, \tau\sigma L_{p})
$.
\end{remark}

%%%

\section{Hurwitz spaces}\label{hurwitz-spaces}

Let $G$ be a finite group. Following the path pursued in \cite{0912.0325}, we consider Hurwitz spaces of branched $G$-covers of a closed disk $D$. These spaces are relevant to arithmetic applications, cf.~\cite{MR1119950} and \cite{MR2316356}.

Let $n\in\N$, and $*$ a marked point in the boundary of $D$.
We consider (not necessarily connected) branched covers of $D$ described by the following data:
\begin{itemize}
\item[--] a branch locus $B \in\Conf_{n}$,
\item[--] an unbranched covering space map $p\colon Y \to D \setminus B$,
\item[--] a marked point $\bullet$ in the fiber of $p$ above $*$, and
\item[--] a group homomorphism $\alpha\colon G \to \Aut(p)$ which induces a free and transitive action of $G$ on any fiber of $p$.
\end{itemize}
An isomorphism of two such covers is a homeomorphism of the total spaces of the coverings which is compatible with the remaining data. By virtue of the Riemann existence theorem, the isomorphism classes can be parametrized by the following data:

\begin{definition}
A \emph{marked $n$-branched $G$-cover of $D$} is defined as a pair $(B, \mu)$, where $B\in\Conf_{n}$ is a configuration of $n$ branch points, and $\mu\colon\pi_{1}(D\setminus B, *) \to G$ is a homomorphism.
\end{definition}

\begin{remark}
We call such covers \emph{marked} since we do not consider the monodromy homomorphisms $\mu\colon\pi_{1}(D\setminus B, *) \to G$ up to conjugacy in the target. This amounts to marking the point $\bullet$ in the fiber of a branched cover above $*$.
\end{remark}

The space of marked $n$-branched $G$-covers must be a covering space of $\Conf_{n}$ with fiber $\Hom(\pi_{1}(D\setminus B, *), G) \cong G^{n}$. The elements of $G^{n}$ are called \emph{Hurwitz vectors}. Such vectors are unique up to the choice of a basis for $\pi_{1}(D\setminus B, *) \cong F_{n}$ which consists of loops around the single points of $B$, i.e., up to the action of $\Map(D\setminus B) \cong \Br_{n}$ on $G^{n}$, given by 
\begin{equation}\label{hurwitz-action}
\sigma_{i} \cdot \underline g = (g_{1}, \ldots, g_{i-1}, g_{i}g_{i+1}g_{i}^{-1}, g_{i}, g_{i+2}, \ldots, g_{n})
\end{equation}
for $i = 1, \ldots, n-1$, cf.~\cite{MR1509816}. The left monodromy action of $\pi_{1}(\Conf_{n}) \cong \Br_{n}$ on $G^{n}$ can be identified with the \emph{Hurwitz action} (\ref{hurwitz-action}) as well. Replacing $\Conf_{n}$ with the classifying space $\BBr_{n}$ for convenience, we obtain:

\begin{definition}\label{def-hurwitz}
The \emph{Hurwitz space for marked $n$-branched $G$-covers of the disk} is defined as the Borel construction
$$
\Hur_{G,n} = \EBr_n \times_{\Br_{n}} G^n,
$$
where $\Br_n$ acts on $G^{n}$ via the Hurwitz action~(\ref{hurwitz-action}).
\end{definition}

\subsection*{Combinatorial invariants}
The connected components of $\Hur_{G,n}$ are indexed by the set of $\Br_{n}$-orbits in $G^{n}$.
Below, we list some $\Br_{n}$-invariant functions on Hurwitz vectors in $G^{n}$. Such invariants must be constant on connected components of $\Hur_{G,n}$.

Let $(B, \mu)$ be a fixed marked $n$-branched $G$-cover of $D$ with Hurwitz vector $\underline g = (g_{1}, \ldots, g_{n})\in G^{n}$.

\begin{itemize}
\item[--] The \emph{global monodromy} of $(B, \mu)$ is the subgroup of $G$ generated by $g_{1}, \ldots, g_{n}$. It is equal to $G$ if and only if the corresponding branched cover is connected. In this case, we say that the branched cover has \emph{full monodromy}.
\item[--] The \emph{boundary monodromy} of $(B, \mu)$ is defined as the product $\partial\underline g = \prod_{i=1}^{n} g_{i}$. Its inverse describes the branching behavior \emph{at infinity}.
\item[--] Given distinct conjugacy classes $c_{1}, \ldots, c_{t}$ such that all entries of $\underline g$ lie in some $c_{i}$, the \emph{shape (vector)} of $(B, \mu)$ is the $t$-tuple $(n_{c_{1}}(B,\mu), \ldots, n_{c_{t}}(B, \mu))$, where $n_{c_{i}}(B, \mu)$ is the number of elements of $\underline g$ that lie in~$c_{i}$.
\end{itemize}

Considering possible shapes for a nontrivial group $G$, we obtain a lower bound $b_{0}(\Hur_{G,n}) \geq n+1$ for the zeroth Betti number. It makes thus sense to consider sequences of subspaces of Hurwitz spaces which do not \emph{a priori} exclude the possibility of the existence of a homological stability theorem.

From now on, let $c = (c_{1}, \ldots, c_{t})$ be a tuple of $t$ distinct nontrivial conjugacy classes in~$G$.
For a shape vector $\XI = (\xi_{1}, \ldots, \xi_{t})\in\N^{t}$ of length $\xi = \sum_{i=1}^{t} \xi_{i}$, we consider the subspace $\Hur^{c}_{G,\XI}$ of $\Hur_{G,\xi}$ which parametrizes covers with shape $\XI$, which is a union of connected components of $\Hur_{G, \xi}$

The space $\Hur^{c}_{G,\XI}$ is a cover of $\BBr_{\xi} \cong \Conf_{\xi}$ with the tuples in $G^\xi$ for which $\xi_{i}$ entries lie in $c_{i}$ as a fiber. This cover factors over $\BBr_{\XI} \cong \Conf_{\XI}$, which is the classifying space of the colored braid group $\Br_{\XI}$. The fiber of the unbranched cover $\Hur^{c}_{G,\XI} \to \BBr_{\XI}$ can be identified with
$$
\cc = c_{1}^{\xi_{1}} \times \ldots \times c_{t}^{\xi_{t}} ,
$$
on which $\Br_{\XI}$ acts via the Hurwitz action. Generalizing to covers with shape vector $n\cdot\XI= (n\xi_{1}, \ldots, n\xi_{t})$, we may write
 $$\Hur_{G,n\cdot\XI}^{c} = \EBr_{n\cdot\XI} \times_{\Br_{n\cdot\XI}} \cc^{n},$$
where we identify $\Br_{n\cdot\XI}$ with the set stabilizer of $\cc^{n}$ under the $\Br_{n\xi}$-action on $G^{n\xi}$.

\subsection*{Structure on Hurwitz spaces}

We will now identify more structure Hurwitz spaces in order to support the homological investigations to follow. Before starting, it is good to know that we may work in the category of CW complexes: Hurwitz spaces and thus all of their components are homotopy equivalent to finite CW complexes by~\cite[Prop.~2.5]{0912.0325}.

For $m, n \in \N_{0}$, we obtain continuous maps
$$
\Hur_{G,m\cdot\XI}^{c} \times \Hur_{G,n\cdot\XI}^{c} \to \Hur_{G, (m+n)\cdot\XI}^{c}
$$
from the inclusions $\Br_{m\cdot\XI} \times \Br_{n\cdot\XI} \to \Br_{(m+n)\cdot\XI}$ and $\cc^{m} \times \cc^{n} \to \cc^{m+n}$ which are defined and associative up to homotopy.
These maps give $\bigsqcup_{n\geq 0} \Hur_{G,n\cdot\XI}^{c}$ the structure of a disconnected $H$-space with homotopy identity $\Hur_{G,0\cdot\XI}^{c}$.

Let $A$ be a commutative ring. The $H$-space structure on the union of Hurwitz spaces induces a graded (grading in the $n$-variable) ring structure on the direct sum of the zeroth homologies:

\begin{definition}\label{def-ring}
The $A$-module
$$
R^{A,c}_{G,\XI} = \bigoplus_{n\geq 0} H_{0}(\Hur_{G,n\cdot \XI}^{c}; A)
$$
is called the \emph{ring of connected components (with coefficient ring $A$)} for the sequence $\{\Hur_{G,n\cdot \XI}^{c} \mid n\geq 0\}$ of Hurwitz spaces. If $G$, $\XI$, $A$, and $c$ are clear from the context, we simply denote the ring by $R$.
\end{definition}

\begin{remark}\label{comb-descr}\label{deg-one}
There is a nice combinatorial description of the ring $R$, cf.~\cite{MR1119950}:
Let $\mathfrak{s} = \bigsqcup_{n\geq 0} \cc^{n}/\Br_{n\cdot\XI}$. Concatenation of Hurwitz vectors gives $\mathfrak{s}$ the structure of a monoid with the empty tuple as the identity. Then, $R$ is the monoid algebra $A[\mathfrak{s}]$.

In particular, $R$ is finitely generated as an $A$-algebra: Its degree one part is generated as an $A$-module by elements $r(g)$ with $g\in\cc/\Br_{\XI}$. Any element of $\cc^{n}$ is the concatenation of $n$ elements in $\cc$, and this concatenation descends to a map $(\cc / \Br_{\XI})^{n} \to \cc^{n} / \Br_{n\cdot\XI}$ by virtue of the natural inclusion $(\Br_{\XI})^{n} \to \Br_{n\cdot\XI}$. Therefore, the degree one elements $r(g)$ generate $R$ as an $A$-algebra.
\end{remark}

The direct sum of the $p$-th homology modules of Hurwitz spaces obtains the structure of a graded $R$-module from the $H$-space structure in connection with the K\"unneth formula: The graded $A$-module (grading in the $n$-variable)
$$
M_{G, \XI,p}^{A,c} = \bigoplus_{n \geq 0} H_{p}(\Hur_{G,n\cdot\XI}^{c};A),
$$
has the structure of a graded $R^{A,c}_{G,\XI}$-module.
If no misunderstandings are possible, we denote it by $M_{p}$. Clearly, we have $M_{0} = R$.

%%%

\section{Homological stability for Hurwitz spaces}\label{homstabhurwitz}

Let $G$ be a finite group, $c = (c_{1}, \ldots, c_{t})$ a tuple of distinct conjugacy classes in $G$, $\XI\in\N^{t}$, and $\cc = c_{1}^{\xi_{1}}\times\ldots\times c_{t}^{\xi_{t}}$. Let moreover $A$ be a commutative ring. In what follows, we use the notion of the \emph{ring of connected components} introduced in Definition~\ref{def-ring}. We usually write $R$ instead of $R^{A,c}_{G,\XI}$.
For a central element $U\in R$ and $R[U]$ the $U$-torsion in $R$, we use the notation $D_{R}(U) = \max\{\deg (R/UR), \deg (R[U])\}$. We work with colored plant complexes of the form
$
\plant = \mathcal{O}_{n\cdot\XI}^{\XI}(D) \cong \cplant.
$

We study the homology of the spaces
$
\Hur_{G, n\cdot \XI}^{c} = \EBr_{n\cdot{\XI}} \times_{\Br_{n\cdot\XI}} \cc ^{n}
$
, for $n\geq 0$.

\subsection*{The purely abelian case}\label{abelian}
At first, we consider the case of a central homogeneous element $U\in R$ such that $D_{R}(U) = 0$. In this case, $U$ is necessarily of degree one and induces an isomorphism $R_{i} \cong R_{i+1}$ in any degree $i\geq 0$, so $R \cong A[x]$ must hold. Hence, there is only one $\Br_{\XI}$-orbit in $\cc$.
This necessarily implies that any conjugacy class in $c$ is a singleton, since otherwise there would be multiple boundary monodromies and thus multiple $\Br_{\XI}$-orbits in $\cc$. In other words, the monodromy $\mu\colon \pi_{1}(D\setminus B) \to G$ of the covers in $\Hur_{G, n\cdot\XI}^{c}$ is a subgroup of the center of $G$. We call such covers \emph{purely abelian}.

Vice versa, if all covers in $\Hur_{G, n\cdot\XI}^{c}$ are purely abelian, any conjugacy class in $c$ is a singleton and thus the single element of $\cc$ defines an element $U\in R$ such that $D_{R}(U) = 0$. We see that we have
\begin{equation}\label{hurwitz-conf}
\Hur_{G, n\cdot\XI}^{c} = \EBr_{n\cdot\XI} \times_{\Br_{n\cdot\XI}} \cc^{n} = \BBr_{n\cdot\XI}
\cong \Conf_{n\cdot\XI}.
\end{equation}
By (\ref{tran}), for all $p\geq 0$, we have
$
H_{p}(\Conf_{n\cdot\XI};\Z) \cong H_{p}(\Conf_{(n+1)\cdot\XI};\Z)
$
with stable range $n\geq \frac{2p}{\min\underline\xi}$. As a result, we obtain:

\begin{corollary}\label{abelian-case}
If there exists a central homogeneous element $U\in R_{>0}$ such that $D_{R}(U) = 0$ (equivalently, if $\Hur_{G,n\cdot\XI}^{c}$ parametrizes purely Abelian covers), there is an isomorphism
$ H_{p}(\Hur^{c}_{G,n\cdot\XI}; \Z) \cong H_{p}(\Hur^{c}_{G,(n+1)\cdot\XI}; \Z)
$ for $n\geq \frac{2p}{\min\XI}$.
\end{corollary}

\subsection*{The spectral sequence}\label{spectral-sequence}

The space $\EBr_{n\cdot\XI} \times_{\Br_{n\cdot\XI}}(\plant \times \cc^{n})$
inherits a semi-simplicial structure from the face maps $\partial_{i}$ on $\plant$, where the left action of $\Br_{n\cdot\XI}$ on the product $\plant \times \cc^{n}$ is the diagonal action. The spectral sequence associated to the semi-simplicial space,
\begin{align}\label{the-sequence}
\begin{split}
E^{1}_{qp} &= H_{p}(\EBr_{n\cdot\XI} \times_{\Br_{n\cdot\XI}}(\plantq \times \cc^{n});A)  \\
&\Longrightarrow H_{p+q}(\EBr_{n\cdot\XI} \times_{\Br_{n\cdot\XI}}(\rplant \times \cc^{n});A),
\end{split}
\end{align}
converges to the homology of the realization of the total complex. By Theorem~\ref{delta-conn}(\ref{colored-conn}), the space $\rplant$ is $( \left\lfloor\frac{n}{2}  \right\rfloor - 2)$-connected. Thus, the target of the spectral sequence (\ref{the-sequence}) is isomorphic to $H_{p+q}(\EBr_{n\cdot\XI} \times_{\Br_{n\cdot\XI}} \cc^{n};A) \cong H_{p+q}(\Hur_{G,n\cdot\XI}^{c}; A)$ in degrees $p+q \leq \left\lfloor\frac{n}{2}  \right\rfloor - 2$.

Now, for $q < n$, we identify each of the 
$
l_{q}^{\XI}
$ 
$\Br_{n\cdot\XI}$-orbits in $\plantq$ with a copy of the quotient $\Br_{n\cdot\XI}/L_{q}$, cf.~Lemmas~\ref{orbits} and~\ref{stabilizers}. The subgroup $L_{q}\cong \Br_{(n-q-1)\cdot\XI}$ acts on the last $(n-q-1)\cdot\xi$ entries of $\cc^{n}$. Consequently, we obtain
\begin{align}
\label{spseq-iso}
\begin{split}
E^{1}_{qp} 
&\cong H_{p}(\{1, \ldots, l_{q}^{\XI}\} \times \cc^{q+1} \times (\EBr_{n\cdot\XI} \times_{L_{q}} \cc^{n-q-1} );A ) \\ 
&\cong A^{l_{q}^{\XI}} \otimes_{A} A\langle \cc^{q+1}\rangle \otimes_{A} H_{p}(\Hur_{(n-q-1)\cdot\XI}^{c};A).
\end{split}
\end{align}

The differentials $\dd$ on the $E^{1}$-page are induced by the alternating sum of the face maps on the semi-simplicial space. In the following, we aim to find an explicit identification of the differentials under the isomorphism (\ref{spseq-iso}).

Let $\underline g \in \cc^{n}\subset G^{n\xi}$ for some $n\in\N$.
We write $(\underline g)^{\leq j}\in\cc^{j}$ for the tuple consisting of the first $j\xi$ entries of $\underline g$, and $(\underline g)^{> j}\in\cc^{n-j}$ for the complementary $(n-j)\xi$-tuple.
By $(\underline g)_{j}\in\cc$, we denote the $j$-th $\xi$-tuple in $\underline g$.

\begin{lemma}\label{spseq-translation}
Let $q<n$.
Under the isomorphism~(\ref{spseq-iso}),
$\dd \colon E^{1}_{qp} \to E^{1}_{q-1,p} $
is represented by the linear map
\begin{align*}
\dd = \sum_{i=0}^{q} (-1)^{i}{{{{\partial_{i}}_{*}}}} \colon  &A^{l_{q}^{\XI}} \otimes_{A} A\langle \cc^{q+1}\rangle \otimes_{A} H_{p}(\Hur_{(n-q-1)\cdot\XI}^{c};A) \\
&\to A^{l_{q-1}^{\XI}} \otimes_{A} A\langle \cc^{q}\rangle \otimes_{A} H_{p}(\Hur_{(n-q)\cdot\XI}^{c};A),
\end{align*}
the ${{{\partial_{i}}_{*}}}$, for $i=0,\ldots, q$, being given by linear extension of 
$$
{{{\partial_{i}}_{*}}}(\omega\otimes \underline h\otimes x) = d_{i}\omega \otimes ((\tau_{i,q}^{\omega})^{-1}\cdot\underline h)^{\leq q}\otimes r(((\tau_{i,q}^{\omega})^{-1}\cdot\underline h)_{q+1})\cdot x,
$$
where $\omega$ is the IC-sequence of a $q$-simplex, $\underline h\in\cc^{q+1}$, and $x\in H_{p}(\Hur_{G,(n-q-1)\cdot\XI}^{c};A)$.
\end{lemma}

\begin{proof}
In combinatorial terms, we may write the face maps $\partial_{i}$ on the semi-simplicial space $\EBr_{n\cdot\XI} \times_{\Br_{n\cdot\XI}}(\plant \times \cc^{n})$ as
$$
[(e, (\omega, \sigma L_{q}), \underline g)]_{\Br_{n\cdot\XI}} \mapsto[(e, (d_{i}\omega, \sigma\tau^{\omega}_{i,q} L_{q-1}), \underline g)]_{\Br_{n\cdot\XI}},
$$
where the $\tau_{q,i}^{\omega}$ are defined as in~(\ref{tau}). This may be rewritten as
\begin{equation}\label{claim-1}
[(e, \omega, \underline g)]_{L_{q}} \mapsto [(e \cdot\tau_{i,q}^{\omega}, d_{i}\omega, (\tau_{i,q}^{\omega})^{-1}\cdot\underline g)]_{L_{q-1}}.
\end{equation}

\textbf{Claim:}
The map (\ref{claim-1}) is $L_{q}$-equivariantly homotopic to
\begin{equation}\label{claim-2}
[(e, \omega, \underline g)]_{L_{q}} \mapsto [(e, d_{i}\omega, (\tau_{i,q}^{\omega})^{-1}\cdot\underline g)]_{L_{q-1}}.
\end{equation}

\textbf{Proof of the claim:}
Let $\iota$ be the identity on $\EBr_{n\cdot\XI}$ and $\tau$ multiplication in $\EBr_{n\cdot\XI}$ by $\tau_{i,q}^{\omega}$.  Now, $\tau$ descends to a map $\BBr_{n\cdot\XI} \to \BBr_{n\cdot\XI}$ which is induced by conjugation with $\tau_{i,q}^{\omega}$ in $\Br_{n\cdot\XI}$. Since $\tau_{i,q}^{\omega}$ commutes with the elements of $L_{q}$, this conjugation restricts to the identity on $L_{q}$. Therefore, both $\iota$ and $\tau$ descend to self-maps of $\BL_{q}$ homotopic to the identity (note that we may use $\EBr_{n\cdot\XI} / L_{q}$ as a model for $\BL_{q}$). Hence, $\tau$ is $L_{q}$-equivariantly freely homotopic to $\iota$. From this fact, the claim follows directly. \qed

Now, the map
\begin{align*}
\EBr_{n\cdot\XI} \times_{L_{q}} \cc^{n}  &\to  \EBr_{n\cdot\XI} \times_{L_{q-1}} \cc^{n}  \\
[(e, \underline g )]_{L_{q}} &\mapsto[(e, \underline g )]_{L_{q-1}}
\end{align*}
is identified with
$$
\cc^{q+1} \times \Hur_{G,(n-q-1)\cdot\XI} \to \cc^{q} \times \Hur_{G,(n-q)\cdot\XI},
$$
given by left concatenation of a Hurwitz vector with the last $\xi$-tuple $(\underline g)_{q+1}$ of $\underline g \in \cc^{n}$, where we also identify $$\Hur^{c}_{G,(n-q-i)\cdot\XI} \cong\EBr_{n\cdot\XI} \times_{L_{q-i}} \cc^{n-q-i}$$ for $i = 0,1$. In homology, this corresponds to multiplication by $r((\underline g)_{q+1}) \in R$.
Finally, note that $\tau_{i,q}^{\omega}$ only acts on $(\underline g)^{\leq q+1}$, while $L_{q-1}$ acts on the $\xi$-tuples $(\underline g)^{>q+1}$. Thus, the induced map in homology of (\ref{claim-2}) yields the desired form of ${{{\partial_{i}}_{*}}}$.

The lemma follows from the fact that the differential on the $E^{1}$-page of (\ref{the-sequence}) is given by the induced map in homology of the alternating sum of the face maps.
\end{proof}

For a graded $R$-module $M$ and $q\geq 0$, we write $M(q)$ for its \emph{shift} by $q$.

\begin{definition}\label{k-complexes}
Let $M$ be a graded left $R$-module. The \emph{$\KK$-complex associated to $M$} is defined as the complex $\KK(M)$ with terms
\begin{align*}
\KK(M)_{0} &= M, \\
\KK(M)_{q+1} &= A^{l_{q}^{\XI}} \otimes_{A} A\langle\cc^{q+1} \rangle \otimes_{A} M(q+1)
\end{align*}
for $q \geq 0$, where $l_{q}^{\XI}$ is given as in Lemma~\ref{orbits}.
The differentials on $\KK(M)$ are the linear maps defined by
\begin{align*}
\dd_{q+1}\colon \KK(M)_{q+1} &\to \KK(M)_{q} \\
\omega \otimes \underline g \otimes x 
&\mapsto \sum_{i=0}^{q} (-1)^{i} [d_{i}\omega \otimes ((\tau_{i,q}^{\omega})^{-1}\cdot\underline g)^{\leq q}\otimes r(((\tau_{i,q}^{\omega})^{-1}\cdot\underline g)_{q+1})\cdot x],
\end{align*}
where $\omega$ is the IC-sequence of a $q$-simplex, $\underline g \in \cc^{q+1}$, and $x \in M(q+1)$.
\end{definition}

In a less general form, $\KK$-complexes were introduced in~\cite[Sect.~4]{0912.0325}.
$\KK(M)$ is in fact a complex of graded left $R$-modules, where the grading on $\KK(M)_{q}$ is induced by the grading on $M(q)$: For $M = M_p$, this is immediate, as the $n$-th graded part of the $\KK$-complex is equal to a row in the spectral sequence (\ref{the-sequence}) by construction. The complex property is only needed in this case. The more general case involves computations which utilize the semi-simplicial identity on the face maps of $\plant$. These are performed in the author's Ph.D.~thesis (currently under review). % change when reviewed & accepted

Note that the differential $\dd_{q}$ preserves the grading: The grading on $\KK(M)_{q}$ is induced by the grading on $M(q)$, on which $\dd_{q}$ acts by the alternating sum of multiplication with degree one elements. This cancels out with the shifted grading on $\KK(M)_{q-1}$.

We resume:

\begin{corollary}\label{specseq-corollary}
There is a homological spectral sequence with
$$
E^{1}_{qp} = n\text{-th graded part of } \KK(M_{p})_{q+1},
$$ differentials on the $E^{1}$-page given by the differentials on $\KK(M_{p})$,
which converges to $H_{p+q}(\Hur_{G,n\cdot\XI}^{c};A)$ for $p+q \leq \left\lfloor \frac{n}{2}\right\rfloor - 2$.
\end{corollary}

We now consider the homology of $\KK$-complexes for $M = M_{0} = R$. In this case, multiplication in $R$ gives $\KK(R)_{q}$ (and hence also the homology modules of $\KK(R)$) the structure of a two-sided graded $R$-module. A simpler version of the following lemma is proved in \cite[Lemma~4.11]{0912.0325}.

\begin{lemma}\label{killing}
For all $q\geq 0$, $H_{q}(\mathcal{K}(R))$ is killed by the right action of $R_{>0}$.
\end{lemma}
\begin{proof}
For simplicity of notation, in this proof we work with P- instead of IC-sequences.

Recall that for a Hurwitz vector $\underline g\in\cc^{q+1}$, we write $\partial\underline g$ for its boundary, which is invariant under the $\Br_{(q+1)\cdot\XI}$-action. $R$ is generated as an $A$-module by orbits $[\underline s] \in \cc^{n}/\Br_{n\cdot\XI}$, for $n\geq 0$. Let $h\in\cc$, such that the elements of the form $r(h)$ generate $R_{>0}$ as an $A$-algebra. We define a  map $S_{h}\colon \KK(R)_{q+1} \to \KK(R)_{q+2}$ by linear extension of
$$
S_{h}(\tilde\omega \otimes \underline g \otimes [\underline s]) = (\xi +\tilde\omega ) \otimes (h^{(\partial \underline g \partial [\underline s])^{-1}}, \underline g) \otimes [\underline s],
$$
where $\tilde\omega$ is the P-sequence of a $q$-simplex, and $\underline g, h, [\underline s]$ as above.
Here, $(\xi +\tilde\omega)$ denotes the P-sequence of a $(q+1)$-simplex obtained by increasing every entry of $\tilde\omega$ by $\xi$ and then attaching $(1, \ldots, \xi)$ from the left. In particular, the equations 
\begin{align}\label{diff-new}
\begin{split}
d_{0}(\xi +\tilde\omega) &= \tilde\omega \\
d_{i}(\xi +\tilde\omega) &= (\xi + d_{i-1}\tilde\omega)
\end{split}
\end{align}
hold for $i = 1, \ldots, q$. Furthermore, since the first $\xi$ positions of the P-sequence $\xi + \tilde\omega$ are given by $1, \ldots, \xi$, the equations
\begin{align}\label{diff-new-2}
\begin{split}
(\tau_{0,q+1}^{(\xi +\tilde\omega)} )^{-1}\cdot (h, \underline g )
&= ( \underline g, h^{\partial\underline g}) \\
(\tau_{i+1,q+1}^{(\xi +\tilde\omega)} )^{-1}\cdot (h, \underline g )
&= (h, (\tau_{i,q}^{\tilde\omega})^{-1}\cdot \underline g )
\end{split}
\end{align}
hold for any $h\in\cc$ and $i=0, \ldots, q$, cf.~the definition of $\tau_{i,q}^{\tilde\omega}$ in~(\ref{tau}).

We claim that $S_{h}$ is a chain homotopy from right multiplication with $r(h)$ to the zero map. Indeed, we have
\begingroup
\allowdisplaybreaks
\begin{align*}
&(\dd_{q+1} S_{h} + S_{h} \dd_{q})(\tilde\omega \otimes \underline g \otimes [\underline s]) \\
&\underset{\phantom{(\ref{diff-new-2})}}{\overset{\phantom{(\ref{diff-new})}}{=}} \dd_{q+1}((\xi +\tilde\omega) \otimes (h^{(\partial \underline g \partial \underline s)^{-1}}, \underline g) \otimes [\underline s] ) \\
&\phantom{\underset{(\ref{diff-new-2})}{\overset{(\ref{diff-new})}{=}}}+ 
\sum_{i=0}^{q} (-1)^{i} S_{h}( [d_{i}\tilde\omega \otimes ((\tau_{i,q}^{\tilde\omega})^{-1}\cdot\underline g)^{\leq q}\otimes [(((\tau_{i,q}^{\tilde\omega})^{-1}\cdot\underline g)_{q+1},  \underline s )]] ) \\
&\underset{(\ref{diff-new-2})}{\overset{(\ref{diff-new})}{=}} \tilde\omega \otimes \underline g \otimes (h^{(\partial \underline s)^{-1}}, [\underline s] ) \\
&+ \sum_{i=0}^{q} (-1)^{i+1}[(\xi + d_{i}\tilde\omega) \otimes (h^{(\partial \underline g\partial\underline s)^{-1}}, ((\tau_{i,q}^{\tilde\omega})^{-1}\cdot \underline g )^{\leq q} ) \otimes [(((\tau_{i,q}^{\tilde\omega})^{-1}\cdot \underline g )_{q+1}, \underline s)]] \\
&+ \sum_{i=0}^{q} (-1)^{i}[(\xi + d_{i}\tilde\omega)
\otimes (h^{(\partial( (\tau_{i,q}^{\tilde\omega})^{-1}\cdot\underline g)\partial\underline s)^{-1}}, ((\tau_{i,q}^{\tilde\omega})^{-1}\cdot \underline g )^{\leq q} ) \otimes [(((\tau_{i,q}^{\tilde\omega})^{-1}\cdot \underline g )_{q+1}, \underline s)]] \\
&\underset{\phantom{(\ref{diff-new-2})}}{\overset{\phantom{(\ref{diff-new})}}{=}} \tilde\omega \otimes \underline g \otimes [\underline s \cdot r(h)],
\end{align*}
\endgroup
as $(h^{(\partial \underline s)^{-1}}, \underline s )$ is equivalent to $(\underline s,h)$ under the $\Br_{(n+1)\cdot\XI}$-action. Hence, $S_{h}$ is the desired chain homotopy.
\end{proof}

\subsection*{Modules over stabilized rings}\label{stabilized}

In the following, we study graded modules over graded rings satisfying a specific stability condition which will eventually form the essential criterion for homological stabilization of Hurwitz spaces. This generalizes the modules $M_{p}$ over the ring $R$.

\begin{definition}\label{def_astable}
Let $R = \bigoplus_{i\in\N_{0}} R_{i}$ be some graded ring, $A = R/R_{>0} \cong R_{0}$ the ring of degree zero elements, and $U\in R$ a central homogeneous element of positive degree. The ring $R$ is called \emph{$A$-stabilized by $U$} if the three following conditions are satisfied:
\begin{enumerate}[(i)]
\item Both kernel and cokernel of the multiplication $R \overset{U\cdot}{\to} R$ have finite degree as graded $R$-modules; in other words, $D_R(U) = \max\left\{\deg (R/UR), \deg (R[U])\right\}$ is finite, \label{isocoker}
\item $A$ is commutative, and \label{comnoet}
\item $R$ is generated in degree one (i.e., by $R_{1}$) as an algebra over $A$. \label{degone}
\end{enumerate}
We call $U$ the \emph{stabilizing element} for $R$.
\end{definition}

In what follows, let $M$ be a graded left $R$-module, where $R$ is $A$-stabilized by $U\in R$. We use the following notation:
\begin{align*}
D_{M}(U) &= \max\{\deg(M[U]), \deg(M/UM)\} \\
\delta_{M}(U) &= \max\{\deg (\Tor_{0}^{R}(R/UR, M)), \deg (\Tor_{1}^{R}(R/UR, M))\}
\end{align*}
Though both quantities depend heavily on the stabilizing element $U\in R$, we usually use the symbols $D_{M}$ and $\delta_{M}$.
Furthermore, we write $H_{i}(M)$ for the graded left $R$-module $\Tor_{i}^R(A, M)$.

We will from now on assume that $D_{R}(U)$ is positive. The case $D_{R}(U) = 0$ was fully treated earlier this section in the part about purely abelian covers.

Section 4 of \cite{0912.0325} is about modules over graded rings $R$ which are $A$-stabilized by $U$. In that article's setting, it makes sense to focus on the case where $A$ is a field, though the proofs of Lemma~4.4 through Lemma~4.10 carry over directly to the case where $A$ is an arbitrary commutative ring. We proved an analogue to \cite[Lemma~4.11]{0912.0325} in Lemma~\ref{killing}. Therefore, the proofs and results of Proposition~4.12 and~4.13 of \cite{0912.0325} are also applicable to our situation.

More specifically, we will need the follwing results:
\begin{align}
D_{M} &\leq \max\{\deg H_{0}(M), \deg H_{1}(M)\} + 5D_{R}, \label{homlem4.6} \\
\deg H_{q}(\KK(R)) &\leq D_{R} + \deg U + q, \label{4-12} \\
\deg H_{q}(\KK(M)) &\leq \max\{\deg H_{0}(M), \deg H_{1}(M)\} + (q+5)\cdot D_{R}+ \deg U \label{4-13},
\end{align}
cf.~\cite[Lemma~4.6 and~4.9]{0912.0325}, \cite[Prop.~4.12]{0912.0325}, and \cite[Prop.~4.13]{0912.0325}, respectively.

\begin{proposition}\label{homprop-4.2}
Let $R$ be $A$-stabilized by $U$.
Beyond, let $M$ be a graded left $R$-module and $h_{i} = \deg (H_{i}(\KK(M)))$. Then, we have
$
h_{q} \leq \max\{h_{0}, h_{1}\} + D_{R}q + (5 D_{R} + \deg U),
$
and multiplication by $U$, $M\overset{U\cdot}{\to}M$, is an isomorphism in source degree greater than or equal to $\max\{h_{0}, h_{1}\} + 5D_{R} + 1$.
\end{proposition}

\begin{proof}
The present proof follows \cite[Thm.~4.2]{0912.0325}.

 We show that for $i = 0,1$, 
\begin{equation}\label{claim-prop}
\deg (H_{i}(M)) \leq h_{i}.
\end{equation}
Using this result, we obtain
\begin{align*}
h_{q} &\overset{{(}\ref{4-13}{)}}{\leq} \max\{\deg (H_{0}(M)), \deg (H_{1}(M))\} + (q+5)\cdot D_{R} + \deg U \\
&\overset{(\ref{claim-prop})}{\leq} \max\{h_{0}, h_{1}\} + D_{R}q +(5 D_{R} + \deg U),
\end{align*}
which is the first part of the statement.

Furthermore, by (\ref{homlem4.6}), multiplication by $U$ is an isomorphism in source degree greater or equal $\max\{\deg H_{0}(M), \deg H_{1}(M)\} + 5D_{R} + 1$. Together with~(\ref{claim-prop}), this gives the second claim of the proposition.

It remains to show~(\ref{claim-prop}). For $i=0$, we have $H_{0}(M) = A \otimes_{R} M = M/R_{>0}M$ and $H_{0}(\KK(M)) = M/\im \dd_{1}$. Now, $\dd_{1}(\omega, g, x) = r(g)\cdot x$ for all elementary tensors in $\KK(M)_{1}$, so $\im \dd_{1} = R_{>0}M$ and the claim is vacuously true.

For $i=1$, we factor the map $\dd_{1}\colon \KK(M)_{1} \to M$ as $\dd_{1} = \beta \circ \alpha$,
$$
\KK(M)_{1} = A^{l_{0}^{\XI}} \otimes_{A} A\langle\cc\rangle \otimes_{A} M(1) \overset{\alpha}{\to} R_{>0}\otimes_{R} M\overset{\beta}{\to} M,
$$
with $\alpha(\omega \otimes g \otimes x) = r(g) \otimes x$ and $\beta (r \otimes x) = r\cdot x$. As $R_{>0}$ is generated by elements of the form $r(g)$, we can factor any $r \in R_{>0}$ as $r = r(g)\cdot r'$ for some $r' \in R$; therefore, $\alpha$ is surjective. It is also degree-preserving -- note that $R_{>0} \otimes_{R} M$ is graded via $\deg(r\otimes x) = \deg r + \deg x$.

Now, we have a sequence
$$
\KK(M)_{2} \overset{\dd_{2}}{\to} \ker\dd_{1} \to H_{1}(\KK(M))\to 0
$$
which is by definition exact in the middle and on the right, so $\ker\dd_{1}$ is generated as an $A$-module in degree at most $\max\{\deg(\im\dd_{2}), \deg (H_{1}(\KK(M)))\}$. Now, the composition
$$
\KK(M)_{2} \overset{\dd_{2}}{\to} \ker\dd_{1} \overset{\alpha}{\to} R_{>0}\otimes_{R} M
$$
is zero, since it maps $\omega \otimes \underline g \otimes x$ to the element 
\begin{align*}
&r((( \tau^{\omega}_{0,1} )^{-1}\underline g)_{1} )
\cdot r((( \tau^{\omega}_{0,1} )^{-1}\underline g)_{2} ) \otimes x 
-
r((( \tau^{\omega}_{1,1} )^{-1}\underline g)_{1} )
\cdot r((( \tau^{\omega}_{1,1} )^{-1}\underline g)_{2} ) \otimes x
\end{align*}
which equals zero because $( \tau^{\omega}_{0,1} )^{-1}\underline g$ and $( \tau^{\omega}_{1,1} )^{-1}\underline g$ are equivalent up to the Hurwitz action.
In other words, the elements of $\im\dd_{2} \subset \ker\dd_{1} \subset \KK(M)_{1}$ are killed by $\alpha$. Therefore, $\alpha(\ker \dd_{1})$ is generated as an $A$-module in degree $\leq \deg (H_{1}(\KK(M)))$. 

Now, as $\alpha$ is surjective, this implies $\deg(\ker \beta) \leq \deg (H_{1}(\KK(M)))$ (recall that $A$ is graded trivially). But the exact sequence
$
0 \to R_{>0} \to R \to A \to 0
$
is a projective resolution of $A = R/R_{>0}$; tensoring with $M$, we get an identification
$
H_{1}(M) = \Tor_{1}^{R}(A,M) = \ker\beta.
$ 
Thus, we have proved (\ref{claim-prop}).
\end{proof}

\subsection*{The proof}

\begin{proof}[Proof of Theorem~\ref{the-theorem}]\label{hurwitz-proof}
We follow the proof strategy of~\cite[Thm.~6.1]{0912.0325}, with an extra focus on the determination of the explicit stable range.

By assumption, $D_{R} = D_{R}(U)$ is finite and positive. We know from Remark~\ref{deg-one} that $R$ is generated in degree one as an algebra over the commutative ring $A$. From these facts, we conclude that $R$ is $A$-stabilized by $U$.

As before, let
$
M_{p} = \bigoplus_{n\geq 0} H_{p}(\Hur_{G,n\cdot\XI}^{c}; A).
$
In order to prove the theorem, we need to show that multiplication by $U$, $M_{p} \overset{U\cdot}{\to} M_{p}$, is an isomorphism in source degree $n > (8 D_{R} + \deg U)p + 7D_{R} + \deg U$. To see this, we show that
\begin{equation}\label{thm-statement}
\deg (H_{q}(\KK(M_{p}))) \leq D_{R} + \deg U + (8 D_{R} + \deg U)p + D_{R}q
\end{equation}
holds for all $q\geq 0$. Then, the theorem follows from the second statement of Proposition~\ref{homprop-4.2}, considering the cases $q = 0,1$. We prove~(\ref{thm-statement}) by induction on~$p$.

For $p = 0$, we have $M_{0} = R$. Now, by~(\ref{4-12}),
$$
\deg (H_{q}(\KK(R))) \overset{(\ref{4-12})}{\leq} D_{R} + \deg U + q \leq D_{R} + \deg U + D_{R} q,
$$
which implies the assertion.

For the inductive step, suppose that~(\ref{thm-statement}) holds for $0\leq p<P$. The final terms of $\KK(M_{P})$ are given by
$$
\KK(M_{P})_{2} \overset{\dd_{2}}{\to} \KK(M_{P})_{1} \overset{\dd_{1}}{\to} M_{P}.
$$
The $n$-th graded part of $\dd_{2}$ is a differential $\dd\colon E^{1}_{1P} \to E^{1}_{0P}$ in the spectral sequence from Corollary~\ref{specseq-corollary}. In the range $p\leq\left\lfloor\frac{n}{2}  \right\rfloor - 2$, we can identify $\dd_{1}$ with an edge map in the same spectral sequence.

Now, $E^{2}_{qp}$ is given by the $n$-th graded part of $H_{q+1}(\KK(M_{p}))$. From the inductive hypothesis, for $j>1$, we obtain $E^{2}_{j, P+1-j} = 0$ for $$n > 10 D_{R} + 2\deg U - (7 D_{R}+\deg U)j + (8D_{R}+\deg U)P.$$
Hence, we have $E^{2}_{0P} = E^{\infty}_{0P}$ for $n>-4 D_{R} + (8D_{R}+\deg U)P$, as there are no nonzero differentials going into or out of $E^{2}_{0P}$.

Similarly, for $j>0$, we have $E^{2}_{j,P-j} = 0$ for
$$n > 2 D_{R} + \deg U - (7 D_{R}+\deg U)j + (8D_{R}+\deg U)P.$$
Thus, for $n > -5 D_{R } + (8D_{R}+\deg U)P$, the only graded piece of $H_{P}(\Hur_{G,n\cdot\XI}^{c};A)$ which does not vanish is $E^{\infty}_{0P}$.

Combining these results, we see that $E^{2}_{0P} \cong E^{\infty}_{0P} \cong H_{P}(\Hur_{G,n\cdot\XI}^{c};A)$ as long as $n> -4 D_{R} + (8D_{R}+\deg U)P$. In particular, the edge map $\coker (d_{2}) \to M_{P}$ is an isomorphism in degrees above $-4 D_{R} + (8D_{R}+\deg U)P$, and so
\begin{equation}\label{pf-bound}
\max\{\deg (H_{0}(\KK(M_{P}))), \deg (H_{1}(\KK(M_{P})))\} \leq -4 D_{R} + (8D_{R}+\deg U)P.
\end{equation}

Now, we make use of the first statement of Proposition~\ref{homprop-4.2} and obtain
\begin{align*}
\deg (H_{q}(\KK(M_{P}))) &\overset{(\ref{pf-bound})}{\leq} -4 D_{R} + (8D_{R}+\deg U)P + D_{R}q + 5D_{R} + \deg U\\
&\overset{\phantom{(\ref{pf-bound})}}{=} D_{R} + \deg U  + (8D_{R}+\deg U)P + D_{R}q,
\end{align*}
which is what we wanted to show.
\end{proof}

\subsection*{Unmarked covers}

The $G$-action on Hurwitz vectors by simultaneous conjugation induces an action of $G$ on the Hurwitz spaces $\Hur_{G,n\cdot\XI}^{c}$. The corresponding \emph{Hurwitz space for unmarked covers} is defined as the quotient $\mathcal{H}_{G,n\cdot\XI}^{c} = \Hur_{G,n\cdot\XI}^{c}/G$. Now, in general, $G$ does not act freely on Hurwitz vectors. Anyway, for suitable stabilizing elements $U$, homological stability for Hurwitz spaces descends to spaces of unmarked covers, as we will see in this section's theorem.

Note that since the Hurwitz action commutes with conjugation, the $G$-action on Hurwitz vectors gives the ring $R$ and the modules $M_{p}$ the structure of modules over the group ring $A[G]$.

\begin{corollary}\label{thm-unmarked}
Let $A$ be a field whose characteristic is either zero or prime to the order of $G$.
Assume that there is a $G$-invariant element $U\in R$ which for any $p\geq 0$ induces isomorphisms $H_{p}(\Hur_{G, n\cdot\XI}^{c} ; A) \overset{\sim}{\to} H_{p}(\Hur_{G, (n+\deg U)\cdot\XI}^{c} ; A)$ for $n\geq r(p)$. Then,
$H_{p}(\mathcal{H}_{G, n\cdot\XI}^{c} ; A) \cong H_{p}(\mathcal{H}_{G, (n+\deg U)\cdot\XI}^{c} ; A)$ holds in the same range.
\end{corollary}

\begin{proof}
By a transfer argument, we have $H_{p}(\mathcal{H}^{c}_{G,n\cdot\XI};A) \cong H_{p}(\Hur_{G,n\cdot\XI}^{c};A)_{G}$ for all $n,p\geq 0$.
The assumption that $U$ is fixed under the action of $G$, together with the $G$-e\-qui\-va\-ri\-ance of the maps $G^{n\xi}\times G^{\deg U\cdot\xi} \to G^{(n+\deg U)\xi}$, implies that in the stable range, $H_{p}(\Hur_{G, n\cdot\XI}^{c} ; A) \cong H_{p}(\Hur_{G, (n+\deg U)\cdot\XI}^{c} ; A)$ is an isomorphism of $A[G]$-modules. Taking $G$-coinvariants yields the result.
\end{proof}

%%%%%

\section{Application}\label{application}

We work in the same setting as in the previous section.

\subsection*{Connected covers and stable homology}

In the purely abelian case, we have seen in Section~\ref{homstabhurwitz} that Hurwitz spaces \emph{are} homotopy equivalent to colored configuration spaces. In the preprint~\cite{1212.0923}, this is made more general by showing that under certain conditions, the stable homology of the components of Hurwitz spaces is isomorphic to the stable homology of the corresponding colored configuration space. Due to a false reasoning in the application of results from the earlier article~\cite{0912.0325}, the named preprint was withdrawn from the arXiv in~2013. Fortunately, the statement of Theorem~\ref{evwii} remains untouched. We refer to Ellenberg's blog post \cite{Quomod} for an explanation of the mistakes and a clarification which results are still correct.

For $\underline a = (a_{1}, \ldots, a_{t})\in \N^{t}$, we define the Hurwitz vector
\begin{equation}\label{v-vector}
V = V(\underline a) = \prod_{i=1}^{t}\prod_{g \in c_{i}} \left( g^{(a_{i} \ord(g))} \right),
\end{equation}
where the product operation means concatenation of tuples. Now, if there is an $n \in \N$ such that we have $\sum_{g\in c_{i}} a_{i}\ord(g) = n\xi_{i}$ for all $i = 1, \ldots, t$, we have  $V \in \cc^{n}$ up to the action of $\Br_{n\cdot\XI}$.  It is not hard to see that for any $\XI\in\N^{t}$, such an $\underline a$ exists. Hence, in this case, $V$ may be interpreted as an element of $R = R^{A,c}_{G,\XI}$, and we write $V\in R$. By construction, we have $\partial V = 1$, so $V$ is central in $R$.

By $\CHur_{G,n\cdot\underline\xi}^{c} \subset \Hur_{G,n\cdot\underline\xi}^{c}$, we denote the union of connected components of $\Hur_{G,n\cdot\underline\xi}^{c}$ parametrizing covers with full monodromy.

\begin{theorem}[\textsc{Ellenberg--Venkatesh--Westerland}, {\cite[Cor.~5.8.2]{1212.0923}}]\label{evwii}
Suppose that for any $p\geq 0$, the element $V$ from (\ref{v-vector}) induces an isomorphism 
$$
H_{p}(\CHur_{G,n\cdot\underline\xi}^{c};\Q)\overset{\sim}{\to}
H_{p}(\CHur_{G,(n+\deg V)\cdot\underline\xi}^{c};\Q)
$$
for $n\geq r(p)$. Then for any connected component $X$ of $\CHur_{G,n\cdot\underline\xi}^{c}$, the branch point map $X \to \Conf_{n\cdot \underline\xi}$ induces an isomorphism
$$
H_{p}(X;\Q) \overset{\sim}{\to} H_{p}(\Conf_{n\cdot\underline\xi};\Q)
$$
whenever $n\geq r(p)$.
\end{theorem}

In Theorem~\ref{the-theorem}, we give a condition for homological stabilization of the sequence $\{\Hur_{G,n\cdot\XI}^{c} \}$, while Theorem~\ref{evwii} is about the subspaces $\CHur_{G,n\cdot\XI}^{c}$ \emph{of connected} covers. In order to make the two theorems compatible, the condition
\begin{equation}\label{all-connected}
\Hur_{G, n\cdot\XI}^{c} = \CHur_{G, n\cdot\XI}^{c} \text{ for all } n\geq 1
\end{equation}
must be satisfied. Now, (\ref{all-connected}) holds if and only if $G$ is \emph{invariably generated} by $c$:

\begin{definition}
We say that $G$ is \emph{invariably generated} by a tuple $c = (c_{1}, \ldots, c_{t})$ of distinct conjugacy classes in $G$ if for all choices of elements $g_{i} \in c_{i}$, $i = 1, \ldots, t$, $G$ is generated by $g_{1}, \ldots, g_{t}$. In this case, we call $c$ an \emph{invariable generation system} for $G$.
\end{definition}

By Jordan's theorem (\cite{Jordan}), a list of all of its nontrivial conjugacy classes invariably generates $G$.

The following proposition is a slight adaptation of a standard result about Hurwitz action orbits.

\begin{proposition}\label{conn-standard}
If $c$ invariably generates $G$, there is an $N\in\N$ such that for all $n \geq N$, concatenation with any $g \in \cc$ yields a bijection
$$
\cc^{n}/\Br_{n\cdot\XI} \overset{1:1}{\longleftrightarrow} \cc^{n+1}/\Br_{(n+1)\cdot\XI}.
$$
Thus, any Hurwitz vector $U\in R = R^{\Z,c}_{G,\XI}$ induces an isomorphism
$$
H_{0}(\Hur_{G, n\cdot\XI}^{c} ; \Z ) \cong H_{0}(\Hur_{G, (n+\deg U)\cdot\XI}^{c} ; \Z)
$$
for all $n \geq N$. In particular, $D_{R}(U)$ is finite for any vector $U \in R$ with $\partial U = 1$.
\end{proposition}
\begin{proof}
We follow the proof in \cite[Prop.~3.4]{0912.0325}, where the first statement is proved for $t = 1$. The last two statements are direct consequences of the first one.

Let $\underline h \in \cc^{n+1}$. We need to show that for $n$ sufficiently large, there is a tuple $\underline h' \in \cc^{n}$ such that $\underline h$ is equivalent under the $\Br_{n\cdot\XI}$-action to $(g, \underline h')$. This shows that the maps $\cc^{n}/\Br_{n\cdot\XI} \to \cc^{n+1}/\Br_{(n+1)\cdot\XI}$ given by concatenation with $g = (g_{1}, \ldots, g_{\xi})$ are surjective for $n \gg 0$; since the involved sets are finite, it follows that these maps are eventually bijective.

In the following, we work with the full $\Br_{n\xi}$-action. If we construct a tuple $(g, \underline h'')$ which is equivalent under the $\Br_{n\xi}$-action to $\underline h$, there is another braid which transforms~$\underline h''$ to an element of $\cc^{n}$, since the Hurwitz action permutes conjugacy types. Thus, it suffices to show that we can realize any $g_{0}\in G$ as the first entry of a tuple which is $\Br_{n\xi}$-equivalent to $\underline h$; the claim follows by successive application of this property.

Assume $g_{0} \in c_{1}$. For $n\gg 0$, there exists an element $g'_{0} \in c_{1}$ that appears at least $d+1 = \ord(g'_{0}) + 1$ times in $\underline h$. We may use the Hurwitz action to pull $d$ of these elements to the front of $\underline h$, resulting in a new tuple $({g'_{0}}^{(d)}, \tilde h_{1}, \ldots, \tilde h_{n\xi - d})$. By the invariable generation property, the elements $\tilde h_{1}, \ldots, \tilde h_{n\xi - d}$ generate $G$ (note that $g'_{0}$ appears at least once in these last $n\xi - d$ entries).

Now for all $i = 1, \ldots, n\xi -d$, there is a braid $\sigma_{i}\in\Br_{n\xi}$ which satisfies
$$
\sigma_{i}\cdot ({g'_{0}}^{(d)}, \tilde h_{1}, \ldots, \tilde h_{n\xi - d}) = ((\tilde h_{i}g'_{0}\tilde h_{i}^{-1})^{(d)}, \tilde h_{1}, \ldots, \tilde h_{n\xi - d}).
$$
It is given by
$$
\sigma_{i} =  \alpha_{i}^{-1} (\sigma_{d+i-1} \cdots \sigma_{i+1} \sigma_{i}^{2} \sigma_{i+1} \cdots \sigma_{d+i-1}) \alpha_{i},
$$
where $\alpha_{i}$ pulls the $d$-tuple $({g'_{0}}^{(d)})$ in front of $\tilde h_{i}$, which works since the boundary of $({g'_{0}}^{(d)})$ is trivial.

Thus, the Hurwitz action may conjugate the elements $g'_{0}$ at the beginning of the tuple by any element in the group generated by $\tilde h_{1}, \ldots, \tilde h_{n\xi - d}$, which is equal to $G$. Thus, we may establish $g_{0}$ as the first entry.
\end{proof}

We are now ready to conclude:

\begin{theorem}\label{thm-connected}
Let $G$ be a finite group, $c = (c_{1}, \ldots, c_{t})$ an invariable generation system for $G$, and $\XI = (\xi_{1}, \ldots, \xi_{t})\in\N^{t}$.
Let $U \in R = R^{\Z.c}_{G,\XI}$ be a Hurwitz vector with $\partial U = 1$.
Then for any $p \geq 0$, there are isomorphisms
\begin{align*}
H_{p}(\Hur_{G, n\cdot\XI}^{c}; \Z ) &\cong H_{p}(\Hur_{G, (n+\deg U)\cdot\XI}^{c}; \Z )\\
H_{p}(\Hur_{G, n\cdot\XI}^{c}; \Q ) &\cong H_{p}(\Hur_{G, (n+1)\cdot\XI}^{c}; \Q )
\end{align*}
for $n> (8D_{R}(U) + \deg U)p + 7D_{R}(U) + \deg U$. For $b = b_{0}(\Hur_{G, (D_{R}(U) + 1)\cdot\XI}^{c})$,
$$
H_{p}(\Hur_{n\cdot\XI}^{c};\Q) \cong H_{p}( \Conf_{n\cdot\XI}; \Q ) \otimes_{\Q} \Q^{b}
$$
in the same range.
\end{theorem}
\begin{proof}
For $G$ abelian, an even stronger statement follows from Corollary~\ref{abelian-case}.

We may thus assume that $D_{R}(U)>0$.
The last statement of Proposition~\ref{conn-standard} tells us that the assumptions of Theorem~\ref{the-theorem} are satisfied for $U$. Indeed, for $p\geq 0$,
\begin{equation*}\label{periodical-stab}
H_{p}(\Hur_{G, n\cdot\XI}^{c}; \Z ) \cong H_{p}(\Hur_{G, (n+\deg U)\cdot\XI}^{c}; \Z )
\end{equation*}
as long as $n> (8D_{R}(U) + \deg U)p + 7D_{R}(U) + \deg U$.

By definition of $D_{R}(U)$, the number $b = b_{0}(\Hur_{G, n\cdot\XI}^{c})$ of connected components is stable for $n> D_{R}(U)$. We note that $V = V(\underline a) \in R$ is also a $\Z$-stabilizing element for $R$ by Proposition~\ref{conn-standard} and the fact $\partial V = 1$.

Fix $p\geq 0$ and $n> (8D_{R}(U) + \deg U)p + 7D_{R}(U) + \deg U$. Now, $n$ is \emph{always} in the stable range for $\{\Conf_{n\cdot\XI} \mid n\geq 0\}$, given by $n \geq \frac{2p}{\min\XI}$. Indeed, we have $D_{R}(U)>0$ because $G$ is non-abelian.

We choose $k \geq 0$ such that $n+k \deg U$ is in the stable range for the stabilizing element~$V$. We obtain
\begingroup
\allowdisplaybreaks
\begin{align*}
H_{p}(\Hur_{n\cdot\XI}^{c};\Q)  \overset{\phantom{((}\ref{the-theorem}\phantom{))}}{\cong}  &H_{p}(\Hur_{(n+k\deg U)\cdot\XI}^{c};\Q) \\
\overset{\phantom{(}\ref{evwii}\phantom{)}}{\cong} &H_{p} (\Conf_{(n+k\deg U)\cdot \XI};\Q) \otimes_{\Q} \Q^{b} \\
\overset{\text{(\ref{tran})}}{\cong}  &H_{p} (\Conf_{n\cdot \XI};\Q) \otimes_{\Q} \Q^{b} \\
\overset{\text{(\ref{tran})}}{\cong}  &H_{p} (\Conf_{(n+k\deg U+1)\cdot \XI};\Q) \otimes_{\Q} \Q^{b} \\
\overset{\phantom{(}\ref{evwii}\phantom{)}}{\cong} &H_{p}(\Hur_{(n+k\deg U+1)\cdot\XI}^{c};\Q) \\
\overset{\phantom{((}\ref{the-theorem}\phantom{))}}{\cong}  &H_{p}(\Hur_{(n+1)\cdot\XI}^{c};\Q),
\end{align*}
\endgroup
which yields the remaining assertions.
\end{proof}

\subsection*{Outlook} The present paper may be understood as a sequel to the topological part of \cite{0912.0325}. There are still open questions regarding the homological stabilization for Hurwitz spaces:

\begin{itemize}
\item[-] 
For $t=1$, the condition from Theorem~\ref{the-theorem} is equivalent to the non-splitting property. Is there an analogous translation for the general case?
\item[-] 
Until now, we only considered the \emph{diagonal} stabilization direction, i.e., we considered the sequence of shapes $\{n\cdot\XI\}$. Is there also a general theorem for the stabilization in the direction of a fixed conjugacy class, i.e., sequences $\{\XI + ne_{i}\}$, where $e_{i}$ is a unit vector? This is motivated by the corresponding result for colored configuration spaces in \cite{1312.6327}.
\item[-] 
Does the homological stabilization carry over to base spaces of higher genus?
\item[-] 
With Harer's theorem for $\mathcal{M}_{g}$ in mind, it is a natural question whether homological stabilization happens not only in the direction of the number of branch points, but also in the direction of the genus of the covered surface (\emph{genus stabilization}).
For the zeroth Betti number, this has been tackled in the articles \cite{MR3428412} and \cite{1301.4409} in the slightly different setting of the substrata $\mathcal{M}_{g}(G)$ of $\mathcal{M}_{g}$ which contain the algebraic curves admitting a faithful action by a fixed finite group $G$.
\end{itemize}

\bibliography{stability_literature}
\bibliographystyle{alpha} 

\end{document}